\documentclass[reqno,11pt]{amsart}
\usepackage[margin=1in]{geometry}
\usepackage{amsmath,amssymb,amsthm,amsfonts,mathrsfs}
\usepackage{color}
\usepackage{cite}
\usepackage{enumitem}
\usepackage[colorlinks=true,linkcolor=blue,citecolor=blue,urlcolor=blue]{hyperref}
\usepackage{tikz}
\usetikzlibrary{arrows,calc}
\newtheorem {theorem}{Theorem}[section]
\newtheorem {lemma}[theorem]{{\bf Lemma}}

\newtheorem {proposition}[theorem]{{\bf Proposition}}
\theoremstyle{remark}
\newtheorem {remark}{{\bf Remark}}[section]

\theoremstyle{plain} \numberwithin {equation}{section}

\newcommand{\R}{{\mathbb R}}
\def\div{ \hbox{\rm div}\,  }
\newcommand\Z{{\mathbb{Z}}}

\def\nn{\nonumber}

\def\u{ \mathbf{u} }
\def\v{ \mathbf{v} }
\def\T{ \mathbb{T} }

\begin{document}
\title[Stability of MHD system]{Linear Growth and Nonlinear Stability of Two-Dimensional\\[1ex] MHD Couette Flow with Vertical Dissipation}

\author[T. Liang]{Tao Liang}
\address[Tao Liang]{\newline   School of Mathematics,
	South China University of Technology,
	Guangzhou, 510640, China}
\email{taolmath@163.com}

\author[J. H. Wu]{Jiahong Wu}
\address[Jiahong Wu]{\newline   Department of Mathematics, University of Notre
	Dame, Notre Dame, IN 46556, USA }
\email{jwu29@nd.edu}

\author[X. P. Zhai]{Xiaoping Zhai}
\address[Xiaoping Zhai]{\newline   School of Mathematics and Statistics, Guangdong University of Technology,
	Guangzhou, 510520, China}
\email{pingxiaozhai@163.com}

\footnote{\today}

\maketitle

\begin{abstract}
We study the two-dimensional incompressible magnetohydrodynamic system near the Couette equilibrium $\bigl((y,0)^{\rm T},(\beta,0)^{\rm T}\bigr)$
on \(\mathbb T\times\mathbb R\), in the anisotropic regime where both viscosity and magnetic diffusivity act only in the vertical direction. For the inviscid linearized problem with \(|\beta|>1/2\), we prove sharp linear-in-time growth of the vorticity and current density at the level of the time rate. In contrast, the horizontal components of the velocity and magnetic perturbations remain uniformly bounded, while the vertical components exhibit quantitative inviscid damping at the rate \(\langle t\rangle^{-1}\).
For the nonlinear problem, we introduce shear-adapted Fourier multipliers that simultaneously capture enhanced dissipation, critical-time effects, and echo-type resonant interactions. Under a suitable horizontal background magnetic field and a quantitative compatibility condition between the magnetic field strength, viscosity, and magnetic diffusivity, we establish global nonlinear stability for divergence-free perturbations satisfying $
\left\|
\bigl(
\mathbf v_{\rm in}-(y,0)^{\rm T},
\mathbf H_{\rm in}-(\beta,0)^{\rm T}
\bigr)
\right\|_{H^N}
\leq
\varepsilon_0\min\{\mu,\nu\}^{1/2}, N\geq4.
$
Moreover, the nonzero horizontal Fourier modes decay in a shear-adapted \(H^N\) norm at the enhanced-dissipation rate
$
e^{-c\min\{\mu,\nu\}^{1/3}t},
$
and the vertical velocity and magnetic components gain an additional inviscid-damping factor \(\langle t\rangle^{-1}\).
\\

\noindent\textsc{Keywords.} Stability threshold; MHD systems; Couette flow \\

  \noindent\textsc{AMS subject classifications.} 35Q35, 35B65, 76W05

\end{abstract}

\tableofcontents

\section{Introduction}
\subsection {Background and recent studies}
In this paper, we consider the two-dimensional MHD system on $(x,y) \in \T \times \R$:
\begin{eqnarray}\label{quanhs}
\left\{\begin{aligned}
&\partial_t {\v}+ {\v} \cdot\nabla {\v} - \nu \partial_{yy}^2 {\v}  + \nabla P  = {\mathbf{H} }  \cdot\nabla {\mathbf{H} },\\
& \partial_t {\mathbf{H} } + {\v} \cdot\nabla {\mathbf{H} } - \mu \partial_{yy}^2 {\mathbf{H} }  =  {\mathbf{H} } \cdot \nabla {\v},\\
& \div {\v} = \div {\mathbf{H} } = 0,\\
& {\v}(0, x,y) = {\v}_{\mathrm{in}}(x,y),\quad {\mathbf{H} }(0,x,y) = {\mathbf{H} }_{\mathrm{in}}(x, y),
\end{aligned}\right.
\end{eqnarray}
where ${\v}$, $ {\mathbf{H} }$,  and $P$ represent the velocity field, magnetic field, and pressure, respectively. The positive constants $\nu$ and
$\mu$ are the viscosity and magnetic diffusivity coefficients. In \eqref{quanhs}, both dissipative terms act only in the vertical direction.

The magnetohydrodynamic equations describe the coupled evolution of an
electrically conducting fluid and its magnetic field. The magnetic
field is transported and stretched by the fluid motion, while its
feedback on the velocity is represented by the Lorentz force. We study
the dynamics of \eqref{quanhs} near the stationary Couette state
\begin{align*}
	{\v}_s = ( y, 0)^{\mathrm{T}}, \quad  {\mathbf{H} }_s = ( \beta, 0)^{\mathrm{T}},
\end{align*}
where $\beta\in\R\setminus\{0\}$ is the strength of the imposed
horizontal magnetic field.

Introducing the perturbations
$$ \u = {\v}-{\v}_s, \   \mathbf{b} =  {\mathbf{H} }  - {\mathbf{H} }_s,$$
 and setting
\begin{align*}
	\u_{\mathrm{in}}
	=\v_{\mathrm{in}}-\v_s,
	\qquad
	\mathbf b_{\mathrm{in}}
	=\mathbf H_{\mathrm{in}}-\mathbf H_s,
\end{align*}
 System \eqref{quanhs} can be rewritten in terms of $(\mathbf{u}, \mathbf{b})$ as
\begin{eqnarray}\label{rewrite}
\left\{\begin{aligned}
&\partial_t \u+ y\partial_x \u + (u^2,0)^{\mathrm{T}} - \beta \partial_x \mathbf{b}  + \u \cdot\nabla \u - \nu \partial_{yy}^2 \u  + \nabla \left( {2\partial_x}{(-\Delta)^{-1}} u^2 \right)  \\
 &\qquad\qquad\qquad+ \nabla \big( {\partial_i \partial_{i^{\prime}}}{(-\Delta)}^{-1} \big( u^i u^{i^{\prime}} - b^i b^{i^{\prime}}\big) \big)  =  \mathbf{b} \cdot \nabla \mathbf{b},\\
& \partial_t \mathbf{b} + y\partial_x \mathbf{b} - (b^2,0)^{\mathrm{T}} - \beta \partial_x \u + \u \cdot\nabla \mathbf{b} - \mu \partial_{yy}^2 \mathbf{b}   = \mathbf{b} \cdot \nabla \u,\\
& \div \u = \div \mathbf{b} =  0,\\
& \u(0, x,y) = \u_{\mathrm{in}}(x,y),\qquad \mathbf{b}(0,x,y) = \mathbf{b}_{\mathrm{in}}(x, y),
\end{aligned}\right.
\end{eqnarray}
where $i, i^{\prime} \in \{1,2\}$, $\partial_i, \partial_{i^{\prime}} \in {\partial_x, \partial_y}$, and the Einstein summation convention over repeated indices $i,i' \in \{1,2\} $ is used.

\vskip .1in
When $\mathbf{H} = 0$, replacing $``-\partial_{yy}^2 "$ by $``-\Delta"$, \eqref{quanhs} is reduced to the  Navier-Stokes equations.
The transition from laminar to turbulent flow has been a fundamental challenge in fluid dynamics since Reynolds' pioneering experiments \cite{Re}. Although certain laminar flows remain linearly stable at all Reynolds numbers \cite{DR,Rom}, they can exhibit nonlinear instability when subjected to finite amplitude perturbations at high Reynolds numbers \cite{DHB,TA}. Kelvin first observed \cite{Kel} that the basin of attraction of laminar flow shrinks as $Re \to \infty$, making nonlinear instability possible.
 This phenomenon motivates the transition-threshold problem,
formulated by Trefethen et al.~\cite{T}: determine the minimum
disturbance amplitude that triggers instability and its scaling with
the Reynolds number.
 The transition threshold
problem was formulated by Bedrossian, Germain and Masmoudi \cite{BGM2017} as follows: Given a norm
$\|\cdot\|_{X},$  determine a $\gamma=\gamma(X)$  so that
\begin{align*}
		&\|\u_{\mathrm{in}}\|_{X}\ll \nu^{\gamma} \Longrightarrow\ \mathrm{stability},\\
		&\|\u_{\mathrm{\mathrm{in}}}\|_{X}\gg \nu^{\gamma}\Longrightarrow\ \mathrm{instability}.
	\end{align*}
The exponent $\gamma$ is referred to as the transition threshold. Significant progress has been made on the transition threshold problem for shear flows in Navier-Stokes equations.
\vskip .1in
On the domain $\T\times\R$, the following important results are known:
\begin{itemize}
	\item If $X$ is Gevrey class $2^-$, then \cite{BMV2016} showed $\gamma\leq 0$, while \cite{DM2023} established $\gamma\geq 0$.
	\item If $X$ is the Sobolev space $H^{\log}_xL^2_y$, then \cite{BVW2018,zhaoweiren2020cpde} obtained $\gamma\leq \tfrac12$, and \cite{LiMasmoudiZhao2022critical} proved $\gamma\geq \tfrac12$.
	\item If $X$ is the Sobolev space $H^\sigma$ with $\sigma\geq 2$, then \cite{zhaoweiren2019,wei2023} showed $\gamma\leq \tfrac13$.
	\item If $X$ is Gevrey class $\tfrac1s$ with $s\in[0,\tfrac12]$, then \cite{LMZ2022G} proved $\gamma\leq \tfrac{1-2s}{3(1-s)}$.
\end{itemize}

In addition, related results are also available for 2D Couette flow in a finite channel (see, e.g., \cite{BHIW2023,BHIW,CLWZ2020}), and for 3D Couette flow (see, e.g., \cite{BGM2017,BGM2020,BGM2015,Chenwei2020,wei2020}).

The presence of a magnetic field substantially enriches and complicates
the stability theory of Couette flow. In contrast to the Navier--Stokes
equations, the velocity and magnetic perturbations are coupled through
Alfv\'enic interactions. Depending on the strength and orientation of the
background magnetic field, such coupling may weaken the classical mixing
and inviscid damping mechanisms associated with Couette flow
\cite{K24,KZ23}, while a sufficiently strong magnetic field may also
produce a stabilizing effect \cite{GBM09,L20,ZZZ21}. Quantifying this
competition has been the subject of several recent works on stability
thresholds for MHD Couette flows.

We summarize the related results in three dimensions, two dimensions, and partially dissipative settings.

\begin{itemize}
  \item \textbf{3D domain $\T\times\R\times\T$.}
  In the fully dissipative case, namely after replacing
  $``-\partial_{yy}^2"$ by $``-\Delta"$, Liss \cite{L20} studied the
  Sobolev stability threshold for the 3D MHD equations near Couette flow
  on $\T\times\R\times\T$ under the assumptions $\nu=\mu>0$ and a
  Diophantine condition on the direction of the imposed magnetic field.
  More precisely, for $\sigma\in\mathbb R\backslash\mathbb Q$ satisfying
  a generic Diophantine condition, the threshold $\gamma=1$ was obtained.
  Using the same approach, Liss further asserted that the threshold
  becomes $\gamma=\frac{4}{3}$ for arbitrary $\sigma\in\mathbb R$ under
  the same dissipative assumptions. These results were later extended by
  Rao, Zhang, and Zi \cite{raoyulin} to the case of unequal viscosity and
  magnetic diffusivity, namely $\mu\neq\nu$. For rationally aligned
  magnetic fields, Wang, Xu, and Zhang \cite{wangxuzhang} obtained the
  corresponding Sobolev stability threshold $\gamma=1$. Recently, Dolce
  \cite{dolce2026} identified a nonlinear transient-growth regime and
  characterized a sharp threshold interval $\left[\frac56,1\right]$.

  \vskip .1in

  \item \textbf{2D domain $\T\times\R$.}
  In two dimensions, the fully dissipative MHD system near Couette flow
  has also been extensively studied. Replacing $``-\partial_{yy}^2"$ by
  $``-\Delta"$, Chen and Zi \cite{chenting} established the Sobolev
  stability threshold $\gamma=\frac56+$ for shear flows near Couette flow
  under the condition $\nu=\mu$. Dolce \cite{D24} obtained the threshold
  $(\gamma_1,\gamma_2)=(\frac23,0)$ in the fluid-dominated regime
  $0<\mu^3\lesssim\nu\leq\mu$. In the case $\mu\leq\nu$, Knobel
  \cite{K24} proved the threshold
  $(\gamma_1,\gamma_2)=(\frac{1}{12},\frac12)$ when
  $\nu^3\lesssim\mu\leq\nu$, and also showed instability together with
  norm inflation of size $\nu\mu^{-\frac13}$ when $\mu\lesssim\nu^3$.
  These results were subsequently generalized to arbitrary
  $\nu,\mu\in(0,1]$ by Wang and Zhang \cite{Wang2024arxiv} for the
  vorticity formulation and by Jin, Ren, and Wei \cite{JRW24} for the
  velocity formulation.

  \vskip .1in

  \item \textbf{Degenerate and partially dissipative MHD models.}
  Stability and instability mechanisms become even more delicate when
  one of the dissipative effects is absent or acts only in a preferred
  direction. In the Gevrey-2 class with $\nu=0$ and $\mu>0$, Zhao and Zi
  \cite{zhaozi} studied the Euler--MHD system with magnetic diffusion and
  obtained the nonlinear threshold $(\gamma_1,\gamma_2)=(0,1)$ in a
  strong uniform magnetic field. More recently, Dolce, Knobel, and
  Zillinger \cite{Dolcek} proved algebraic instability and large norm
  inflation of the magnetic current in Gevrey classes for the viscous but
  non-resistive case $\nu>0,\mu=0$. Knobel and Zillinger \cite{KZ23}
  investigated the Sobolev stability threshold for the two-dimensional
  MHD equations with horizontal magnetic dissipation. These works show
  that the stability threshold is highly sensitive not only to the
  relative sizes of $\nu$ and $\mu$, but also to the direction and
  degeneracy of the dissipative mechanism.
\end{itemize}

\subsection{Main results}
In this paper, we consider the anisotropic regime in which both viscosity and magnetic
diffusion act only in the vertical direction. This dissipation is
compatible with Couette mixing: for the background shear
$(y,0)^{\mathrm T}$, the transport operator
$\partial_t+y\partial_x$ shifts the vertical frequency of a nonzero
horizontal mode according to
$
	\xi\longmapsto\xi-kt.
$
Thus, the shear transfers nonzero horizontal modes to large vertical
frequencies, where the dissipation $\partial_{yy}^2$ becomes effective.
The model therefore retains an enhanced-dissipation mechanism despite
the absence of horizontal diffusion.

Such anisotropic diffusion is also consistent with models of magnetized
fluids having direction-dependent transport coefficients. Here the
vertical direction is selected by the shear gradient, while the imposed
magnetic field is horizontal. This formulation  isolates the interplay between Couette mixing, vertical diffusion, and Alfv\'enic coupling.

The inviscid linearization reveals the mechanism underlying the
nonlinear analysis. Its vorticity and current
density exhibit transient growth, whereas the corresponding velocity
and magnetic fields undergo inviscid damping. This linear structure
motivates the time-dependent Fourier multipliers used in the nonlinear
analysis.

To  formulate this linear dynamics, we  introduce the scalar
vorticity and the scalar current density
\begin{align*}
	\omega
	=
	\nabla\times\u
	=
	\partial_xu^2-\partial_yu^1,
	\qquad
	j
	=
	\nabla\times\mathbf b
	=
	\partial_xb^2-\partial_yb^1.
\end{align*}
The linearization of the perturbation system around the Couette state,
after setting $\nu=\mu=0$, is given in the vorticity--current
formulation by
\begin{equation}\label{rewrite1}
\left\{
\begin{aligned}
	&\partial_t\omega+y\partial_x\omega
	-\beta\partial_xj=0,
	\\
	&\partial_tj+y\partial_xj
	-\beta\partial_x\omega
	+2\partial_x\partial_y\phi=0,
	\\
	&\u=\nabla^\perp\psi
	=(-\partial_y\psi,\partial_x\psi),
	\qquad
	\Delta\psi=\omega,
	\\
	&\mathbf b=\nabla^\perp\phi
	=(-\partial_y\phi,\partial_x\phi),
	\qquad
	\Delta\phi=j,
	\\
	&\omega(0,x,y)=\omega_{\mathrm{in}}(x,y),
	\qquad
	j(0,x,y)=j_{\mathrm{in}}(x,y).
\end{aligned}
\right.
\end{equation}

For a function $f=f(x,y)$ on $\T\times\R$, we denote its
$x$-average and its nonzero horizontal component by
\begin{align*}
	f_0(y)
	:=
	\frac{1}{2\pi}\int_{\T}f(x,y)\,dx,
	\qquad
	f_{\neq}(x,y)
	:=
	f(x,y)-f_0(y),
\end{align*}
respectively. Thus, $f_{\neq}$ consists precisely of the nonzero
horizontal Fourier modes. Since shear-induced mixing acts only on these
modes, all the growth and damping statements below concern
$(\omega_{\neq},j_{\neq})$ and the corresponding components of
$(\u_{\neq},\mathbf b_{\neq})$.

We use the following Fourier convention. For
$(k,\xi)\in\Z\times\R$, define
\begin{align*}
	\widehat f_k(t,\xi)
	=\mathcal{F}f(t,k,\xi)=
	\widehat f(t,k,\xi)
	:=
	\frac{1}{2\pi}
	\int_{\T\times\R}
	f(t,x,y)e^{-i(kx+\xi y)}
	\,dx\,dy,
\end{align*}
with inverse transform
\begin{align*}
	f(t,x,y)
	=
	\frac{1}{2\pi}
	\sum_{k\in\Z}
	\int_{\R}
	\widehat f_k(t,\xi)e^{i(kx+\xi y)}
	\,d\xi.
\end{align*}
For $s\in\mathbb{R}$, we define the Sobolev space
$H^s(\mathbb{T}\times\mathbb{R})$ through the norm
\begin{align*}
\|f\|_{H^s}^2
:=
\sum_{k\in\mathbb{Z}}
\int_{\mathbb{R}}
\langle k,\xi\rangle^{2s}
|\widehat f_k(\xi)|^2
\,d\xi,
\qquad
\langle k,\xi\rangle
:=
\bigl(1+k^2+\xi^2\bigr)^{1/2}.
\end{align*}

The first result gives linear growth of the vorticity and current together with inviscid damping of the vertical velocity
and magnetic components.

\begin{theorem}[Linear growth and inviscid damping]\label{thm1.1}
Let $|\beta|>1/2$, and let
$(\omega,j)$ solve \eqref{rewrite1} with
\begin{align*}
	(\omega_{\mathrm{in}},j_{\mathrm{in}})
	\in H^1(\T\times\R).
\end{align*}
There exists $C=C(\beta)>0$ such that, for every $t\geq0$,
\begin{align}\label{linear_growth_result}
	C^{-1}\langle t\rangle
	\left\|
	\partial_x^2(-\Delta)^{-1}
	(\omega_{\mathrm{in},\neq},j_{\mathrm{in},\neq})
	\right\|_{L^2}
	&\leq
	\|(\omega_{\neq},j_{\neq})(t)\|_{L^2}
	\leq
	C\langle t\rangle
	\|(\omega_{\mathrm{in},\neq},
	j_{\mathrm{in},\neq})\|_{L^2}.
\end{align}
Moreover,
\begin{equation}\label{linear_horizontal_bound}
	\|(u_{\neq}^1,b_{\neq}^1)(t)\|_{L^2}
	\leq
	C
	\|(\omega_{\mathrm{in},\neq},
	j_{\mathrm{in},\neq})\|_{L^2},
\end{equation}
and
\begin{equation}\label{linear_vertical_damping}
	\|(u_{\neq}^2,b_{\neq}^2)(t)\|_{L^2}
	\leq
	C\langle t\rangle^{-1}
	\|(\omega_{\mathrm{in},\neq},
	j_{\mathrm{in},\neq})\|_{H^1}.
\end{equation}
In \eqref{linear_growth_result}, $(-\Delta)^{-1}$ is understood on
the nonzero horizontal modes.
\end{theorem}

\vskip .1in
\begin{remark}
For the two-dimensional Euler equations linearized around Couette flow,
the vorticity is transported by the background shear and its $L^2$
norm is conserved. In \eqref{rewrite1}, the coupling terms
$
	-\beta\partial_xj$,
	$
	-\beta\partial_x\omega,
$
and the nonlocal term
$
	2\partial_x\partial_y\phi$,
	$
	\Delta\phi=j,
$
alter this structure. In particular, \eqref{linear_growth_result} shows
linear-in-time growth of the vorticity--current pair. The rate is sharp
at the level of time growth, since the upper and lower bounds are both
of order $\langle t\rangle$, with the lower bound expressed in the
weaker norm
\begin{align*}
	\left\|
	\partial_x^2(-\Delta)^{-1}
	(\omega_{\mathrm{in},\neq},j_{\mathrm{in},\neq})
	\right\|_{L^2}.
\end{align*}
This growth is absent from the linearized Euler dynamics and is a
principal difficulty in the nonlinear MHD problem.
\end{remark}

\begin{remark}
The growth of $(\omega_{\neq},j_{\neq})$ is consistent with the
inviscid-damping estimate \eqref{linear_vertical_damping}. Vorticity and
current contain one derivative of the velocity and magnetic fields and
are amplified by the shear, whereas the Biot--Savart law gains one
derivative. Phase mixing therefore yields decay of
$(u_{\neq}^2,b_{\neq}^2)$, even though
$(\omega_{\neq},j_{\neq})$ grows. The horizontal components
$(u_{\neq}^1,b_{\neq}^1)$ remain bounded by
\eqref{linear_horizontal_bound}.
\end{remark}

We now turn to the nonlinear system with vertical viscosity and
vertical magnetic diffusion:
\begin{equation}\label{main}
\left\{
\begin{aligned}
	&\partial_t\u+y\partial_x\u+(u^2,0)^{\mathrm T}
	-\beta\partial_x\mathbf b
	+\u\cdot\nabla\u-\nu\partial_{yy}^2\u
	\\
	&\qquad
	+\nabla\left(
	2\partial_x(-\Delta)^{-1}u^2
	\right)
	+\nabla\big(
	\partial_i\partial_{i'}(-\Delta)^{-1}
	\bigl(u^iu^{i'}-b^ib^{i'}\bigr)
	\big)
	=\mathbf b\cdot\nabla\mathbf b,
	\\
	&\partial_t\mathbf b+y\partial_x\mathbf b-(b^2,0)^{\mathrm T}
	-\beta\partial_x\u
	+\u\cdot\nabla\mathbf b-\mu\partial_{yy}^2\mathbf b
	=\mathbf b\cdot\nabla\u,
	\\
	&\div\u=\div\mathbf b=0,
	\\
	&\u(0,x,y)=\u_{\mathrm{in}}(x,y),
	\qquad
	\mathbf b(0,x,y)=\mathbf b_{\mathrm{in}}(x,y).
\end{aligned}
\right.
\end{equation}
Here and below, repeated indices $i,{i'}\in\{1,2\}$ are summed over.

\vskip .1in
The main challenge in establishing the global stability of System \eqref{main} lies in controlling the nonlinear interactions, which are potentially amplified by the linear transient growth identified in Theorem \ref{thm1.1}. Although the dissipative mechanisms are highly degenerate--acting exclusively in the vertical direction--the background shear flow efficiently tilts the nonzero horizontal modes, transporting them to high vertical frequencies where the operator $\partial_{yy}^2$ becomes strongly elliptic. This interplay between Couette mixing and anisotropic diffusion generates an \emph{enhanced dissipation} effect. Our second main result demonstrates that if the background magnetic field is sufficiently strong, this mechanism rigorously suppresses the destabilizing nonlinear resonances (echoes), yielding global stability.

\begin{theorem}[Nonlinear stability]\label{thm1.2}
Let $\mu,\nu\in(0,1]$, $N\geq4$,  and $\vartheta\in(0,1/6)$. Assume that
\begin{equation}\label{betatiaojian}
   |\beta|>7\pi+1,\qquad \frac{6(\mu+\nu)}{|\beta|\sqrt{\mu\nu}}
    \leq1-6\vartheta.
\end{equation}
There exists  $\varepsilon_0>0$ depending only on  $N,\beta, \vartheta,$ such that if
\begin{equation}\label{chuzhixiao}
    \bigl\|
        (\mathbf u_{\mathrm{in}},\mathbf b_{\mathrm{in}})
    \bigr\|_{H^N}
    \leq\varepsilon_0\kappa^{1/2},\qquad\kappa:=\min\{\mu,\nu\},
\end{equation}
then the perturbation \eqref{main} admits a unique global solution $(\mathbf u,\mathbf b)$ satisfying
\begin{align*}
    (\mathbf u,\mathbf b)
    \in
    C\bigl(
        [0,\infty);
        H^N(\mathbb T\times\mathbb R)
    \bigr),
\qquad
    \sqrt{\nu}\,\partial_y\mathbf u,\,
    \sqrt{\mu}\,\partial_y\mathbf b
    \in
    L_{\mathrm{loc}}^2
    \bigl(
        [0,\infty);
        H^N(\mathbb T\times\mathbb R)
    \bigr).
\end{align*}
Moreover, there exists a constant
$
    C=C(N,\beta,\vartheta)>0,
$
independent of $t$, $\mu$, and $\nu$, such that, for every $t\geq0$,
\begin{align}
    \bigl\|\Lambda_t^N(\mathbf u,\mathbf b)(t)\bigr\|_{L^2}
    &\leq
    C\varepsilon_0\kappa^{1/2},
    \label{GWP}\\
    \bigl\|(u_{\neq}^1,b_{\neq}^1)(t)\bigr\|_{L^2}
    +
    \langle t\rangle
    \bigl\|(u_{\neq}^2,b_{\neq}^2)(t)\bigr\|_{L^2}
    &\leq
    C\varepsilon_0\kappa^{1/2}
    e^{-\frac{1}{64}\kappa^{1/3}t},
    \label{invisciddamping}\\
    \bigl\|
        \Lambda_t^N
        (\mathbf u_{\neq},\mathbf b_{\neq})(t)
    \bigr\|_{L^2}
    &\leq
    C\varepsilon_0\kappa^{1/2}
    e^{-\frac{1}{64}\kappa^{1/3}t}.
    \label{enhancedis}
\end{align}
\end{theorem}

\vskip .1in
\begin{remark}
The proof is carried out at the level of the velocity and magnetic
fields. At the vorticity--current level, Theorem~\ref{thm1.1} exhibits
linear-in-time growth, and the nonlinear terms carry additional
derivatives. These features make the resonant interactions less
transparent and lead to derivative losses in a direct energy argument.
The velocity formulation retains the smoothing provided by the
Biot--Savart law and is better suited to the enhanced-dissipation
estimates. The multiplier $\mathcal{M}_k^{(3)}(t,\xi)$ introduced below is designed to
control the remaining echo-type interactions.
\end{remark}

\begin{remark}\label{rmk:beta_threshold}
The lower bounds imposed on $|\beta|$ in Theorem~\ref{thm1.2} are sufficient conditions dictated by our symmetrized energy framework, rather than sharp physical thresholds for nonlinear stability.

First, the condition $\frac{6(\mu+\nu)}{|\beta|\sqrt{\mu\nu}} \leq 1-6\vartheta$ is required to absorb the linear cross-dissipation error induced by unequal diffusivities ($\mu \neq \nu$). Specifically, it ensures that the effective coefficient of the physical dissipation satisfies $\frac16-\frac{\mu+\nu}{|\beta|\sqrt{\mu\nu}} \geq \vartheta>0$.

Second, the absolute bound $|\beta|>7\pi+1$ guarantees the strict positivity of the Cauchy--Kovalevskaya (CK) coefficients. In particular, it ensures $\frac{1}{6\pi} - \frac{7\pi+1}{6\pi|\beta|}>0$ and $\frac{1}{6D_\delta} - \frac{12D_\delta+1}{12D_\delta|\beta|}>0$, corresponding to $\operatorname{CK}_1$ and $\operatorname{CK}_2$, respectively. These numerical constraints come from the particular choice of multipliers and estimates used in the proof. Relaxing $|\beta|$ closer to the linear stability threshold $|\beta|>1/2$ via a more refined symmetrizer remains an interesting open question.
\end{remark}

\begin{remark}\label{rmk:decay-rate}
The numerical constant $1/64$ used in the exponential weight satisfies
${1}/{64}<{1}/{(12\pi)}
    $ and
    ${1}/{64}<1/6.$
These two inequalities are used in the multiplier coercivity estimate
\eqref{pre8} and in the linear multiplier-error estimate
\eqref{quanju38}, respectively. This choice is independent of the
positivity condition for $c_1$, which follows solely from
$|\beta|>7\pi+1$.
\end{remark}

\vskip .1in
\subsection{Strategy of the proof of Theorem~\ref{thm1.2}}
\label{subsec:proof-strategy}

We briefly explain the main ideas in the proof of
Theorem~\ref{thm1.2}. The main difficulty is the interaction between linear transient growth and nonlinear resonances generated by the
Couette shear.  Although the viscosity and
the magnetic diffusivity act only in the vertical direction, the shear
continuously transfers nonzero horizontal modes toward large vertical
frequencies. This mixing mechanism converts the vertical dissipation into
an effective enhanced dissipation on the time scale
\begin{align*}
t_{\rm ED}\sim \kappa^{-1/3},
\qquad
\kappa:=\min\{\mu,\nu\}.
\end{align*}
The proof is based on a time-dependent weighted energy adapted
simultaneously to this enhanced dissipation, the critical-time geometry,
and the Alfv\'enic coupling.

\medskip
\noindent
\textbf{Time-dependent multipliers.}
We introduce three Fourier multipliers
\begin{align*}
\mathcal{M}^{(1)}_k(\xi),\qquad
\mathcal{M}^{(2)}_k(\xi),\qquad
\mathcal{M}^{(3)}_k(t,\xi),
\end{align*}
whose precise definitions are given in Section~\ref{sec:preliminaries}. They are designed to
capture three complementary mechanisms.

The first multiplier is associated with enhanced dissipation and satisfies,
schematically,
\begin{align*}
\kappa \xi^2+k\partial_\xi \mathcal{M}^{(1)}_k(\xi)
\gtrsim
\kappa^{1/3}|k|^{2/3},
\qquad k\neq 0.
\end{align*}
Thus the combination of Couette mixing and vertical diffusion provides an
effective horizontal coercivity of order
$\kappa^{1/3}|\partial_x|^{2/3}$.

The second multiplier is localized near the critical-time region and
generates a nonnegative weight
\begin{align*}
k\partial_\xi \mathcal{M}^{(2)}_k(\xi)
\gtrsim \delta'_k(\xi)\geq 0,
\end{align*}
which is used to control interactions associated with the critical
variable $\xi/k$.

The third multiplier is time-dependent and is constructed so that
\begin{align*}
\mathcal{G}_k(t,\xi)
:=
\bigl(-\partial_t+k\partial_\xi\bigr)
\mathcal{M}^{(3)}_k(t,\xi)
\geq 0.
\end{align*}
The weight $\mathcal{G}_k$ is adapted to the resonant convolution geometry and controls the echo-type nonlinear interactions.

To incorporate the expected decay of the nonzero horizontal modes, we set
\begin{align*}
\mathcal{A}_k(t,\xi)
:=
\exp\Big(
\frac{1}{32}\kappa^{1/3}t\,\mathbf 1_{\{k\neq0\}}
\Big)
\end{align*}
and define
\begin{align*}
\mathcal{W}_k(t,\xi)
:=
\mathcal{A}_k(t,\xi)
\Big(
\mathcal{M}^{(1)}_k(\xi)
+\mathcal{M}^{(2)}_k(\xi)
+\mathcal{M}^{(3)}_k(t,\xi)
+\frac23
\Big).
\end{align*}
The individual multipliers are uniformly bounded, and hence
\begin{align*}
\mathcal{W}_k(t,\xi)\approx \mathcal{A}_k(t,\xi).
\end{align*}
Moreover, for $\sigma\in\{\nu,\mu\}$, the physical dissipation together
with the material derivative of $\mathcal{W}$ gives the  coercive
estimate
\begin{align}
2\sigma\xi^2 \mathcal{W}_k(t,\xi)
+\bigl(-\partial_t+k\partial_\xi\bigr)\mathcal{W}_k(t,\xi)
\gtrsim
\mathcal{A}_k(t,\xi)
\Big(
\sigma\xi^2
+\kappa^{1/3}|k|^{2/3}
+\delta'_k(\xi)
+\mathcal{G}_k(t,\xi)
\Big).
\end{align}
These four terms correspond respectively to the physical vertical
dissipation, enhanced dissipation, critical-time control, and echo
control.

\medskip
\noindent
\textbf{Shear-adapted energy.}
To follow the Couette characteristics, we use  the time-dependent elliptic
operator introduced in \cite{Deng2021JFA}
\begin{align*}
\Lambda_t^2
:=
1-\partial_x^2-(\partial_y+t\partial_x)^2,
\end{align*}
whose Fourier symbol is
\begin{align*}
\Lambda_t^2(k,\xi)
=
1+k^2+(\xi+kt)^2.
\end{align*}
The key identity
\begin{align*}
(\partial_t-k\partial_\xi)\Lambda_t^N(k,\xi)=0
\end{align*}
shows that $\Lambda_t^N$ commutes with the linear Couette transport at the
Fourier level. We therefore measure the solution by the weighted norm
\begin{align*}
X(t)
:=
\left\|
\sqrt{\mathcal{W}}\,\Lambda_t^N(\mathbf u,\mathbf b)(t)
\right\|_{L^2}.
\end{align*}

The standard energy $X^2(t)$ alone is not sufficient because of the
Alfv\'enic coupling between the velocity and the magnetic field. Guided by
the linearized system, we add suitable cross terms involving the horizontal
and vertical components and define a modified energy $  E(t)$.
For $|\beta|$ sufficiently large, these cross terms are perturbative and
\begin{align*}
  E(t)\approx X^2(t).
\end{align*}
More importantly, the modified energy cancels the leading linear coupling
responsible for the transient growth appearing in Theorem~\ref{thm1.1}.

After differentiating $  E(t)$ and using the multiplier coercivity,
we obtain an energy inequality of the schematic form
\begin{align*}
\begin{aligned}
\frac{d}{dt}  E(t)
&+
c_D\, \operatorname{Dis}_1(t)
+c_1\,\mathrm{CK}_1(t)
+c_2\,\mathrm{CK}_2(t)
+c_3\,\mathrm{CK}_3(t)
\leq
C\sum_{i=1}^{8}\mathcal{Z}_i(t),
\end{aligned}
\end{align*}
where
\begin{align*}
\mathrm{CK}_1
\sim
\kappa^{1/3}
\left\|
|\partial_x|^{1/3}
\sqrt{\mathcal{W}}\Lambda_t^N(\mathbf u,\mathbf b)
\right\|_{L^2}^2,
\end{align*}
while $\mathrm{CK}_2$ and $\mathrm{CK}_3$ denote the spacetime controls
generated by $\delta'_k(\xi)$ and $\mathcal G_k(t,\xi)$, respectively.

The condition
\begin{align*}
\frac{6(\mu+\nu)}
{|\beta|\sqrt{\mu\nu}}
\leq 1-6\vartheta
\end{align*}
ensures that the error caused by the mismatch between the viscosity and
the magnetic diffusivity can be absorbed by the physical dissipation.
The additional lower bound on $|\beta|$ guarantees that the remaining
linear errors are absorbed by the CK terms.

\medskip
\noindent
\textbf{Nonlinear estimates and bootstrap closure.}
The nonlinear terms are estimated using incompressibility and finite-difference bounds for the multiplier
$\mathcal W_k(t,\xi)\Lambda_t^{2N}(k,\xi)$. For transport-type nonlinearities, the leading
derivative is removed by incompressibility and the remaining commutator
contains differences of the form
\begin{align*}
\mathcal W_k(t,\xi)\Lambda_t^{2N}(k,\xi)
-
\mathcal W_\ell(t,\eta)\Lambda_t^{2N}(\ell,\eta).
\end{align*}
The estimates established in Section~\ref{sec:preliminaries} allow these differences to be
controlled by the physical dissipation and the three CK quantities.

Away from the resonant region, the Sobolev weight and enhanced dissipation
are sufficient. Near critical times, we use the weight $\delta'_k(\xi)$,
while $\mathcal G_k(t,\xi)$ controls the resonant echo interactions. The zero horizontal modes are treated separately by means of the
weighted energy and the physical vertical dissipation. In this way, all
nonlinear contributions satisfy a schematic estimate
\begin{align*}
\sum_{i=1}^{8}\mathcal{Z}_i(t)
\lesssim
\kappa^{-1/2} X(t)
\bigl(
\operatorname{Dis}_1(t)
+\mathrm{CK}_1(t)
+\mathrm{CK}_2(t)
+\mathrm{CK}_3(t)
\bigr),
\end{align*}
up to terms of the same perturbative type.

We finally impose the bootstrap assumption
\begin{align*}
X(t)\leq c\,\kappa^{1/2},
\qquad 0\leq t\leq T,
\end{align*}
with $c>0$ sufficiently small. The nonlinear terms can then be absorbed
by the dissipative terms on the left-hand side of the energy inequality.
Choosing the initial constant $\varepsilon_0$ sufficiently small gives an
improvement of the bootstrap bound, and the standard continuation
argument yields global existence.

Since
\begin{align*}
\mathcal W_k(t,\xi)
\approx
e^{\frac{1}{32}\kappa^{1/3}t},
\qquad k\neq0,
\end{align*}
the uniform weighted estimate gives
\begin{align*}
\left\|
\Lambda_t^N(\mathbf u_{\neq},\mathbf b_{\neq})(t)
\right\|_{L^2}
\lesssim
\varepsilon_0\kappa^{1/2}
e^{-\frac{1}{64}\kappa^{1/3}t}.
\end{align*}
Finally, the divergence-free condition yields the additional
$\langle t\rangle^{-1}$ inviscid-damping factor for the vertical
components in Theorem~\ref{thm1.2}.

\subsection{Organization of the paper}\label{subsec:organization}
Section~\ref{sec:preliminaries} establishes the multiplier, commutator, and convolution estimates used in the nonlinear analysis.
Section~\ref{proofth1} proves Theorem~\ref{thm1.1} through sheared coordinates, symmetric variables, and a modified energy adapted to the Alfv\'enic coupling.
Section~\ref{proofth2} proves Theorem~\ref{thm1.2} by combining the weighted energy estimates with local well-posedness and a bootstrap argument.

\subsection*{Notation}

Throughout the paper, $C>0$ denotes a generic constant that may change from line to line. Unless otherwise stated, $C$ depends only on the fixed parameters $N$, $\beta$, and $\vartheta$, and is  independent of $t, \mu, \nu,$ and $\kappa$. Here $\vartheta \in (0,1/6)$ quantifies the uniform margin in the magnetic-field condition
\begin{align*}
  \frac{6(\mu+\nu)}{|\beta|\sqrt{\mu\nu}} \leq 1-6\vartheta.
\end{align*}

\vskip .1in
We use the notation $(k,\xi)$ for Fourier variables in both Sections~\ref{proofth1}
and~\ref{proofth2}, but with different coordinate conventions. More precisely, in
Section~\ref{proofth1}, $(k,\xi)$ denotes the Fourier variables associated with the
sheared coordinates
\begin{align*}
    X=x-ty,\qquad Y=y,
\end{align*}
whereas in Section~\ref{proofth2}, $(k,\xi)$ denotes the Fourier variables associated
with the original coordinates $(x,y)$.
Thus, in Section~\ref{proofth1}, the symbol of $-\Delta_L$ is
$
    k^2+(\xi-kt)^2,
$
while in Section~\ref{proofth2}, the symbol of $\Lambda_t^2$ is
$
    1+k^2+(\xi+kt)^2.
$
The meaning of $(k,\xi)$ will always be clear from the section under
consideration.

\vskip .1in
We write $a\lesssim b$ if $a\leq Cb$, and $a\approx b$ if both $a\lesssim b$ and $b\lesssim a$. Dependence of the implicit constant on an additional parameter $q$ is indicated by $a\lesssim_q b$.

\vskip .2in
\section{Preliminary lemmas}\label{sec:preliminaries}

In this section, we define the time-dependent Fourier multipliers and prove the bounds used in the nonlinear energy
estimates.

Let $\chi\in C^\infty([0,\infty))$ be a fixed smooth cutoff satisfying
\begin{align*}
0\leq \chi\leq 1,\qquad
\chi(r)=1\quad\text{for }0\leq r\leq 1,\qquad
\chi(r)=0\quad\text{for }r\geq 2.
\end{align*}

For $k\neq0$, we define the first multiplier as
\begin{align*}
\mathcal{M}_k^{(1)}(\xi)
=\frac{1}{3\pi}
\arctan\left(
\kappa^{1/3}|k|^{-1/3}\operatorname{sgn}(k)\,\xi
\right),
\end{align*}
and set $\mathcal{M}_0^{(1)}(\xi)=0$.

Although the horizontal Fourier frequency $k$ is ultimately
restricted to $\mathbb Z$, in estimates involving
$\partial_k\mathcal M_k^{(1)}(\xi)$ we regard the defining formula for
$\mathcal M_k^{(1)}(\xi)$ as a smooth function of the real parameter
$k\in\mathbb R\setminus\{0\}$. More precisely, for
$q\in\mathbb R\setminus\{0\}$, we set
\begin{align*}
\mathcal M_q^{(1)}(\xi)
:=
\frac{1}{3\pi}
\arctan\left(
\kappa^{1/3}|q|^{-1/3}\operatorname{sgn}(q)\,\xi
\right).
\end{align*}
We first differentiate this smooth extension with respect to $q$ and
then restrict the resulting estimate to the integer frequencies
$q=k\in\mathbb Z\setminus\{0\}$. We continue to write
$\partial_k\mathcal M_k^{(1)}(\xi)$ for this restricted derivative in the following argument.

For $k\neq0$, the second multiplier is defined by
\begin{align*}
\mathcal{M}_k^{(2)}(\xi)
=\frac{1}{6D_\delta}
\chi(\kappa^{1/2}|k|)
\delta\left({\xi}/{k}\right),
\qquad
\mathcal{M}_0^{(2)}(\xi)=0,
\end{align*}
where the primitive function $\delta$ is determined by $\delta(0)=0$ and
\begin{align*}
\delta'(x)
=
\frac{1}
{\langle x\rangle[\log(1+\langle x\rangle)]^{2}},
\end{align*}
with the normalization constant $D_{\delta}$ given by
\begin{align*}
D_{\delta}
:=\lim_{x\to+\infty}\delta(x)
=\int_{0}^{\infty}
\frac{ds}
{\sqrt{1+s^{2}}\,
\bigl[\log\bigl(1+\sqrt{1+s^{2}}\bigr)\bigr]^{2}}
\approx 3.123.
\end{align*}

Finally, we define the echo multiplier $\mathcal{M}_k^{(3)}(t,\xi)$ as
\begin{align*}
\mathcal{M}_k^{(3)}(t,\xi)
:=
-\frac{1}{6D_N}\chi(\kappa^{1/2}|k|)
\sum_{\ell\neq k}
&\int_0^t\int_{\mathbb R}
\frac{1}
{
\bigl(\langle\ell\rangle^{2N-2}+|\eta+\ell\tau|^{2N-2}\bigr)
}
\notag\\
&\quad\times
\frac{d\eta\,d\tau}
{\left(
1+
\left|
\frac{\xi-\eta+k(t-\tau)}{k-\ell}
\right|
\right)
\left[
\log\left(
1+
\left\langle
\frac{\xi-\eta+k(t-\tau)}{k-\ell}
\right\rangle
\right)
\right]^{2}
},
\end{align*}
where the  normalization  constant $D_N$ is
\begin{align*}
D_N
:=
\frac{2\pi}{N-1}\,
\csc\left(\frac{\pi}{2N-2}\right)
\sum_{\ell\in\mathbb Z}
\langle \ell\rangle^{-(2N-3)}\int_{0}^{\infty}
\frac{{d}s}
{(1+s)\bigl[\log\bigl(1+\sqrt{1+s^{2}}\bigr)\bigr]^{2}}.
\end{align*}

Unlike $\mathcal{M}_0^{(1)}(\xi)$ and $\mathcal{M}_0^{(2)}(\xi)$, we do not set $\mathcal{M}_0^{(3)}(t,\xi)$ to zero. Nonlinear interactions between two nonzero horizontal modes (with frequencies $\ell$ and $k-\ell=-\ell$) can cascade into the zero horizontal mode ($k=0$). Therefore, the multiplier governing these resonant echo interactions must remain active at $k=0$.

We further define the localized weight
\begin{align*}
\delta_k'(\xi)
:=
\mathbf{1}_{k\neq0}\,
\chi(\kappa^{1/2}|k|)
\delta'\left({\xi}/{k}\right),
\end{align*}
and the non-negative dissipation symbol associated with $\mathcal{M}_k^{(3)}(t,\xi)$,
\begin{align*}
\mathcal{G}_k(t,\xi)
:=
\bigl(-\partial_t+k\partial_\xi\bigr)\mathcal{M}_k^{(3)}(t,\xi).
\end{align*}
By the transport identity $\bigl(-\partial_t+k\partial_\xi\bigr) \frac{\xi-\eta+k(t-\tau)}{k-\ell} = 0$, the derivative acts exclusively on the upper limit of the $\tau$-integral, yielding
\begin{align*}
\mathcal{G}_k(t,\xi)
=\frac{1}{6D_N}
\chi(\kappa^{1/2}|k|)
\sum_{\ell\neq k}
\int_{\mathbb R}
\frac{d\eta}
{
\bigl(\langle\ell\rangle^{2N-2}+|\eta+\ell t|^{2N-2}\bigr)
\left(
1+
\left|
\frac{\xi-\eta}{k-\ell}
\right|
\right)
\left[
\log\left(
1+
\left\langle
\frac{\xi-\eta}{k-\ell}
\right\rangle
\right)
\right]^{2}}.
\end{align*}
Thus $\mathcal{G}_k(t,\xi) \geq 0$.

To formulate the uniform energy, for $t\geq0$, let
\begin{align*}
\mathcal{A}_k(t,\xi)
:=
\exp\left(
\frac{1}{32}\kappa^{1/3}t\,\mathbf{1}_{k\neq0}
\right),
\end{align*}
and define the full multiplier
\begin{align*}
\mathcal{W}_k(t,\xi)
:=
\mathcal{A}_k(t,\xi)
\left(
\mathcal{M}_k^{(1)}(\xi)
+
\mathcal{M}_k^{(2)}(\xi)
+
\mathcal{M}_k^{(3)}(t,\xi)
+
\frac23
\right).
\end{align*}
We use $\mathcal A$, $\mathcal W$, and $\mathcal G$ to denote the
Fourier multipliers with symbols
$\mathcal A_k(t,\xi)$, $\mathcal W_k(t,\xi)$, and
$\mathcal G_k(t,\xi)$, respectively.

The following lemma establishes the boundedness and coercivity
properties of the multipliers. In particular, it shows that
$\mathcal W_k(t,\xi)$ is positive and uniformly comparable to
$\mathcal A_k(t,\xi)$.
\begin{lemma}\label{chengziyinli}
Assume $N\geq4$. The multipliers introduced above satisfy the following estimates uniformly in $t\geq0$, $k\in\mathbb Z$, and $\xi\in\mathbb R$:
\begin{align}\label{pre1}
 \bigl|\mathcal{M}_k^{(1)}(\xi)\bigr|\le\frac16,\qquad\bigl|\mathcal{M}_k^{(2)}(\xi)\bigr|\le\frac16,
 \qquad -\frac16 \le \mathcal{M}_k^{(3)}(t,\xi) \le 0.
\end{align}
Consequently,
\begin{align}\label{chengzijie}
\frac16
\leq
\mathcal M_k^{(1)}(\xi)
+
\mathcal M_k^{(2)}(\xi)
+
\mathcal M_k^{(3)}(t,\xi)
+
\frac23
\leq1,
\end{align}
and hence
\begin{align}\label{WA-equivalence}
\frac16\mathcal A_k(t,\xi)
\leq
\mathcal W_k(t,\xi)
\leq
\mathcal A_k(t,\xi).
\end{align}
Moreover, for every sufficiently regular function $f$,
\begin{align}\label{pre2}
 \sum_{k\in\mathbb Z}\int_{\mathbb R}
 \left(\kappa\xi^2+k\partial_\xi \mathcal{M}_k^{(1)}(\xi)\right)
 |\widehat f(k,\xi)|^2\,d\xi
 \geq
 \frac{1}{6\pi}\kappa^{1/3}
 \bigl\||\partial_x|^{1/3}f\bigr\|_{L^2}^2,
\end{align}
and
\begin{align}\label{pre3}
 0
 \leq
 \sum_{k\in\mathbb Z}\int_{\mathbb R}
 k\partial_\xi \mathcal{M}_k^{(2)}(\xi)|\widehat f(k,\xi)|^2\,d\xi
 =\frac{1}{6D_{\delta}}
 \sum_{k\in\mathbb Z}\int_{\mathbb R}
 \delta_k'(\xi)|\widehat f(k,\xi)|^2\,d\xi .
\end{align}
For $k\neq0$, the derivatives satisfy
\begin{align}\label{pre4}
 \bigl|\partial_k \mathcal{M}_k^{(1)}(\xi)\bigr| + \bigl|\partial_\xi \mathcal{M}_k^{(2)}(\xi)\bigr|
  \lesssim |k|^{-1},\quad \bigl|\partial_\xi \mathcal{M}_k^{(1)}(\xi)\bigr|
 &\lesssim \kappa^{1/3}|k|^{-1/3},
 \quad
 \bigl|\partial_\xi \mathcal{M}_k^{(3)}(t,\xi)\bigr|
 \lesssim |k|^{-1}.
\end{align}
If $k,\ell\in\mathbb Z\setminus\{0\}$ satisfy $k\ell>0$ and $|k-\ell|\leq\frac12|\ell|$, then
\begin{align}\label{pre5}
 &\bigl|\mathcal{M}_k^{(1)}(\xi)-\mathcal{M}_\ell^{(1)}(\eta)\bigr|
 \lesssim
 \frac{|k-\ell|}{|k|}
 +\kappa^{1/3}|k|^{-1/3}|\xi-\eta|.
\end{align}
Finally, for every $\sigma\in\{\nu,\mu\}$, we have the coercivity bound
\begin{align}\label{pre8}
 &2\sigma\xi^2\mathcal{W}_k(t,\xi)
 +\bigl(-\partial_t+k\partial_\xi\bigr)\mathcal{W}_k(t,\xi)\nn\\
 &\quad\geq
 \mathcal{A}_k(t,\xi)
 \Big(
 \frac16\sigma\xi^2
 +\frac{1}{6\pi}\kappa^{1/3}|k|^{2/3}
 +\frac{1}{6D_\delta}\delta_k'(\xi)
 +\mathcal{G}_k(t,\xi)
 \Big).
\end{align}
\end{lemma}

\begin{proof}
For $\mathcal{M}_k^{(1)}(\xi)$, the definition and
$|\arctan z|\leq\pi/2$ give, for $k\neq0$,
\[
\left|\mathcal{M}_k^{(1)}(\xi)\right|
\leq
\frac{1}{3\pi}\cdot\frac{\pi}{2}
=\frac16.
\]
The same bound is trivial for $k=0$, since
$\mathcal{M}_0^{(1)}(\xi)=0$.

Next, since $\delta(0)=0$ and $\delta'$ is even and nonnegative,
$\delta$ is odd and increasing. Moreover,
\[
\lim_{x\to+\infty}\delta(x)=D_\delta,
\qquad
\lim_{x\to-\infty}\delta(x)=-D_\delta.
\]
Consequently,
\[
|\delta(x)|\leq D_\delta,
\qquad x\in\mathbb R.
\]
It follows that, for $k\neq0$,
\[
\left|\mathcal{M}_k^{(2)}(\xi)\right|
\leq
\frac{1}{6D_\delta}
\chi(\kappa^{1/2}|k|)
\left|\delta\left(\frac{\xi}{k}\right)\right|
\leq\frac16.
\]
Again, the estimate is immediate for $k=0$, since
$\mathcal{M}_0^{(2)}(\xi)=0$.

It remains to estimate $\mathcal{M}_k^{(3)}(t,\xi)$. Define
\[
h(r):=
\frac{1}
{(1+|r|)
 [\log(1+\langle r\rangle)]^2}.
\]
Then $h\in L^1(\mathbb R)$. For fixed $\ell\neq k$, consider the
change of variables
\[
\rho=
\frac{\xi-\eta+k(t-\tau)}{k-\ell},
\qquad
q=\eta+\ell\tau.
\]
Its Jacobian has absolute value one. Enlarging the image of
$[0,t]\times\mathbb R$ to $\mathbb R^2$, we obtain
\begin{align*}
\left|\mathcal{M}_k^{(3)}(t,\xi)\right|
\leq&
\frac{1}{6D_N}
\chi(\kappa^{1/2}|k|)
\sum_{\ell\neq k}
\int_{\mathbb R^2}
\frac{h(\rho)}
{\langle\ell\rangle^{2N-2}+|q|^{2N-2}}
\,d\rho\,dq
\\
\leq&
\frac{1}{6D_N}
\sum_{\ell\in\mathbb Z}
\left(
\int_{\mathbb R}
\frac{dq}
{\langle\ell\rangle^{2N-2}+|q|^{2N-2}}
\right)
\left(
\int_{\mathbb R}h(\rho)\,d\rho
\right).
\end{align*}
For $N\geq4$,
\[
\int_{\mathbb R}
\frac{dq}
{\langle\ell\rangle^{2N-2}+|q|^{2N-2}}
=
\langle\ell\rangle^{-(2N-3)}
\frac{\pi}{N-1}
\csc\left(\frac{\pi}{2N-2}\right),
\]
while
\[
\int_{\mathbb R}h(\rho)\,d\rho
=
2\int_0^\infty
\frac{ds}
{(1+s)
[\log(1+\sqrt{1+s^2})]^2}.
\]
Therefore, by the definition of $D_N$,
\[
\left|\mathcal{M}_k^{(3)}(t,\xi)\right|
\leq
\frac{D_N}{6D_N}
=\frac16.
\]
Thus, for all $t\geq0$, $k\in\mathbb Z$, and $\xi\in\mathbb R$, we prove \eqref{pre1}. Moreover,
since $\mathcal{M}_k^{(1)}(\xi)$ and $\mathcal{M}_k^{(2)}(\xi)$ take values in
$[-1/6,1/6]$, while $\mathcal{M}_k^{(3)}(t,\xi)\leq0$, we have
\begin{align*}
\frac16
\leq&
\mathcal{M}_k^{(1)}(\xi)
+\mathcal{M}_k^{(2)}(\xi)
+\mathcal{M}_k^{(3)}(t,\xi)
+\frac23
\leq1.
\end{align*}
This proves \eqref{chengzijie}.

For $k\neq0$, direct computation yields
\begin{align*}
 k\partial_\xi \mathcal{M}_k^{(1)}(\xi)
 =
 \frac{\kappa^{1/3}|k|^{2/3}}
 {3\pi(1+\kappa^{2/3}|k|^{-2/3}\xi^2)}.
\end{align*}
If $|\xi|\leq\kappa^{-1/3}|k|^{1/3}$, then $k\partial_\xi \mathcal{M}_k^{(1)}(\xi) \geq \frac{1}{6\pi}\kappa^{1/3}|k|^{2/3}$. Conversely, if $|\xi|\geq\kappa^{-1/3}|k|^{1/3}$, then $\kappa\xi^2\geq\kappa^{1/3}|k|^{2/3}$. Summing the pointwise lower bound over $k$ establishes \eqref{pre2}. Estimate \eqref{pre3} follows directly from the definition $k\partial_\xi \mathcal{M}_k^{(2)}(\xi)=\frac{1}{6D_\delta}\delta_k'(\xi)\geq0$.

Using the smooth extension introduced above, direct differentiation
gives
\begin{align*}
\partial_\xi\mathcal M_k^{(1)}(\xi)
&=
\frac{\kappa^{1/3}|k|^{-1/3}\operatorname{sgn}(k)}
{3\pi\left(
1+\kappa^{2/3}|k|^{-2/3}\xi^2
\right)},
\quad
\partial_k\mathcal M_k^{(1)}(\xi)
=
-\frac{
\kappa^{1/3}|k|^{-4/3}\xi
}{
9\pi\left(
1+\kappa^{2/3}|k|^{-2/3}\xi^2
\right)
}.
\end{align*}
Consequently,
\begin{align*}
\left|\partial_\xi\mathcal M_k^{(1)}(\xi)\right|
&\lesssim
\kappa^{1/3}|k|^{-1/3}.
\end{align*}
Moreover, we obtain
\begin{align*}
\left|\partial_k\mathcal M_k^{(1)}(\xi)\right|
&=
\frac{1}{9\pi|k|}
\frac{\kappa^{1/3}|k|^{-1/3}|\xi|}{1+(\kappa^{1/3}|k|^{-1/3}|\xi|)^2}
\lesssim
|k|^{-1}.
\end{align*}
After restricting $k$ to $\mathbb Z\setminus\{0\}$, these are the
corresponding bounds in \eqref{pre4}. The estimate for
$\partial_\xi\mathcal M_k^{(2)}(\xi)$ follows directly from its definition.

For $\mathcal{M}_k^{(3)}(t,\xi)$, differentiating under the integral and using the same change of variables yields
\begin{align*}
 |\partial_\xi \mathcal{M}_k^{(3)}(t,\xi)|
 \lesssim
 \sum_{\ell\neq k}
 \frac{1}{|k-\ell|}
 \int_{\mathbb R}
 \frac{dq}
 {\langle\ell\rangle^{2N-2}+|q|^{2N-2}}
 \lesssim
 \sum_{\ell\neq k}
 \frac{\langle\ell\rangle^{-(2N-3)}}{|k-\ell|}.
\end{align*}
This sum is bounded by $C|k|^{-1}$ on the region $|k-\ell|\geq|k|/2$. On the complementary region, $|\ell|\approx|k|$, and the sum $\sum_{0<|k-\ell|<|k|/2} \langle\ell\rangle^{-(2N-3)}|k-\ell|^{-1} \lesssim |k|^{-(2N-3)}\log(2+|k|) \lesssim |k|^{-1}$, completing the proof of \eqref{pre4}.

Estimate \eqref{pre5} follows from the mean value theorem applied to
the smooth extension of $\mathcal M_k^{(1)}(\xi)$ in the real parameter
$k\neq0$.
Under the assumptions
$
k\ell>0$,
$
|k-\ell|\leq \frac12|\ell|,
$
we have

$$
|k-s(k-\ell)|\approx |k|,
\qquad s\in[0,1].
$$
Applying the mean value theorem and \eqref{pre4}, we obtain
\begin{align*}
\bigl|\mathcal{M}_k^{(1)}(\xi)-\mathcal{M}_\ell^{(1)}(\eta)\bigr|
&\le \int_0^1 \big( |\partial_k\mathcal{M}_k^{(1)}|\,|k-\ell| + |\partial_\xi \mathcal{M}_k^{(1)}|\,|\xi-\eta| \big) ds \\
&\lesssim \int_0^1 \left( \frac{|k-\ell|}{|k-s(k-\ell)|}
+\kappa^{\frac13}\,|k-s(k-\ell)|^{-\frac13}\,|\xi-\eta| \right) ds \\
&\lesssim \frac{|k-\ell|}{|k|}+\kappa^{\frac13}\,|k|^{-\frac13}\,|\xi-\eta|.
\end{align*}

To establish the coercivity bound \eqref{pre8}, observe that $(-\partial_t+k\partial_\xi)\mathcal{A}_k(t,\xi) = -\frac{1}{32}\kappa^{1/3}\mathbf{1}_{k\neq0}\mathcal{A}_k(t,\xi)$.
For $k\neq0$, the product rule, the identities $k\partial_\xi\mathcal{M}_k^{(2)}(\xi) = \frac{1}{6D_\delta}\delta_k'(\xi)$, $(-\partial_t+k\partial_\xi)\mathcal{M}_k^{(3)}(t,\xi) = \mathcal{G}_k(t,\xi)$, and the bounds in \eqref{chengzijie} along with $\kappa\le\sigma$ yield
\begin{align}\label{pre8-intermediate}
&2\sigma\xi^2\mathcal{W}_k(t,\xi)
+\bigl(-\partial_t+k\partial_\xi\bigr)\mathcal{W}_k(t,\xi)
\nonumber\\
&\quad\ge
\mathcal{A}_k(t,\xi)\biggl[
\frac16\sigma\xi^2
+\frac16\kappa\xi^2
+k\partial_\xi\mathcal{M}_k^{(1)}
+\frac{1}{6D_\delta}\delta_k'
+\mathcal{G}_k
-\frac{1}{32}\kappa^{1/3}
\biggr].
\end{align}
Setting $s=\kappa^{2/3}|k|^{-2/3}\xi^2 \ge 0$, the explicit formula for $\mathcal{M}_k^{(1)}(\xi)$ gives
\begin{align*}
\frac16\kappa\xi^2
+k\partial_\xi\mathcal{M}_k^{(1)}(\xi)
=
\kappa^{1/3}|k|^{2/3}
\left(
\frac{s}{6}
+\frac{1}{3\pi(1+s)}
\right).
\end{align*}
The function $s \mapsto \frac{s}{6}+\frac{1}{3\pi(1+s)}$ is strictly increasing for $s \ge 0$. Hence, its minimum occurs at $s=0$, yielding
\begin{align*}
\frac16\kappa\xi^2
+k\partial_\xi\mathcal{M}_k^{(1)}(\xi)
\ge
\frac{1}{3\pi}\kappa^{1/3}|k|^{2/3}.
\end{align*}
Since $|k|\geq1$ for $k\neq0$ and
$
\frac{1}{32}<\frac{1}{6\pi},
$
we have
\begin{align*}
\frac{1}{32}\kappa^{1/3}
&\leq
\frac{1}{32}
\kappa^{1/3}|k|^{2/3}
\leq
\frac{1}{6\pi}
\kappa^{1/3}|k|^{2/3}.
\end{align*}
Consequently,
\begin{align*}
\left(
\frac{1}{3\pi}-\frac{1}{32}
\right)
\kappa^{1/3}|k|^{2/3}
\geq
\frac{1}{6\pi}
\kappa^{1/3}|k|^{2/3}.
\end{align*}
Inserting this estimate into \eqref{pre8-intermediate} gives
\eqref{pre8} for $k\neq0$.

For $k=0$, we have $\mathcal A_0(t,\xi)=1$, $\mathcal{M}_0^{(1)}(\xi)=\mathcal{M}_0^{(2)}(\xi)=\delta_0'(\xi)=0$, and $-\partial_t\mathcal{W}_0(t,\xi) = \mathcal{G}_0(t,\xi)$. Using $\mathcal{W}_0(t,\xi) \ge 1/6$, we obtain
\begin{align*}
2\sigma\xi^2\mathcal{W}_0(t,\xi)
-\partial_t\mathcal{W}_0(t,\xi)
=
2\sigma\xi^2\mathcal{W}_0(t,\xi)+\mathcal{G}_0(t,\xi)
\ge
\frac16\sigma\xi^2+\mathcal{G}_0(t,\xi),
\end{align*}
which completes the proof of Lemma \ref{chengziyinli}.
\end{proof}

We next derive commutator estimates for the composite weight $\mathcal{W}_k(t,\xi)\Lambda_t^{2N}(k,\xi)$. These estimates will be used to control derivative
losses in the nonlinear transport and magnetic stretching terms.

\begin{lemma}[Commutator bounds]\label{commutator}
Assume that
\begin{align*}
k,\ell\in\mathbb Z\setminus\{0\},\qquad
k\ell>0,\qquad
|k-\ell|\le \frac12|\ell|,
\qquad
\min\{|k|,|\ell|\}\ge 2 \kappa^{-1/2}.
\end{align*}
Then, for all $\xi,\eta\in\mathbb R$,
\begin{align}\label{hdga}
\big|
\mathcal{W}_k(t,\xi)-\mathcal{W}_\ell(t,\eta)
\big|
\lesssim
\mathcal{A}_k^{1/2}(t,\xi)\mathcal{A}^{1/2}_\ell(t,\eta)
\left(
{\Lambda_t(k-\ell,\xi-\eta)}/{|k|}
+
\kappa^{1/3}|k|^{-1/3}|\xi-\eta|
\right).
\end{align}
Moreover, for all $k,\ell\in\mathbb Z$ and all $\xi,\eta\in\mathbb R$,
\begin{align}\label{keyjiaohuanzi}
&
\big|
\ell \mathcal{W}_k(t,\xi)\Lambda_t^{2N}(k,\xi)
-
k \mathcal{W}_\ell(t,\eta)\Lambda_t^{2N}(\ell,\eta)
\big|
\notag
\\
&\quad\lesssim
|\ell|\,
\mathcal{A}_k^{1/2}(t,\xi)
\mathcal{A}^{1/2}_\ell(t,\eta)
\mathcal{A}_{k-\ell}^{1/2}(t,\xi-\eta)
\Lambda_t^N(k,\xi)
\Lambda_t^N(k-\ell,\xi-\eta)
\notag
\\
&\qquad
+
|k|\,
\mathcal{A}_k^{1/2}(t,\xi)
\mathcal{A}^{1/2}_\ell(t,\eta)
\mathcal{A}_{k-\ell}^{1/2}(t,\xi-\eta)
\Lambda_t^N(\ell,\eta)
\Lambda_t^N(k-\ell,\xi-\eta)
\notag
\\
&\qquad
+
\mathcal{A}_k^{1/2}(t,\xi)
\mathcal{A}^{1/2}_\ell(t,\eta)
\mathcal{A}_{k-\ell}^{1/2}(t,\xi-\eta)
\Lambda_t^N(k,\xi)
\Lambda_t^N(\ell,\eta)\Lambda_t(k-\ell,\xi-\eta)
\notag
\\
&\qquad
+
\kappa^{-1/6}
\mathcal{A}_k^{1/2}(t,\xi)
\mathcal{A}^{1/2}_\ell(t,\eta)
|k|^{1/3}\Lambda_t^N(k,\xi)
|\ell|^{1/3}\Lambda_t^N(\ell,\eta)
\Lambda_t(k-\ell,\xi-\eta).
\end{align}
\end{lemma}

\begin{proof}
We first prove \eqref{hdga}. The assumptions imply that
$k,\ell\neq0$ and $|k|\approx|\ell|$. Moreover,
\[
\kappa^{1/2}|k|\geq2,
\qquad
\kappa^{1/2}|\ell|\geq2.
\]
Since $\chi(r)=0$ for $r\geq2$, we have
\[
\chi(\kappa^{1/2}|k|)
=
\chi(\kappa^{1/2}|\ell|)
=0.
\]
Consequently,
\[
\mathcal M_k^{(2)}(\xi)
=
\mathcal M_\ell^{(2)}(\eta)
=
\mathcal M_k^{(3)}(t,\xi)
=
\mathcal M_\ell^{(3)}(t,\eta)
=0.
\]
In addition,
\[
\mathcal A_k(t,\xi)
=
\mathcal A_\ell(t,\eta)
=
e^{\frac{1}{32}\kappa^{1/3}t}.
\]
Therefore,
\begin{align*}
\mathcal W_k(t,\xi)-\mathcal W_\ell(t,\eta)
=
\mathcal A_k(t,\xi)
\left(
\mathcal M_k^{(1)}(\xi)
-
\mathcal M_\ell^{(1)}(\eta)
\right).
\end{align*}
The definition of $\mathcal M^{(1)}_k(\xi)$ and the assumptions
$k\ell>0$, $|k-\ell|\leq|\ell|/2$ imply
\begin{align*}
\left|
\mathcal M_k^{(1)}(\xi)
-
\mathcal M_\ell^{(1)}(\eta)
\right|
\lesssim
\frac{|k-\ell|}{|k|}
+
\kappa^{1/3}|k|^{-1/3}|\xi-\eta|.
\end{align*}
Hence
\begin{align*}
\left|
\mathcal W_k(t,\xi)-\mathcal W_\ell(t,\eta)
\right|
\lesssim&
\mathcal A_k(t,\xi)
\left(
{|k-\ell|}/{|k|}
+
\kappa^{1/3}|k|^{-1/3}|\xi-\eta|
\right)
\\
\lesssim&
\mathcal A^{1/2}_k(t,\xi)
\mathcal A^{1/2}_\ell(t,\eta)
\left(
{\Lambda_t(k-\ell,\xi-\eta)}/{|k|}
+
\kappa^{1/3}|k|^{-1/3}|\xi-\eta|
\right),
\end{align*}
because
\[
|k-\ell|\leq \Lambda_t(k-\ell,\xi-\eta).
\]
This proves \eqref{hdga}.

We turn to \eqref{keyjiaohuanzi}. The case $k=\ell=0$ is trivial. We shall use
\begin{align*}
\Lambda_t(k,\xi)
\lesssim
\Lambda_t(\ell,\eta)+\Lambda_t(k-\ell,\xi-\eta),
\end{align*}
and
\begin{align*}
\mathcal{A}_k(t,\xi)
\le
\mathcal{A}_k^{1/2}(t,\xi)
\mathcal{A}^{1/2}_\ell(t,\eta)
\mathcal{A}_{k-\ell}^{1/2}(t,\xi-\eta),
\end{align*}
as well as the analogous bound with $k$ and $\ell$ interchanged. The
latter inequality follows from $k=\ell+(k-\ell)$ and the definition of
$\mathcal{A}$.

We split the argument into four cases.

\medskip
\noindent
\emph{Case 1.}
Assume
$
\Lambda_t(k-\ell,\xi-\eta)
\ge
\frac12\Lambda_t(\ell,\eta).
$
Then
\begin{align*}
\Lambda_t(k,\xi)\lesssim \Lambda_t(k-\ell,\xi-\eta).
\end{align*}
Hence, by the uniform bound $|\mathcal{W}_m|\lesssim \mathcal A_m(t,\cdot)$,
\begin{align*}
\begin{aligned}
&
\big|
\ell \mathcal{W}_k(t,\xi)\Lambda_t^{2N}(k,\xi)
-
k\mathcal{W}_\ell(t,\eta)\Lambda_t^{2N}(\ell,\eta)
\big|
\\
&\quad
\le
|\ell|\,|\mathcal{W}_k(t,\xi)|\Lambda_t^{2N}(k,\xi)
+
|k|\,|\mathcal{W}_\ell(t,\eta)|\Lambda_t^{2N}(\ell,\eta)
\\
&\quad
\lesssim
|\ell|\,
\mathcal{A}_k^{1/2}(t,\xi)\mathcal{A}^{1/2}_\ell(t,\eta)
\mathcal{A}_{k-\ell}^{1/2}(t,\xi-\eta)
\Lambda_t^N(k,\xi)\Lambda_t^N(k-\ell,\xi-\eta)
\\
&\qquad
+
|k|\,
\mathcal{A}_k^{1/2}(t,\xi)\mathcal{A}^{1/2}_\ell(t,\eta)
\mathcal{A}_{k-\ell}^{1/2}(t,\xi-\eta)
\Lambda_t^N(\ell,\eta)\Lambda_t^N(k-\ell,\xi-\eta).
\end{aligned}
\end{align*}
This gives the first two terms on the right-hand side of \eqref{keyjiaohuanzi}.

\medskip
\noindent
\emph{Case 2.}
Assume
$
\Lambda_t(k-\ell,\xi-\eta)
<
\frac12\Lambda_t(\ell,\eta),
$ and $
|k-\ell|\ge \frac12|\ell|.
$
Then
\begin{align*}
\Lambda_t(k,\xi)\approx \Lambda_t(\ell,\eta),
\qquad
|k|+|\ell|\lesssim |k-\ell| \lesssim  \Lambda_t(k-\ell,\xi-\eta).
\end{align*}
Therefore
\begin{align*}
\begin{aligned}
&
\big|
\ell \mathcal{W}_k(t,\xi)\Lambda_t^{2N}(k,\xi)
-
k\mathcal{W}_\ell(t,\eta)\Lambda_t^{2N}(\ell,\eta)
\big|
\\
&\quad
\lesssim
\mathcal{A}_k^{1/2}(t,\xi)\mathcal{A}^{1/2}_\ell(t,\eta)
\mathcal{A}_{k-\ell}^{1/2}(t,\xi-\eta)
\Lambda_t^N(k,\xi)\Lambda_t^N(\ell,\eta)\Lambda_t(k-\ell,\xi-\eta),
\end{aligned}
\end{align*}
which is the third term in \eqref{keyjiaohuanzi}.

\medskip
\noindent
\emph{Case 3.}
Assume
$
\Lambda_t(k-\ell,\xi-\eta)
<
\frac12\Lambda_t(\ell,\eta),
$$
|k-\ell|<\frac12|\ell|,
$ and $
\min\big\{|k|,|\ell|\big\}\le 2\kappa^{-1/2}.
$
Then $k,\ell\ne0$,
\begin{align*}
|k|\approx|\ell|,
\qquad
\Lambda_t(k,\xi)\approx\Lambda_t(\ell,\eta),
\qquad
|k|\approx|\ell|\lesssim \kappa^{-1/2}.
\end{align*}
In particular,
\begin{align*}
|k|+|\ell|
\lesssim
\kappa^{-1/6}|k|^{1/3}|\ell|^{1/3}.
\end{align*}
We decompose
\begin{align*}
\begin{aligned}
\big|
\ell \mathcal{W}_k(t,\xi)\Lambda_t^{2N}(k,\xi)
-
k\mathcal{W}_\ell(t,\eta)\Lambda_t^{2N}(\ell,\eta)
\big|
\le&
|\ell|\Lambda_t^{2N}(k,\xi)
|\mathcal{W}_k(t,\xi)-\mathcal{W}_\ell(t,\eta)|
\\
&
+
|k-\ell|\,|\mathcal{W}_\ell(t,\eta)|\Lambda_t^{2N}(k,\xi)
\\
&
+
|k|\,|\mathcal{W}_\ell(t,\eta)|
\big|
\Lambda_t^{2N}(k,\xi)-\Lambda_t^{2N}(\ell,\eta)
\big|.
\end{aligned}
\end{align*}
Since $|\mathcal{W}_k-\mathcal{W}_\ell|\lesssim
\mathcal{A}_k^{1/2}(t,\xi)\mathcal{A}^{1/2}_\ell(t,\eta)$ and
$\Lambda_t(k-\ell,\xi-\eta)\ge1$, the first term is bounded by the last
term in \eqref{keyjiaohuanzi}. The second term is bounded by the third term in
\eqref{keyjiaohuanzi}. For the last term, the standard difference estimate gives
\begin{align*}
\big|
\Lambda_t^{2N}(k,\xi)-\Lambda_t^{2N}(\ell,\eta)
\big|
\lesssim
\Lambda_t^{2N-1}(k,\xi)\Lambda_t(k-\ell,\xi-\eta)
\end{align*}
in the present region. Since $|k|\le\Lambda_t(k,\xi)$, this term is again
controlled by the third term in \eqref{keyjiaohuanzi}.

\medskip
\noindent
\emph{Case 4.}
Assume
$
\Lambda_t(k-\ell,\xi-\eta)
<
\frac12\Lambda_t(\ell,\eta),
$ $
|k-\ell|<\frac12|\ell|,
$ and
$\min\{|k|,|\ell|\}>2\kappa^{-1/2}.
$
Then
\begin{align*}
|k|\approx|\ell|,
\qquad
\Lambda_t(k,\xi)\approx\Lambda_t(\ell,\eta),
\qquad
\kappa^{-1/6}|k|^{1/3}|\ell|^{1/3}\gtrsim1.
\end{align*}
We use the same decomposition as in Case 3. The terms containing
$|k-\ell|$ and the difference of the Sobolev weights are estimated as
above and are bounded by the third and fourth terms in \eqref{keyjiaohuanzi},
respectively.

It remains to treat the term
$
|\ell|\Lambda_t^{2N}(k,\xi)
|\mathcal{W}_k(t,\xi)-\mathcal{W}_\ell(t,\eta)|.
$
Indeed, by \eqref{hdga}, we have
\begin{align*}
\begin{aligned}
|\ell|\Lambda_t^{2N}(k,\xi)
|\mathcal{W}_k(t,\xi)-\mathcal{W}_\ell(t,\eta)|
\lesssim&
\mathcal{A}_k^{1/2}(t,\xi)\mathcal{A}^{1/2}_\ell(t,\eta)
\Lambda_t^{2N}(k,\xi)\Lambda_t(k-\ell,\xi-\eta)
\\
&
+
\mathcal{A}_k^{1/2}(t,\xi)\mathcal{A}^{1/2}_\ell(t,\eta)
\kappa^{1/3}|k|^{2/3}
\Lambda_t^{2N}(k,\xi)|\xi-\eta|.
\end{aligned}
\end{align*}
The first term is bounded by the last term in \eqref{keyjiaohuanzi}. For the
second term, if $k=\ell$, then
\begin{align*}
|\xi-\eta|\le \Lambda_t(k-\ell,\xi-\eta),
\end{align*}
and the desired bound follows. If $k\ne\ell$, then
\begin{align*}
|\xi-\eta|
\le
|\xi-\eta+(k-\ell)t|+t|k-\ell|
\lesssim
\langle t\rangle\Lambda_t(k-\ell,\xi-\eta).
\end{align*}
Using $\langle t\rangle e^{-\frac{1}{64}\kappa^{1/3}t}
\lesssim
\kappa^{-1/3}$
and the factor
$
\mathcal{A}_{k-\ell}^{1/2}(t,\xi-\eta)=e^{\frac{1}{64}\kappa^{1/3}t},
$
we get
\begin{align*}
&
\mathcal{A}_k^{1/2}(t,\xi)\mathcal{A}^{1/2}_\ell(t,\eta)
\kappa^{1/3}|k|^{2/3}
\Lambda_t^{2N}(k,\xi)|\xi-\eta|
\\
&\quad\lesssim
\mathcal{A}_k^{1/2}(t,\xi)\mathcal{A}^{1/2}_\ell(t,\eta)
\mathcal{A}_{k-\ell}^{1/2}(t,\xi-\eta)
|k|^{2/3}
\Lambda_t^N(k,\xi)\Lambda_t^N(\ell,\eta)
\Lambda_t(k-\ell,\xi-\eta).
\end{align*}
Since $|k|\approx|\ell|$ and $\kappa^{-1/6}\ge1$, this is dominated by
the last term in \eqref{keyjiaohuanzi}. Combining the four cases proves \eqref{keyjiaohuanzi} and completes the proof of Lemma \ref{commutator}.
\end{proof}

We will use the following convolution estimate for the nonlinear product terms in the sheared Sobolev norm $\Lambda_t^N$.
\begin{lemma}\label{leyoung}
For $N\geq4$, one has
\begin{align}\label{sanxianyang}
&\left|
\sum_{k,\ell\in\Z}\int_{\R^2}
\Lambda_t(k-\ell,\xi-\eta)
\Lambda_t^N(k,\xi)
\Lambda_t^N(\ell,\eta)
f_{k-\ell}(\xi-\eta)
g_\ell(\eta)
\overline{h_k(\xi)}
\,d\eta\,d\xi
\right|
\nn\\
&\quad\lesssim
\|\Lambda_t^N f\|_{L^2}
\|\Lambda_t^N g\|_{L^2}
\|\Lambda_t^N h\|_{L^2}.
\end{align}
\end{lemma}
\begin{proof}
The estimate follows from the convolution version of Young's
inequality. We only need to observe that
$\Lambda_t^{-s}\in \ell_k^1L_\xi^1$ whenever $s>2$.
Indeed, recalling that
$
\Lambda_t(k,\xi)
=
\bigl(1+k^2+(\xi+kt)^2\bigr)^{1/2},
$
and making the change of variables
$
r=\frac{\xi+kt}{\sqrt{1+k^2}},
$
we obtain
\begin{align*}
\sum_{k\in\mathbb Z}\int_{\mathbb R}
\Lambda_t^{-s}(k,\xi)\,d\xi
&=
\sum_{k\in\mathbb Z}\int_{\mathbb R}
\bigl(1+k^2+(\xi+kt)^2\bigr)^{-s/2}\,d\xi
\\
&=
\left(
\int_{\mathbb R}(1+r^2)^{-s/2}\,dr
\right)
\sum_{k\in\mathbb Z}
(1+k^2)^{(1-s)/2}.
\end{align*}
Moreover,
\begin{align*}
\int_{\mathbb R}(1+r^2)^{-s/2}\,dr
=
\sqrt{\pi}\,
\frac{
\Gamma\bigl(\frac{s-1}{2}\bigr)
}{
\Gamma\bigl(\frac{s}{2}\bigr)
}
<\infty
\qquad\text{for }s>1,
\end{align*}
whereas
\begin{align*}
\sum_{k\in\mathbb Z}
(1+k^2)^{(1-s)/2}
<\infty
\qquad\text{for }s>2.
\end{align*}
Consequently,
\begin{align*}
\sum_{k\in\mathbb Z}\int_{\mathbb R}
\Lambda_t^{-s}(k,\xi)\,d\xi
\lesssim 1,
\qquad s>2,
\end{align*}
uniformly in $t\geq0$. The desired trilinear estimate then follows
from Young's convolution inequality. This proves
Lemma~\ref{leyoung}.
\end{proof}

\section{The proof of Theorem~\ref{thm1.1}}\label{proofth1}

In this section, we prove Theorem~\ref{thm1.1}. Throughout this
section, all the estimates are understood to hold for the nonzero
horizontal modes.

\subsection{Sheared coordinates}
We first eliminate the transport operator
$\partial_t+y\partial_x$ by introducing the sheared coordinates
\begin{align*}
	X=x-ty,
	\qquad
	Y=y.
\end{align*}
We define the transformed unknowns by
\begin{align*}
	\mathbf U(t,X,Y)
	&=\u(t,x,y),
	&
	\mathbf B(t,X,Y)
	&=\mathbf b(t,x,y),&
	\Omega(t,X,Y)
	&=\omega(t,x,y),\\
	J(t,X,Y)
	&=j(t,x,y),
	&
	\Psi(t,X,Y)
	&=\psi(t,x,y),
	&
	\Phi(t,X,Y)
	&=\phi(t,x,y).
\end{align*}
In these coordinates, system~\eqref{rewrite1} becomes
\begin{equation}\label{rewrite2}
\left\{
\begin{aligned}
	&\partial_t\Omega-\beta\partial_XJ=0,
	\\
	&\partial_tJ-\beta\partial_X\Omega
		+2\partial_X\partial_Y^L\Phi=0,
	\\
	&\mathbf U
		=\nabla_L^\perp\Psi
		=(-\partial_Y^L\Psi,\partial_X\Psi),
	\qquad
	\Delta_L\Psi=\Omega,
	\\
	&\mathbf B
		=\nabla_L^\perp\Phi
		=(-\partial_Y^L\Phi,\partial_X\Phi),
	\qquad
	\Delta_L\Phi=J,
\end{aligned}
\right.
\end{equation}
where
\begin{align*}
	\nabla_L
	=
	\begin{pmatrix}
		\partial_X\\
		\partial_Y^L
	\end{pmatrix}
	=
	\begin{pmatrix}
		\partial_X\\
		\partial_Y-t\partial_X
	\end{pmatrix},
	\qquad
	\Delta_L
	=
	\partial_X^2+(\partial_Y-t\partial_X)^2.
\end{align*}

Passing to Fourier variables
$
	(k,\xi)\in\mathbb Z\times\mathbb R
$
associated with $(X,Y)$, we denote the symbol of $-\Delta_L$ by
\begin{align*}
	p(t,k,\xi)
	:=
	k^2+(\xi-kt)^2.
\end{align*}
Its time derivative is
\begin{align*}
	p_t(t,k,\xi)
	=
	-2k(\xi-kt).
\end{align*}

\subsection{Symmetric variables and a modified energy}

Following~\cite{D24}, we introduce the symmetric variables
\begin{align*}
	Z:=(-\Delta_L)^{-\frac12}\Omega_{\neq},
	\qquad
	Q:=(-\Delta_L)^{-\frac12}J_{\neq}.
\end{align*}
For every $k\neq0$, their Fourier transforms satisfy
\begin{equation}\label{rewrite3}
\left\{
\begin{aligned}
	&\partial_t\widehat Z_k
		+\frac12\frac{p_t}{p}\widehat Z_k
		-i\beta k\widehat Q_k=0,
	\\
	&\partial_t\widehat Q_k
		-\frac12\frac{p_t}{p}\widehat Q_k
		-i\beta k\widehat Z_k=0.
\end{aligned}
\right.
\end{equation}

For each $k\neq0$, define the modified energy
\begin{equation*}
	\widetilde E_k(t,\xi)
	:=
	\frac12
	\left[
	|\widehat Z_k|^2+|\widehat Q_k|^2
	+
	\frac2\beta
	\operatorname{Re}
	\left(
	-\frac{i(\xi-kt)}{p}
	\widehat Z_k\overline{\widehat Q_k}
	\right)
	\right].
\end{equation*}
Since $k\neq0$, we have
\begin{align*}
	\frac{|\xi-kt|}{p}
	=
	\frac{|\xi-kt|}
	{k^2+(\xi-kt)^2}
	\leq
	\frac1{2|k|}
	\leq
	\frac12.
\end{align*}
It follows that
\begin{align*}
	\left|
	\frac1\beta
	\operatorname{Re}
	\left(
	-\frac{i(\xi-kt)}p
	\widehat Z_k\overline{\widehat Q_k}
	\right)
	\right|
	&\leq
	\frac1{2|\beta|}
	|\widehat Z_k|\,|\widehat Q_k|
	\leq
	\frac1{4|\beta|}
	\left(
	|\widehat Z_k|^2+|\widehat Q_k|^2
	\right).
\end{align*}
Therefore, if $|\beta|>\frac12$, then
\begin{equation}\label{xianxingdengjia}
\begin{aligned}
	\left(
	\frac12-\frac1{4|\beta|}
	\right)
	\left(
	|\widehat Z_k|^2+|\widehat Q_k|^2
	\right)
	&\leq
	\widetilde E_k(t,\xi)
	\leq
	\left(
	\frac12+\frac1{4|\beta|}
	\right)
	\left(
	|\widehat Z_k|^2+|\widehat Q_k|^2
	\right).
\end{aligned}
\end{equation}
In particular,
\begin{align*}
	\widetilde E_k(t,\xi)
	\approx_\beta
	|\widehat Z_k(t,\xi)|^2
	+
	|\widehat Q_k(t,\xi)|^2.
\end{align*}

We next compute the time derivative of $\widetilde E_k$. System
\eqref{rewrite3} can be rewritten as
\begin{align*}
	\partial_t\widehat Z_k
	=
	-\frac12\frac{p_t}{p}\widehat Z_k
	+i\beta k\widehat Q_k,
	\qquad
	\partial_t\widehat Q_k
	=
	\frac12\frac{p_t}{p}\widehat Q_k
	+i\beta k\widehat Z_k.
\end{align*}
Taking the real part after multiplying by the corresponding complex
conjugates gives
\begin{equation}\label{jiben}
	\frac{d}{dt}
	\frac12
	\left(
	|\widehat Z_k|^2+|\widehat Q_k|^2
	\right)
	=
	\frac12\frac{p_t}{p}
	\left(
	|\widehat Q_k|^2-|\widehat Z_k|^2
	\right).
\end{equation}
The terms generated by the Alfv\'enic coupling cancel in the sum.

Moreover,
\begin{align*}
	\frac{d}{dt}
	\left(
	\widehat Z_k\overline{\widehat Q_k}
	\right)
	&=
	\left(
	-\frac12\frac{p_t}{p}\widehat Z_k
	+i\beta k\widehat Q_k
	\right)
	\overline{\widehat Q_k}
	+
	\widehat Z_k
	\left(
	\frac12\frac{p_t}{p}\overline{\widehat Q_k}
	-i\beta k\overline{\widehat Z_k}
	\right)
	\nn\\
	&=
	i\beta k
	\left(
	|\widehat Q_k|^2-|\widehat Z_k|^2
	\right).
\end{align*}
Consequently,
\begin{align}\label{jiben2}
	&\frac{d}{dt}
	\left[
	\frac1\beta
	\operatorname{Re}
	\left(
	-\frac{i(\xi-kt)}p
	\widehat Z_k\overline{\widehat Q_k}
	\right)
	\right]
	\nn\\
	&\quad=
	\frac1\beta
	\partial_t
	\left(
	\frac{\xi-kt}{p}
	\right)
	\operatorname{Im}
	\left(
	\widehat Z_k\overline{\widehat Q_k}
	\right)
	+
	k\frac{\xi-kt}{p}
	\left(
	|\widehat Q_k|^2-|\widehat Z_k|^2
	\right).
\end{align}
Since
$
	p_t=-2k(\xi-kt),
$
we have
$
	k\frac{\xi-kt}{p}
	=
	-\frac12\frac{p_t}{p}.
$
Thus, the last term in
\eqref{jiben2} cancels exactly with the
right-hand side of \eqref{jiben}. We therefore
obtain
\begin{align*}
	\frac{d}{dt}\widetilde E_k(t,\xi)
	&=
	\frac1\beta
	\partial_t
	\left(
	\frac{\xi-kt}{p}
	\right)
	\operatorname{Im}
	\left(
	\widehat Z_k\overline{\widehat Q_k}
	\right)
	=
	\frac{k\bigl((\xi-kt)^2-k^2\bigr)}
	{\beta p^2}
	\operatorname{Im}
	\left(
	\widehat Z_k\overline{\widehat Q_k}
	\right).
\end{align*}

Furthermore,
\begin{align*}
	\left|
	\partial_t
	\left(
	\frac{\xi-kt}{p}
	\right)
	\right|
	=
	\frac{|k|
	\left|(\xi-kt)^2-k^2\right|}
	{p^2}
	\leq
	\frac{|k|}{p}
	\leq
	\frac{k^2}{p},
\end{align*}
where the last inequality follows from
$k\in\mathbb Z\setminus\{0\}$. Combining this bound with
\begin{align*}
	\left|
	\operatorname{Im}
	\left(
	\widehat Z_k\overline{\widehat Q_k}
	\right)
	\right|
	\leq
	\frac12
	\left(
	|\widehat Z_k|^2+|\widehat Q_k|^2
	\right)
	\le C_\beta
	\widetilde E_k(t,\xi),
\end{align*}
we deduce that
\begin{equation}\label{jiben3}
	\left|
	\frac{d}{dt}\widetilde E_k(t,\xi)
	\right|
	\leq
	C_\beta
	\frac{k^2}{p(t,k,\xi)}
	\widetilde E_k(t,\xi).
\end{equation}

For every $k\neq0$ and $\xi\in\mathbb R$,
\begin{align*}
	\int_0^\infty
	\frac{k^2}
	{k^2+(\xi-kt)^2}
	\,dt
	&\leq
	\int_{-\infty}^{\infty}
	\frac{k^2}{k^2+s^2}
	\frac{\mathrm ds}{|k|}
	=
	\int_{-\infty}^{\infty}
	\frac{\mathrm dr}{1+r^2}
	=
	\pi.
\end{align*}
Applying Gr\"onwall's inequality to
\eqref{jiben3}, we obtain
\begin{align*}
	e^{-\pi C_\beta}\widetilde E_k(0,\xi)
	\leq
	\widetilde E_k(t,\xi)
	\leq
	e^{\pi C_\beta}\widetilde E_k(0,\xi).
\end{align*}
Together with \eqref{xianxingdengjia}, this yields the
pointwise-in-frequency estimate
\begin{equation*}
	|\widehat Z_k(t,\xi)|^2
	+
	|\widehat Q_k(t,\xi)|^2
	\approx_\beta
	|\widehat Z_k(0,\xi)|^2
	+
	|\widehat Q_k(0,\xi)|^2.
\end{equation*}
Recalling the definitions of $Z$ and $Q$, we conclude that
\begin{equation}\label{approx1}
\begin{aligned}
	&
	p(t,k,\xi)^{-1}
	\left(
	|\widehat\Omega_k(t,\xi)|^2
	+
	|\widehat J_k(t,\xi)|^2
	\right)
	\approx_\beta
	(k^2+\xi^2)^{-1}
	\left(
	|\widehat\Omega_k(0,\xi)|^2
	+
	|\widehat J_k(0,\xi)|^2
	\right).
\end{aligned}
\end{equation}

\subsection{Growth of the vorticity and current}

We first derive the upper bound. Since
\begin{align*}
	p(t,k,\xi)
	=
	k^2+(\xi-kt)^2
	\lesssim
	\langle t\rangle^2
	(1+k^2+\xi^2),
\end{align*}
and since $k\neq0$ implies
\begin{align*}
	k^2+\xi^2
	\leq
	1+k^2+\xi^2
	\lesssim
	k^2+\xi^2,
\end{align*}
estimate \eqref{approx1} gives
\begin{align*}
	&
	\int_{\mathbb R}
	\left(
	|\widehat\Omega_k(t,\xi)|^2
	+
	|\widehat J_k(t,\xi)|^2
	\right)
	\,d\xi
	\le C_\beta
	\langle t\rangle^2
	\int_{\mathbb R}
	\left(
	|\widehat\Omega_k(0,\xi)|^2
	+
	|\widehat J_k(0,\xi)|^2
	\right)
	\,d\xi.
\end{align*}
Summing over $k\neq0$, we obtain
\begin{equation}\label{jiben5}
	\|(\Omega_{\neq},J_{\neq})(t)\|_{L^2}
	\le C_\beta
	\langle t\rangle
	\|(\Omega_{\neq},J_{\neq})(0)\|_{L^2}.
\end{equation}

For the lower bound, we use the elementary inequality
\begin{align}\label{bdsh}
\left\langle\frac{\xi}{k}-t\right\rangle
	\left\langle\frac{\xi}{k}\right\rangle
	\gtrsim
	\langle t\rangle,
	\qquad k\neq0.
\end{align}
It follows that
\begin{align*}
	\frac{p(t,k,\xi)}{k^2+\xi^2}
	&=
	\frac{
	1+\left|\frac{\xi}{k}-t\right|^2
	}{
	1+\left|\frac{\xi}{k}\right|^2
	}
	\gtrsim
	\frac{\langle t\rangle^2}
	{\left\langle\frac{\xi}{k}\right\rangle^4}
	=
	\langle t\rangle^2
	\frac{k^4}{(k^2+\xi^2)^2}.
\end{align*}
Therefore, by \eqref{approx1},
\begin{align*}
	&
	\int_{\mathbb R}
	\left(
	|\widehat\Omega_k(t,\xi)|^2
	+
	|\widehat J_k(t,\xi)|^2
	\right)
	\,d\xi
	\\
	&\quad\ge C_\beta^{-1}
	\langle t\rangle^2
	\int_{\mathbb R}
	\frac{k^4}{(k^2+\xi^2)^2}
	\left(
	|\widehat\Omega_k(0,\xi)|^2
	+
	|\widehat J_k(0,\xi)|^2
	\right)
	\,d\xi.
\end{align*}
Summing over $k\neq0$, we arrive at
\begin{equation}\label{jiben6}
\begin{aligned}
	\|(\Omega_{\neq},J_{\neq})(t)\|_{L^2}
	&\gtrsim_\beta
	\langle t\rangle
	\left\|
	\partial_X^2(-\Delta)^{-1}
	(\Omega_{\neq},J_{\neq})(0)
	\right\|_{L^2}.
\end{aligned}
\end{equation}
Since the transformation $(x,y)\mapsto(X,Y)$ preserves the
$L^2$-norm, estimates \eqref{jiben5} and
\eqref{jiben6} give the asserted growth estimates in
the original variables.

\subsection{Bounds for the velocity and magnetic fields}

By the Biot--Savart law,
\begin{align*}
	\mathbf U
	=
	\left(
	-\partial_Y^L\Delta_L^{-1}\Omega,\,
	\partial_X\Delta_L^{-1}\Omega
	\right),
	\qquad
	\mathbf B
	=
	\left(
	-\partial_Y^L\Delta_L^{-1}J,\,
	\partial_X\Delta_L^{-1}J
	\right).
\end{align*}
Hence,
\begin{align*}
	|\widehat U_k^1(t,\xi)|^2
	=
	\frac{|\xi-kt|^2}{p^2}
	|\widehat\Omega_k(t,\xi)|^2,
	\qquad
	|\widehat U_k^2(t,\xi)|^2
	=
	\frac{k^2}{p^2}
	|\widehat\Omega_k(t,\xi)|^2.
\end{align*}

Using \eqref{approx1} and
\begin{align*}
	\frac{|\xi-kt|^2}{p}\leq1,
\end{align*}
we obtain
\begin{align*}
	\|\widehat U_k^1(t)\|_{L_\xi^2}^2
	&=
	\int_{\mathbb R}
	\frac{|\xi-kt|^2}{p}
	\left|
	p^{-\frac12}
	\widehat\Omega_k(t,\xi)
	\right|^2
	\,d\xi
	\\
	&\le C_\beta
	\int_{\mathbb R}
	\frac{|\widehat\Omega_k(0,\xi)|^2+|\widehat J_k(0,\xi)|^2}
	{k^2+\xi^2}
	\,d\xi
	\\
	&\leq
	\frac{C_\beta}{k^2}
	\|(\widehat\Omega_k(0),\widehat J_k(0))\|_{L_\xi^2}^2.
\end{align*}
Since $|k|\geq1$, summing over $k\neq0$ gives
\begin{equation}\label{uyijie}
	\|U_{\neq}^1(t)\|_{L^2}
	\le C_\beta
	\|(\Omega_{\neq}(0),J_{\neq}(0))\|_{L^2}.
\end{equation}

For the second component, estimate \eqref{approx1} yields
\begin{align*}
	\|\widehat U_k^2(t)\|_{L_\xi^2}^2
	&=
	\int_{\mathbb R}
	\frac{k^2}{p}
	\left|
	p^{-\frac12}
	\widehat\Omega_k(t,\xi)
	\right|^2
	\,d\xi
	\\
	&\le C_\beta
	\int_{\mathbb R}
	\frac{k^2}
	{p(t,k,\xi)(k^2+\xi^2)}
	(|\widehat\Omega_k(0,\xi)|^2+|\widehat J_k(0,\xi)|^2)
	\,d\xi.
\end{align*}
The inequality \eqref{bdsh} implies
\begin{align*}
	\frac{k^2}
	{p(t,k,\xi)(k^2+\xi^2)}
	=
	\frac1{
	k^2
	\left\langle\frac{\xi}{k}-t\right\rangle^2
	\left\langle\frac{\xi}{k}\right\rangle^2
	}
	\lesssim
	\frac{\langle t\rangle^{-2}}{k^2}
	\leq
	\langle t\rangle^{-2}
	\langle k,\xi\rangle^2.
\end{align*}
Consequently,
\begin{align*}
	\|\widehat U_k^2(t)\|_{L_\xi^2}^2
	\le C_\beta
	\langle t\rangle^{-2}
	\int_{\mathbb R}
	\langle k,\xi\rangle^2
	(|\widehat\Omega_k(0,\xi)|^2+|\widehat J_k(0,\xi)|^2)
	\,d\xi.
\end{align*}
After summing over $k\neq0$, we obtain
\begin{equation}\label{uerjie}
	\|U_{\neq}^2(t)\|_{L^2}
	\le C_\beta
	\langle t\rangle^{-1}
	\|(\Omega_{\neq}(0),J_{\neq}(0))\|_{H^1}.
\end{equation}

The same argument, with $J$ in place of $\Omega$, gives
\begin{equation}\label{Bbounds}
	\|B_{\neq}^1(t)\|_{L^2}
	\le C_\beta
	\|(\Omega_{\neq}(0),J_{\neq}(0))\|_{L^2},
	\qquad
	\|B_{\neq}^2(t)\|_{L^2}
	\le C_\beta
	\langle t\rangle^{-1}
	\|(\Omega_{\neq}(0),J_{\neq}(0))\|_{H^1}.
\end{equation}
Combining \eqref{uyijie}, \eqref{uerjie}, and
\eqref{Bbounds}, and returning to the original variables, we conclude
that
\begin{align*}
	\|(u_{\neq}^1,b_{\neq}^1)(t)\|_{L^2}
	\le C_\beta
	\|(\omega_{\mathrm{in},\neq},
	j_{\mathrm{in},\neq})\|_{L^2},
\end{align*}
and
\begin{align*}
	\|(u_{\neq}^2,b_{\neq}^2)(t)\|_{L^2}
	\le C_\beta
	\langle t\rangle^{-1}
	\|(\omega_{\mathrm{in},\neq},
	j_{\mathrm{in},\neq})\|_{H^1}.
\end{align*}
This completes the proof of Theorem~\ref{thm1.1}. \hfill $\square$

\section{The proof of Theorem~\ref{thm1.2}}\label{proofth2}

In this section, we establish the weighted energy estimates that form the core of the nonlinear stability proof.

\subsection{A priori energy estimates}\label{sec5}

To rigorously track the pressure gradient in Fourier space, we introduce the Riesz transforms:
\begin{align*}
  \mathcal R_1 := \partial_x(-\Delta)^{-1/2}, \qquad \mathcal R_2 := \partial_y(-\Delta)^{-1/2}.
\end{align*}
Consequently, $\partial_i\partial_{i'}(-\Delta)^{-1} = \mathcal R_i\mathcal R_{i'}$. Because these operators act exclusively on the nonzero horizontal modes in our nonlocal pressure coupling, they correspond to the non-singular, uniformly bounded Fourier symbols:
\begin{align*}
  \widehat{\mathcal R}_1(k,\xi) = \frac{ik}{\sqrt{k^2+\xi^2}}, \qquad \widehat{\mathcal R}_2(k,\xi) = \frac{i\xi}{\sqrt{k^2+\xi^2}}, \qquad k\neq 0.
\end{align*}

Applying the sheared Sobolev operator $\Lambda^N_t$ to system~\eqref{main}, the differentiated velocity and magnetic fields evolve according to:
\begin{equation}\label{quanju2}
\left\{
\begin{aligned}
&\partial_t \bigl(\Lambda^N_t \mathbf{u}\bigr) + y \partial_x \bigl(\Lambda^N_t \mathbf{u}\bigr)
+ (\Lambda^N_tu^2,0)^{\mathrm{T}}
- \beta \partial_x \Lambda^N_t \mathbf{b}
- \nu \partial_{yy}^2 \Lambda^N_t \mathbf{u} + \nabla \Lambda^N_t P
= \Lambda^N_t \mathcal{NL}_{\mathbf{u}},\\[2mm]
&\partial_t \bigl(\Lambda^N_t \mathbf{b}\bigr) + y \partial_x \bigl(\Lambda^N_t \mathbf{b}\bigr)
- (\Lambda^N_tb^2,0)^{\mathrm{T}} - \beta \partial_x \Lambda^N_t \mathbf{u}
- \mu \partial_{yy}^2 \Lambda^N_t \mathbf{b}
= \Lambda^N_t \mathcal{NL}_{\mathbf{b}},\\[1mm]
&\operatorname{div} \mathbf{u} = \operatorname{div} \mathbf{b} = 0,\\[1mm]
&\mathbf{u}(0,x,y) = \mathbf{u}_{\rm in}(x,y), \qquad \mathbf{b}(0,x,y) = \mathbf{b}_{\rm in}(x,y).
\end{aligned}
\right.
\end{equation}
Here, the nonlinear transport and stretching terms are defined as
\begin{align*}
\mathcal{NL}_{\mathbf{u}} := - \mathbf{u} \cdot \nabla \mathbf{u} + \mathbf{b} \cdot \nabla \mathbf{b},
\qquad
\mathcal{NL}_{\mathbf{b}} := - \mathbf{u} \cdot \nabla \mathbf{b} + \mathbf{b} \cdot \nabla \mathbf{u},
\end{align*}
and the associated pressure gradient takes the form
\begin{align*}
\nabla \Lambda^N_t P
= \nabla \big( 2 \partial_x(-\Delta)^{-1} \Lambda^N_t u^2 \big)
+ \nabla \big( \mathcal{R}_i \mathcal{R}_{i'} \Lambda^N_t ( u^i u^{i'} - b^i b^{i'} ) \big).
\end{align*}

To capture the Alfv\'enic coupling, we define the modified energy functional:
\begin{equation}\label{quanju1}
\begin{aligned}
    E(t)
    :=&
    \big\|
        \sqrt{\mathcal{W}}\Lambda_t^N(\mathbf u,\mathbf b)
    \big\|_{L^2}^2+
    \operatorname{Re}
    \left\langle
        \frac{\mathbf 1_{k\neq0}}{ik\beta}\Lambda_t^N\widehat b_k^2,
        \mathcal{W}_k\Lambda_t^N\widehat u_k^1
    \right\rangle-
    \operatorname{Re}
    \left\langle
        \frac{\mathbf 1_{k\neq0}}{ik\beta}\Lambda_t^N\widehat u_k^2,
        \mathcal{W}_k\Lambda_t^N\widehat b_k^1
    \right\rangle.
\end{aligned}
\end{equation}
We also define the dissipation and the Cauchy--Kovalevskaya (CK) terms generated by the multiplier:
\begin{align*}
\operatorname{Dis}_1(t)
&:=
\sum_{k\in\mathbb Z}\int_{\mathbb R}
\xi^2\mathcal W_k(t,\xi)\Lambda_t^{2N}(k,\xi)
\left(
\nu|\widehat{\mathbf u}_k(t,\xi)|^2
+
\mu|\widehat{\mathbf b}_k(t,\xi)|^2
\right)\,d\xi,
\\
\operatorname{CK}_1(t)
&:=
\kappa^{1/3}
\sum_{k\in\mathbb Z}\int_{\mathbb R}
|k|^{2/3}\mathcal A_k(t,\xi)\Lambda_t^{2N}(k,\xi)
|(\widehat{\mathbf u}_k,\widehat{\mathbf b}_k)(t,\xi)|^2
\,d\xi,
\\
\operatorname{CK}_2(t)
&:=
\sum_{k \neq 0}\int_{\mathbb R}
\delta_k'(\xi)\mathcal A_k(t,\xi)\Lambda_t^{2N}(k,\xi)
|(\widehat{\mathbf u}_k,\widehat{\mathbf b}_k)(t,\xi)|^2
\,d\xi,
\\
\operatorname{CK}_3(t)
&:=
\sum_{k\in\mathbb Z}\int_{\mathbb R}
\mathcal G_k(t,\xi)\mathcal A_k(t,\xi)
\Lambda_t^{2N}(k,\xi)
|(\widehat{\mathbf u}_k,\widehat{\mathbf b}_k)(t,\xi)|^2
\,d\xi.
\end{align*}
For $j=1,2,3$, we decompose
\begin{align*}
\operatorname{CK}_j(t)
=
\operatorname{CK}_{j,\mathbf u}(t)
+
\operatorname{CK}_{j,\mathbf b}(t),
\end{align*}
where
\begin{align*}
\operatorname{CK}_{1,\mathbf u}(t)
&:=
\kappa^{1/3}
\sum_{k\in\mathbb Z}\int_{\mathbb R}
|k|^{2/3}\mathcal A_k(t,\xi)
\Lambda_t^{2N}(k,\xi)
|\widehat{\mathbf u}_k(t,\xi)|^2\,d\xi,
\\
\operatorname{CK}_{2,\mathbf u}(t)
&:=
\sum_{k\neq0}\int_{\mathbb R}
\delta_k'(\xi)\mathcal A_k(t,\xi)
\Lambda_t^{2N}(k,\xi)
|\widehat{\mathbf u}_k(t,\xi)|^2\,d\xi,
\\
\operatorname{CK}_{3,\mathbf u}(t)
&:=
\sum_{k\in\mathbb Z}\int_{\mathbb R}
\mathcal G_k(t,\xi)\mathcal A_k(t,\xi)
\Lambda_t^{2N}(k,\xi)
|\widehat{\mathbf u}_k(t,\xi)|^2\,d\xi.
\end{align*}
The quantities $\operatorname{CK}_{j,\mathbf b}$ are defined
analogously.

We now state the  {\it a priori} energy estimate.

\begin{proposition}\label{quanju3}
Let $N \geq 4$ and  $\vartheta \in (0,1/6)$. For $\mu, \nu \in (0,1]$, define $\kappa := \min\{\mu,\nu\}$. Assume the background magnetic field satisfies
\begin{align*}
    |\beta|>7\pi+1\quad\hbox{and}\quad
	|\beta| \geq \frac{6}{1-6\vartheta}\frac{\mu+\nu}{\sqrt{\mu\nu}}.
\end{align*}
Define the structural constants:
\begin{align*}
	c_D := \frac{1}{6} - \frac{\mu+\nu}{|\beta|\sqrt{\mu\nu}}, \quad
	c_1 := \frac{1}{6\pi} - \frac{7\pi+1}{6\pi|\beta|}, \quad
	c_2 := \frac{1}{6D_\delta} - \frac{12D_\delta+1}{12D_\delta|\beta|}, \quad
	c_3 := 1 - \frac{1}{2|\beta|}.
\end{align*}
The condition on $\beta$  guarantees $c_D \geq \vartheta > 0$, $c_1 > 0$, $c_2 > 0$, and $c_3 > 0$.

Define the energy norms:
\begin{align*}
	X_{\mathbf{u}}(t) &:= \big\|\sqrt{\mathcal{W}}\Lambda_t^N\mathbf{u}(t)\big\|_{L^2},
	&
	X_{\mathbf{b}}(t) &:= \big\|\sqrt{\mathcal{W}}\Lambda_t^N\mathbf{b}(t)\big\|_{L^2},\\
	X_{\mathbf{u},\neq}(t) &:= \big\|\sqrt{\mathcal{W}}\Lambda_t^N\mathbf{u}_{\neq}(t)\big\|_{L^2},
	&
	X_{\mathbf{b},\neq}(t) &:= \big\|\sqrt{\mathcal{W}}\Lambda_t^N\mathbf{b}_{\neq}(t)\big\|_{L^2},\\
	X(t) &:= \big\|\sqrt{\mathcal{W}}\Lambda_t^N(\mathbf{u},\mathbf{b})(t)\big\|_{L^2},
	&
	H_{\neq}(t) &:= \big\||\partial_x|^{1/3}\sqrt{\mathcal{W}}\Lambda_t^N(\mathbf{u}_{\neq},\mathbf{b}_{\neq})(t)\big\|_{L^2},
\end{align*}
and the associated nonlinear remainders:
\begin{align*}
	\mathcal{Z}_1(t) &:= X_{\mathbf{b},\neq}(t)X_{\mathbf{u},\neq}(t)X_{\mathbf{u}}(t), &
	\mathcal{Z}_2(t) &:= X_{\mathbf{b},\neq}^2(t)X_{\mathbf{b}}(t), &
	\mathcal{Z}_3(t) &:= X_{\mathbf{u},\neq}^2(t)X_{\mathbf{b}}(t), \\
	\mathcal{Z}_4(t) &:= \kappa^{-\frac{1}{2}}\operatorname{Dis}_1^{\frac{1}{2}} \operatorname{CK}_2^{\frac{1}{2}}X(t), &
	\mathcal{Z}_5(t) &:= X_{\mathbf{u},\neq}^2(t)X_{\mathbf{u}}(t), &
	\mathcal{Z}_6(t) &:= X_{\mathbf{b},\neq}^2(t)X_{\mathbf{u}}(t),
\end{align*}
together with the remaining nonlinear remainders
\begin{align*}
\mathcal{Z}_7(t)
  :=&
  \kappa^{-1/6}H_{\neq}^2(t)X(t),\nn\\
\mathcal Z_8(t)
:=\;&
\kappa^{-1/3}
X_{\mathbf u}(t)
\operatorname{CK}_1^{1/2}(t)
\left(
\operatorname{CK}_2^{1/2}(t)
+
\operatorname{CK}_3^{1/2}(t)
\right)\nn\\
&+
\kappa^{-5/12}
X(t)
\operatorname{CK}_3^{1/2}(t)
\operatorname{CK}_2^{1/4}(t)
\operatorname{CK}_1^{1/4}(t).
\end{align*}

Then there exists a positive constant $C = C(N, \beta, \vartheta)$, independent of $t$, $\mu$, and $\nu$ such that for all $t \geq 0$,
\begin{equation*}
	\frac{d}{dt}E(t) + c_D\operatorname{Dis}_1(t) + c_1\operatorname{CK}_1(t) + c_2\operatorname{CK}_2(t) + c_3\operatorname{CK}_3(t) \leq C\sum_{i=1}^{8}\mathcal{Z}_i(t).
\end{equation*}

Furthermore, $E(t)$ is coercive: there exist constants $0 < c_E \leq C_E < \infty$, depending only on $\beta$ and the uniform equivalence between $\mathcal{W}$ and $\mathcal{A}$, such that
\begin{equation}\label{quanju4}
	c_E X^2(t) \leq E(t) \leq C_E X^2(t).
\end{equation}
\end{proposition}

\begin{proof}
All implicit constants below may depend on
$N$, $\beta$, and $\vartheta$, but are
independent of $t$, $\mu$, and $\nu$.

\medskip
\noindent
\textbf{Step 1. Coercivity of the energy.}

By the definition of $E(t)$, the two cross terms are supported on
$k\neq0$. Since $|k|^{-1}\leq1$ on this set, the Cauchy--Schwarz
inequality gives
\begin{align}\label{quanju5}
\left|
    \operatorname{Re}
    \left\langle
        \frac{1}{ik\beta}\Lambda_t^N\widehat b_k^2,
        \mathcal{W}_k(t,\xi)\Lambda_t^N\widehat u_k^1
    \right\rangle
\right|
\leq&
\frac{1}{|\beta|}
\big\|
    \sqrt{\mathcal{W}}\Lambda_t^N b_{\neq}^2
\big\|_{L^2}
\big\|
    \sqrt{\mathcal{W}}\Lambda_t^N u_{\neq}^1
\big\|_{L^2}
\nonumber\\
\leq&
\frac{1}{2|\beta|}
\left(
    \big\|
        \sqrt{\mathcal{W}}\Lambda_t^N b_{\neq}^2
    \big\|_{L^2}^2
    +
    \big\|
        \sqrt{\mathcal{W}}\Lambda_t^N u_{\neq}^1
    \big\|_{L^2}^2
\right),
\end{align}
and, similarly,
\begin{align}\label{quanju6}
&\left|
    \operatorname{Re}
    \left\langle
        \frac{1}{ik\beta}\Lambda_t^N\widehat u_k^2,
        \mathcal{W}_k(t,\xi)\Lambda_t^N\widehat b_k^1
    \right\rangle
\right|
\leq
\frac{1}{2|\beta|}
\left(
    \big\|
        \sqrt{\mathcal{W}}\Lambda_t^N u_{\neq}^2
    \big\|_{L^2}^2
    +
    \big\|
        \sqrt{\mathcal{W}}\Lambda_t^N b_{\neq}^1
    \big\|_{L^2}^2
\right).
\end{align}
Adding \eqref{quanju5} and
\eqref{quanju6}, we obtain the  estimate
\begin{equation}\label{quanju7}
\begin{aligned}
&\left|
    \operatorname{Re}
    \left\langle
        \frac{1}{ik\beta}\Lambda_t^N\widehat b_k^2,
        \mathcal{W}_k(t,\xi)\Lambda_t^N\widehat u_k^1
    \right\rangle
\right|
+
\left|
    \operatorname{Re}
    \left\langle
        \frac{1}{ik\beta}\Lambda_t^N\widehat u_k^2,
        \mathcal{W}_k(t,\xi)\Lambda_t^N\widehat b_k^1
    \right\rangle
\right|\\
&\qquad\leq
\frac{1}{2|\beta|}
\big\|
    \sqrt{\mathcal{W}}\Lambda_t^N
    (\mathbf u_{\neq},\mathbf b_{\neq})
\big\|_{L^2}^2
\leq
\frac{1}{2|\beta|}X^2(t).
\end{aligned}
\end{equation}
It therefore follows from \eqref{quanju1} and
\eqref{quanju7} that
\begin{equation}\label{quanju8}
    \left(1-\frac{1}{2|\beta|}\right)X^2(t)
    \leq E(t)
    \leq
    \left(1+\frac{1}{2|\beta|}\right)X^2(t).
\end{equation}
In particular, since the assumption on the background magnetic field
implies
$
    |\beta|>{3}/{2},
$
the lower coefficient in \eqref{quanju8} is
strictly positive. Thus \eqref{quanju4} holds with
\begin{equation*}
    c_E:=1-\frac{1}{2|\beta|}>0,
    \qquad
    C_E:=1+\frac{1}{2|\beta|}.
\end{equation*}

\medskip\noindent
\textbf{Step 2. Weighted energy identity.}

\smallskip

Applying $\Lambda_t^N$ to the velocity equation in
\eqref{quanju2}, taking twice the real part of its $L^2$ inner product
with $\mathcal{W}\Lambda_t^N\mathbf u$, and integrating by parts in
$\xi$, we obtain
\begin{align}\label{quanju10}
&\frac{d}{dt}
 \left\|
   \sqrt{\mathcal{W}}\Lambda_t^N\mathbf u
 \right\|_{L^2}^2
 +2\nu
 \left\|
   \partial_y\sqrt{\mathcal{W}}\Lambda_t^N\mathbf u
 \right\|_{L^2}^2
 +\sum_{k\in\mathbb Z}\int_{\mathbb R}
  \bigl(-\partial_t+k\partial_\xi\bigr)
  \mathcal{W}_k(t,\xi)
  \left|
    \Lambda_t^N(k,\xi)\widehat{\mathbf u}_k(t,\xi)
  \right|^2
  \,d\xi
\notag\\
&\qquad
 +2\operatorname{Re}
  \left\langle
    \Lambda_t^Nu^2,
    \mathcal{W}\Lambda_t^Nu^1
  \right\rangle
 -2\beta\operatorname{Re}
  \left\langle
    \partial_x\Lambda_t^N\mathbf b,
    \mathcal{W}\Lambda_t^N\mathbf u
  \right\rangle
\notag\\
&\quad=
 2\operatorname{Re}
 \left\langle
   \Lambda_t^N\mathcal{NL}_{\mathbf u},
   \mathcal{W}\Lambda_t^N\mathbf u
 \right\rangle .
\end{align}
Here we used
$
  [\Lambda_t^N,\partial_t+y\partial_x]=0
$
and the fact that $\mathcal{W}$ is a real, positive, self-adjoint scalar
Fourier multiplier. In particular, $\mathcal{W}$ commutes with spatial
derivatives and preserves the divergence-free constraint. Hence
\begin{align*}
  \operatorname{div}
  \bigl(\mathcal{W}\Lambda_t^N\mathbf u\bigr)
  =
  \mathcal{W}\Lambda_t^N\operatorname{div}\mathbf u
  =0.
\end{align*}
Consequently, the pressure contribution vanishes:
\begin{align*}
\begin{aligned}
  \left\langle
    \nabla\Lambda_t^NP,
    \mathcal{W}\Lambda_t^N\mathbf u
  \right\rangle
  =
  -\left\langle
    \Lambda_t^NP,
    \operatorname{div}
    \bigl(\mathcal{W}\Lambda_t^N\mathbf u\bigr)
  \right\rangle
=0.
\end{aligned}
\end{align*}

The same calculation for the magnetic equation gives
\begin{align}\label{quanju11}
&\frac{d}{dt}
 \left\|
   \sqrt{\mathcal{W}}\Lambda_t^N\mathbf b
 \right\|_{L^2}^2
 +2\mu
 \left\|
   \partial_y\sqrt{\mathcal{W}}\Lambda_t^N\mathbf b
 \right\|_{L^2}^2
 +\sum_{k\in\mathbb Z}\int_{\mathbb R}
  \bigl(-\partial_t+k\partial_\xi\bigr)
  \mathcal{W}_k(t,\xi)
  \left|
    \Lambda_t^N(k,\xi)\widehat{\mathbf b}_k(t,\xi)
  \right|^2
  \,d\xi
\notag\\
&\quad=
 2\operatorname{Re}
 \left\langle
   \Lambda_t^Nb^2,
   \mathcal{W}\Lambda_t^Nb^1
 \right\rangle
 +2\beta\operatorname{Re}
 \left\langle
   \partial_x\Lambda_t^N\mathbf u,
   \mathcal{W}\Lambda_t^N\mathbf b
 \right\rangle
 +2\operatorname{Re}
 \left\langle
   \Lambda_t^N\mathcal{NL}_{\mathbf b},
   \mathcal{W}\Lambda_t^N\mathbf b
 \right\rangle .
\end{align}
Since $\partial_x$ is skew-adjoint and commutes with
$\mathcal{W}\Lambda_t^N$, we have
\begin{align*}
\begin{aligned}
&\operatorname{Re}
 \left\langle
   \partial_x\Lambda_t^N\mathbf b,
   \mathcal{W}\Lambda_t^N\mathbf u
 \right\rangle
 +\operatorname{Re}
 \left\langle
   \partial_x\Lambda_t^N\mathbf u,
   \mathcal{W}\Lambda_t^N\mathbf b
 \right\rangle
 =0.
\end{aligned}
\end{align*}
Adding \eqref{quanju10} and \eqref{quanju11}, we therefore
obtain
\begin{align}\label{quanju12}
&\frac{d}{dt}
 \left\|
   \sqrt{\mathcal{W}}\Lambda_t^N(\mathbf u,\mathbf b)
 \right\|_{L^2}^2
 +2
 \left\|
   \partial_y\sqrt{\mathcal{W}}\Lambda_t^N
   \bigl(\sqrt{\nu}\,\mathbf u,\sqrt{\mu}\,\mathbf b\bigr)
 \right\|_{L^2}^2
\notag\\
&\quad
 +\sum_{k\in\mathbb Z}\int_{\mathbb R}
  \bigl(-\partial_t+k\partial_\xi\bigr)
  \mathcal{W}_k(t,\xi)
  \left|
    \Lambda_t^N(k,\xi)
    \bigl(\widehat{\mathbf u}_k,\widehat{\mathbf b}_k\bigr)(t,\xi)
  \right|^2
  \,d\xi
\notag\\
&=
 -2\operatorname{Re}
 \left\langle
   \Lambda_t^Nu^2,
   \mathcal{W}\Lambda_t^Nu^1
 \right\rangle
 +2\operatorname{Re}
 \left\langle
   \Lambda_t^Nb^2,
   \mathcal{W}\Lambda_t^Nb^1
 \right\rangle
\notag\\
&\quad
 +2\operatorname{Re}
 \left\langle
   \Lambda_t^N\mathcal{NL}_{\mathbf u},
   \mathcal{W}\Lambda_t^N\mathbf u
 \right\rangle
 +2\operatorname{Re}
 \left\langle
   \Lambda_t^N\mathcal{NL}_{\mathbf b},
   \mathcal{W}\Lambda_t^N\mathbf b
 \right\rangle .
\end{align}

To compensate for the first two linear terms on the right-hand side
of \eqref{quanju12}, we introduce two cross-energy functionals. The
basic energy identity \eqref{quanju12} is summed over all horizontal
Fourier modes $k\in\mathbb Z$. However, the cross-energy functionals
below are restricted to the nonzero horizontal modes, since they
contain the factor $1/(ik\beta)$.

For completeness, we explain why the zero horizontal mode does not
produce any additional contribution. For $k=0$, the
incompressibility conditions give
\[
\xi\widehat{u_0^2}(\xi)=0,
\qquad
\xi\widehat{b_0^2}(\xi)=0.
\]
Since $\widehat{u_0^2},\widehat{b_0^2}\in L^2(\mathbb R)$, it follows
that
\[
\widehat{u_0^2}(\xi)
=
\widehat{b_0^2}(\xi)
=
0
\quad\text{for almost every }\xi\in\mathbb R.
\]
Thus, the zero horizontal mode contributes nothing to the linear
terms that are compensated by the cross-energy functionals. The
restriction to $k\neq0$ is therefore consistent, and the ensuing
cancellation is exact.

In what follows, every Fourier inner product containing the factor
$1/(ik\beta)$ is understood to be restricted to the nonzero
horizontal modes:
\begin{align*}
  \langle f,g\rangle_{\neq}
  :=
  \sum_{k\neq0}\int_{\mathbb R}
  f_k(\xi)\overline{g_k(\xi)}\,d\xi.
\end{align*}
When no confusion is possible, we suppress the subscript $\neq$. The first cross-energy identity is
\begin{align*}
&\frac{d}{dt}
 \operatorname{Re}
 \left\langle
   \frac{1}{ik\beta}
   \Lambda_t^N(k,\xi)\widehat{b_k^2},
   \mathcal{W}_k(t,\xi)\Lambda_t^N(k,\xi)\widehat{u_k^1}
 \right\rangle_{\neq}
\notag\\
&\quad=
 \operatorname{Re}
 \left\langle
   \Lambda_t^Nu^2,
    \mathcal W_k(t,\xi)\Lambda_t^Nu^1
 \right\rangle
 -
 \operatorname{Re}
 \left\langle
   \Lambda_t^Nb^2,
    \mathcal{W}_k(t,\xi)\Lambda_t^Nb^1
 \right\rangle+
 \operatorname{Re}
 \left\langle
   \frac{1}{ik\beta}
   \Lambda_t^N(k,\xi)\widehat{b_k^2},
   \widehat g_2(u^1,b^1)
 \right\rangle_{\neq}\nn\\
 &\qquad
 -
 \operatorname{Re}
 \left\langle
   \frac{1}{ik\beta}
   \Lambda_t^N(k,\xi)\widehat{b_k^2},
   \mathcal{W}_k(t,\xi)\Lambda_t^N(k,\xi)\widehat{u_k^2}
 \right\rangle_{\neq}+
 \operatorname{Re}
 \left\langle
   \frac{1}{ik\beta}\widehat g_1(u^2,b^2),
   \mathcal{W}_k(t,\xi)\Lambda_t^N(k,\xi)\widehat{u_k^1}
 \right\rangle_{\neq}.
\end{align*}
Here
\begin{align*}
\widehat g_1(u^2,b^2)
:=&
 \mathcal{F}{ \bigl(\Lambda_t^N
(-\mathbf u\cdot\nabla b^2)\bigr)}
 +
\mathcal{F}{\bigl(\Lambda_t^N
 (\mathbf b\cdot\nabla u^2)\bigr)} +k\partial_\xi
 \left(
   \Lambda_t^N(k,\xi)\widehat{b_k^2}
 \right)
 -\mu\xi^2
  \Lambda_t^N(k,\xi)\widehat{b_k^2},
\end{align*}
whereas
\begin{align*}
\widehat g_2(u^1,b^1)&
:=
 \mathcal W_k(t,\xi)
 \mathcal{F}{\bigl(\Lambda_t^N
 (-\mathbf u\cdot\nabla u^1)\bigr)}
 +
 \mathcal W_k(t,\xi)
\mathcal{F}{\bigl(\Lambda_t^N
 (\mathbf b\cdot\nabla b^1)\bigr)}
+(\partial_t\mathcal W_k(t,\xi))
 \left(
   \Lambda_t^N(k,\xi)\widehat{u_k^1}
 \right)\notag\\
& +\mathcal W_k(t,\xi) k\partial_\xi
 \left(
   \Lambda_t^N(k,\xi)\widehat{u_k^1}
 \right)
-\nu\xi^2\mathcal W_k(t,\xi)
  \Lambda_t^N(k,\xi)\widehat{u_k^1}+\frac{2k^2}{k^2+\xi^2}\,
  \mathcal W_k(t,\xi)\Lambda_t^N(k,\xi)\widehat{u_k^2}\\
&-ik\,
  \widehat{\mathcal R}_i
  \widehat{\mathcal R}_{i'}\,
  \mathcal W_k(t,\xi)\Lambda_t^N(k,\xi)
  \mathcal{F}{\bigl(
    (u^iu^{i'}-b^ib^{i'})\bigr)
  }.
\end{align*}

The second cross-energy identity is
\begin{align*}
&\frac{d}{dt}
 \left[
 -\operatorname{Re}
 \left\langle
   \frac{1}{ik\beta}
   \Lambda_t^N(k,\xi)\widehat{u_k^2},
   \mathcal W_k(t,\xi)\Lambda_t^N(k,\xi)\widehat{b_k^1}
 \right\rangle_{\neq}
 \right]
\notag\\
&\quad=
 \operatorname{Re}
 \left\langle
   \Lambda_t^Nu^2,
    \mathcal W_k(t,\xi)\Lambda_t^Nu^1
 \right\rangle
 -
 \operatorname{Re}
 \left\langle
   \Lambda_t^Nb^2,
    \mathcal W_k(t,\xi)\Lambda_t^Nb^1
 \right\rangle -
 \operatorname{Re}
 \left\langle
   \frac{1}{ik\beta}
   \Lambda_t^N(k,\xi)\widehat{u_k^2},
   \widehat g_4(u^1,b^1)
 \right\rangle_{\neq}
\notag\\
&\qquad-
 \operatorname{Re}
 \left\langle
   \frac{1}{ik\beta}
   \Lambda_t^N(k,\xi)\widehat{u_k^2},
   \mathcal W_k(t,\xi)\Lambda_t^N(k,\xi)\widehat{b_k^2}
 \right\rangle_{\neq}
 -
 \operatorname{Re}
 \left\langle
   \frac{1}{ik\beta}\widehat g_3(u^2,b^2),
   \mathcal W_k(t,\xi)\Lambda_t^N(k,\xi)\widehat{b_k^1}
 \right\rangle_{\neq}.
\end{align*}
Here
\begin{align*}
\widehat g_3(u^2,b^2)
:=&
 \mathcal{F}{\bigl(\Lambda_t^N
 (-\mathbf u\cdot\nabla u^2)\bigr)}
 +
\mathcal{F}{\bigl(\Lambda_t^N
 (\mathbf b\cdot\nabla b^2)\bigr)}
 -i\xi\,
  \widehat{\mathcal R}_i
  \widehat{\mathcal R}_{i'}\,
  \Lambda_t^N(k,\xi)
  \mathcal{F}{\bigl(
    (u^iu^{i'}-b^ib^{i'})\bigr)
  }
\notag\\
&
 +\frac{2k\xi}{k^2+\xi^2}\,
  \Lambda_t^N(k,\xi)\widehat{u_k^2}
 +k\partial_\xi
  \left(
    \Lambda_t^N(k,\xi)\widehat{u_k^2}
  \right) -\nu\xi^2
  \Lambda_t^N(k,\xi)\widehat{u_k^2},
\end{align*}
and
\begin{align*}
\widehat g_4(u^1,b^1)
:=&
 \mathcal W_k(t,\xi)
 \mathcal{F}{\bigl(\Lambda_t^N
 (-\mathbf u\cdot\nabla b^1)\bigr)}
 +
 \mathcal W_k(t,\xi)
 \mathcal{F}{\bigl(\Lambda_t^N
 (\mathbf b\cdot\nabla u^1)\bigr)}
\notag\\
&
 +(\partial_t\mathcal W_k(t,\xi))
 \left(
   \Lambda_t^N(k,\xi)\widehat{b_k^1}
 \right)
 +\mathcal W_k(t,\xi) k\partial_\xi
 \left(
   \Lambda_t^N(k,\xi)\widehat{b_k^1}
 \right)
 -\mu\xi^2\mathcal W_k(t,\xi)
  \Lambda_t^N(k,\xi)\widehat{b_k^1}.
\end{align*}

In differentiating the cross terms, all sums are over $k\neq0$.
The cancellation of the remaining linear terms is exact:
\begin{align*}
&\operatorname{Re}
 \left\langle
   \frac{1}{ik\beta}
   \Lambda_t^N(k,\xi)\widehat b_k^2,
   \mathcal W_k(t,\xi)\Lambda_t^N(k,\xi)\widehat u_k^2
 \right\rangle_{\neq}
+
\operatorname{Re}
 \left\langle
   \frac{1}{ik\beta}
   \Lambda_t^N(k,\xi)\widehat u_k^2,
   \mathcal W_k(t,\xi)\Lambda_t^N(k,\xi)\widehat b_k^2
 \right\rangle_{\neq}
=0.
\end{align*}
Indeed, the two inner products are complex conjugates up to the purely
imaginary factor $1/(ik\beta)$.

Differentiating $E(t)$, substituting the equations for
$(\mathbf u,\mathbf b)$, using $[\Lambda_t^N,\partial_t+y\partial_x]=0,$ and integrating by parts in $\xi$, we arrive at the following exact
weighted energy identity:
\begin{align*}
&\frac{d}{dt}E(t)
+
\left\langle
  \left(
    2\nu\xi^2\mathcal W_k(t,\xi)
    +(-\partial_t+k\partial_\xi)\mathcal W_k(t,\xi)
  \right)
  \Lambda_t^N\widehat{\mathbf u}_k,
  \Lambda_t^N\widehat{\mathbf u}_k
\right\rangle
\notag\\
&\quad+
\left\langle
  \left(
    2\mu\xi^2\mathcal W_k(t,\xi)
    +(-\partial_t+k\partial_\xi)\mathcal W_k(t,\xi)
  \right)
  \Lambda_t^N\widehat{\mathbf b}_k,
  \Lambda_t^N\widehat{\mathbf b}_k
\right\rangle
=
\sum_{j=1}^3{\mathscr L}_j
+
\sum_{j=1}^4{\mathscr N}_j.
\end{align*}

Here
${\mathscr L}_1$,
${\mathscr L}_2$, and
${\mathscr L}_3$
denote the linear errors arising, respectively, from unequal vertical
dissipation, the transport derivative of the multiplier in the cross
energy, and the linear nonlocal coupling. More precisely,
\begin{align*}
{\mathscr L}_1
:=&
 -(\mu+\nu)\operatorname{Re}
 \left\langle
   \frac{\xi^2}{ik\beta}
   \Lambda_t^N\widehat b_k^2,
   \mathcal W_k(t,\xi)\Lambda_t^N\widehat u_k^1
 \right\rangle_{\neq}
 +(\mu+\nu)\operatorname{Re}
 \left\langle
   \frac{\xi^2}{ik\beta}
   \Lambda_t^N\widehat u_k^2,
   \mathcal W_k(t,\xi)\Lambda_t^N\widehat b_k^1
 \right\rangle_{\neq},
\end{align*}
\begin{align*}
{\mathscr L}_2
:=&
 \operatorname{Re}
 \left\langle
   \frac{1}{ik\beta}
   \Lambda_t^N\widehat b_k^2,
   (\partial_t-k\partial_\xi)\mathcal W_k(t,\xi)\,
   \Lambda_t^N\widehat u_k^1
 \right\rangle_{\neq}
 -
 \operatorname{Re}
 \left\langle
   \frac{1}{ik\beta}
   \Lambda_t^N\widehat u_k^2,
   (\partial_t-k\partial_\xi)\mathcal W_k(t,\xi)\,
   \Lambda_t^N\widehat b_k^1
 \right\rangle_{\neq},
\end{align*}
and
\begin{align}\label{quanju22}
{\mathscr L}_3
:=&
 2\operatorname{Re}
 \left\langle
   \frac{1}{ik\beta}
   \Lambda_t^N\widehat b_k^2,
   \frac{k^2-\xi^2}{k^2+\xi^2}\,
   \mathcal W_k(t,\xi)\Lambda_t^N\widehat u_k^2
 \right\rangle_{\neq}.
\end{align}

The multiplier
$  \frac{k^2-\xi^2}{k^2+\xi^2}
$
in \eqref{quanju22} is the Fourier symbol generated by the
linear nonlocal pressure correction associated with the Couette
background.

The terms
${\mathscr N}_1,\ldots,{\mathscr N}_4$
capture all nonlinear transport, magnetic stretching, and pressure
contributions. They are given explicitly as follows:
\begin{align*}
{\mathscr N}_1
:=&
 2\operatorname{Re}
 \left\langle
   \Lambda_t^N\mathcal{NL}_{\mathbf u},
   \mathcal{W}\Lambda_t^N\mathbf u
 \right\rangle
 +2\operatorname{Re}
 \left\langle
   \Lambda_t^N\mathcal{NL}_{\mathbf b},
   \mathcal{W}\Lambda_t^N\mathbf b
 \right\rangle ,
\end{align*}
\begin{align*}
{\mathscr N}_2
:=&
 \operatorname{Re}
 \left\langle
   \frac{1}{ik\beta}
    \mathcal{F}{\bigl(
     \Lambda_t^N
     (
       -\mathbf u\cdot\nabla b^2
       +\mathbf b\cdot\nabla u^2
     )\bigr)
   },
   \mathcal W_k(t,\xi)\Lambda_t^N\widehat u_k^1
 \right\rangle_{\neq}
\notag\\
&\quad
 +
 \operatorname{Re}
 \left\langle
   \frac{1}{ik\beta}
   \Lambda_t^N\widehat b_k^2,
   \mathcal W_k(t,\xi)
   \mathcal{F}{\bigl(
     \Lambda_t^N
     (
       -\mathbf u\cdot\nabla u^1
       +\mathbf b\cdot\nabla b^1
     )\bigr)
   }
 \right\rangle_{\neq},
\end{align*}
\begin{align*}
{\mathscr N}_3
:=&
 \operatorname{Re}
 \left\langle
   \frac{1}{ik\beta}
   \mathcal{F}{\bigl(
     \Lambda_t^N
     (
       \mathbf u\cdot\nabla u^2
       -\mathbf b\cdot\nabla b^2
     )\bigr)
   },
   \mathcal W_k(t,\xi)\Lambda_t^N\widehat b_k^1
 \right\rangle_{\neq}
\notag\\
&\quad
 +
 \operatorname{Re}
 \left\langle
   \frac{1}{ik\beta}
   \Lambda_t^N\widehat u_k^2,
   \mathcal W_k(t,\xi)
   \mathcal{F}{\bigl(
     \Lambda_t^N
     (
       \mathbf u\cdot\nabla b^1
       -\mathbf b\cdot\nabla u^1
     )\bigr)
   }
 \right\rangle_{\neq},
\end{align*}
and
\begin{align}\label{quanju26}
{\mathscr N}_4
:=&
 \operatorname{Re}
 \left\langle
   \frac{i\xi}{ik\beta}
   \widehat{\mathcal R}_i
   \widehat{\mathcal R}_{i'}
   \mathcal{F}{\bigl(
     \Lambda_t^N
     (u^iu^{i'}-b^ib^{i'})\bigr)
   },
   \mathcal W_k(t,\xi)\Lambda_t^N\widehat b_k^1
 \right\rangle_{\neq}
\notag\\
&\quad
 -
 \operatorname{Re}
 \left\langle
   \frac{1}{ik\beta}
   \Lambda_t^N\widehat b_k^2,
   ik\widehat{\mathcal R}_i
   \widehat{\mathcal R}_{i'}\,
   \mathcal W_k(t,\xi)
   \mathcal{F}{\bigl(
     \Lambda_t^N
     (u^iu^{i'}-b^ib^{i'})
   \bigr)}
 \right\rangle_{\neq}.
\end{align}

\medskip\noindent
\textbf{Step 3. Coercivity supplied by the multiplier.}

We first recall from \eqref{pre8} that
\begin{align}\label{quanju27}
&2\sigma\xi^2\mathcal W_k(t,\xi)+(-\partial_t+k\partial_\xi)\mathcal W_k(t,\xi)\nn\\
&\quad\ge \mathcal{A}_k(t,\xi)
\left[
\frac16\sigma\xi^2
+\frac{1}{6\pi}\kappa^{1/3}|k|^{2/3}
+\frac{1}{6D_\delta}\delta_k'(\xi)
+\mathcal{G}_k(t,\xi)
\right].
\end{align}

We apply \eqref{quanju27} with $\sigma=\nu$ to
$\widehat{\mathbf u}_k$. Multiplying by
$\Lambda_t^{2N}(k,\xi)|\widehat{\mathbf u}_k(\xi)|^2$, integrating
over $\xi$, and summing over $k\in\mathbb Z$, we obtain
\begin{align}\label{quanju28}
&\left\langle
    \left(
        2\nu\xi^2\mathcal{W}_k(t,\xi)
        +(-\partial_t+k\partial_\xi)\mathcal{W}_k(t,\xi)
    \right)
    \Lambda_t^N\widehat{\mathbf u}_k,
    \Lambda_t^N\widehat{\mathbf u}_k
\right\rangle
\nonumber\\
&\quad\geq
\sum_{k\in\mathbb Z}\int_{\mathbb R}
    \mathcal{A}_k(t,\xi)\Lambda_t^{2N}(k,\xi)
    \left[
\frac16\nu\xi^2
+\frac{1}{6\pi}\kappa^{1/3}|k|^{2/3}
+\frac{1}{6D_\delta}\delta_k'(\xi)
+\mathcal{G}_k(t,\xi)
\right]
    |\widehat{\mathbf u}_k|^2
\,d\xi.
\end{align}
Similarly, applying \eqref{quanju27} with
$\sigma=\mu$ to $\widehat{\mathbf b}_k$ yields
\begin{align}\label{quanju29}
&\left\langle
    \left(
        2\mu\xi^2\mathcal{W}_k(t,\xi)
        +(-\partial_t+k\partial_\xi)\mathcal{W}_k(t,\xi)
    \right)
    \Lambda_t^N\widehat{\mathbf b}_k,
    \Lambda_t^N\widehat{\mathbf b}_k
\right\rangle
\nonumber\\
&\quad\geq
\sum_{k\in\mathbb Z}\int_{\mathbb R}
    \mathcal{A}_k(t,\xi)\Lambda_t^{2N}(k,\xi)
    \left[
\frac16\mu\xi^2
+\frac{1}{6\pi}\kappa^{1/3}|k|^{2/3}
+\frac{1}{6D_\delta}\delta_k'(\xi)
+\mathcal{G}_k(t,\xi)
\right]
    |\widehat{\mathbf b}_k|^2
\,d\xi.
\end{align}
Adding \eqref{quanju28} and
\eqref{quanju29}, and using the definitions of
$\operatorname{Dis}_1$, $\operatorname{CK}_1$,
$\operatorname{CK}_2$, and $\operatorname{CK}_3$, gives
\begin{align}\label{quanju30}
&\left\langle
    \left(
        2\nu\xi^2\mathcal W_k(t,\xi)
        +(-\partial_t+k\partial_\xi)\mathcal W_k(t,\xi)
    \right)
    \Lambda_t^N\widehat{\mathbf u}_k,
    \Lambda_t^N\widehat{\mathbf u}_k
\right\rangle
\nn\\
&\qquad+
\left\langle
    \left(
        2\mu\xi^2\mathcal W_k(t,\xi)
        +(-\partial_t+k\partial_\xi)\mathcal W_k(t,\xi)
    \right)
    \Lambda_t^N\widehat{\mathbf b}_k,
    \Lambda_t^N\widehat{\mathbf b}_k
\right\rangle
\nonumber\\
&\quad\geq\frac{1}{6}
\operatorname{Dis}_1(t)
+\frac{1}{6\pi}\operatorname{CK}_1(t)
+\frac{1}{6D_\delta}\operatorname{CK}_2(t)
+\operatorname{CK}_3(t).
\end{align}
\medskip\noindent
\textbf{Step 4. Linear error estimates for ${\mathscr L}_1$, ${\mathscr L}_2$, and ${\mathscr L}_3$.}

We first estimate the unequal-dissipation error ${\mathscr L}_1$.
For $k\neq0$, incompressibility gives
\begin{align}
 |\widehat u_k^2|
 &\leq
 \frac{|k|}{(k^2+\xi^2)^{1/2}}
 |\widehat{\mathbf u}_k|,
&
 |\widehat u_k^1|
 &\leq
 \frac{|\xi|}{(k^2+\xi^2)^{1/2}}
 |\widehat{\mathbf u}_k|,
\label{quanju31}
\\
 |\widehat b_k^2|
 &\leq
 \frac{|k|}{(k^2+\xi^2)^{1/2}}
 |\widehat{\mathbf b}_k|,
&
 |\widehat b_k^1|
 &\leq
 \frac{|\xi|}{(k^2+\xi^2)^{1/2}}
 |\widehat{\mathbf b}_k|.
\label{quanju32}
\end{align}
Using \eqref{quanju31}--\eqref{quanju32}, we obtain
\begin{align*}
\left|
    \left\langle
        \frac{\xi^2}{ik}
        \Lambda_t^N\widehat b_k^2,
        \mathcal W_k(t,\xi)\Lambda_t^N\widehat u_k^1
    \right\rangle_{\neq}
\right|
&\leq
\sum_{k\neq0}\int_{\mathbb R}
    \mathcal W_k(t,\xi)\xi^2\Lambda_t^{2N}
    |\widehat{\mathbf b}_k|
    |\widehat{\mathbf u}_k|
\,d\xi,
\\
\left|
    \left\langle
        \frac{\xi^2}{ik}
        \Lambda_t^N\widehat u_k^2,
        \mathcal W_k(t,\xi)\Lambda_t^N\widehat b_k^1
    \right\rangle_{\neq}
\right|
&\leq
\sum_{k\neq0}\int_{\mathbb R}
    \mathcal W_k(t,\xi)\xi^2\Lambda_t^{2N}
    |\widehat{\mathbf u}_k|
    |\widehat{\mathbf b}_k|
\,d\xi.
\end{align*}
Consequently,
\begin{align*}
|{\mathscr L}_1|
&\leq
\frac{2(\mu+\nu)}{|\beta|}
\sum_{k\neq0}\int_{\mathbb R}
    \mathcal W_k(t,\xi)\xi^2\Lambda_t^{2N}
    |\widehat{\mathbf u}_k|
    |\widehat{\mathbf b}_k|
\,d\xi.
\end{align*}
By the weighted Young inequality,
\begin{align*}
2\sqrt{\mu\nu}\,
|\widehat{\mathbf u}_k|
|\widehat{\mathbf b}_k|
\leq
\nu|\widehat{\mathbf u}_k|^2
+
\mu|\widehat{\mathbf b}_k|^2,
\end{align*}
and hence
\begin{align}\label{quanju37}
|{\mathscr L}_1|
&\leq
\frac{\mu+\nu}{|\beta|\sqrt{\mu\nu}}
\sum_{k\neq0}\int_{\mathbb R}
    \mathcal W_k(t,\xi)\xi^2\Lambda_t^{2N}
    \left(
        \nu|\widehat{\mathbf u}_k|^2
        +
        \mu|\widehat{\mathbf b}_k|^2
    \right)
\,d\xi
\nn\\
&\leq
\frac{\mu+\nu}{|\beta|\sqrt{\mu\nu}}
\operatorname{Dis}_1.
\end{align}
We now bound the linear error ${\mathscr L}_2$ arising from the material derivative of the multiplier within the cross-energy functional:
\begin{align*}
{\mathscr L}_2
=&
\operatorname{Re}
\left\langle
    \frac{1}{ik\beta}\Lambda_t^N\widehat b_k^2,
    (\partial_t-k\partial_\xi)\mathcal W_k(t,\xi)
    \Lambda_t^N\widehat u_k^1
\right\rangle_{\neq}
-
\operatorname{Re}
\left\langle
    \frac{1}{ik\beta}\Lambda_t^N\widehat u_k^2,
    (\partial_t-k\partial_\xi)\mathcal W_k(t,\xi)
    \Lambda_t^N\widehat b_k^1
\right\rangle_{\neq}.
\end{align*}

Recall the factorization $\mathcal W_k(t,\xi) = \mathcal A_k(t,\xi) \bigl(\mathcal M_k^{(1)} + \mathcal M_k^{(2)} + \mathcal M_k^{(3)} + \frac23\bigr)$. Using the uniform bound \eqref{chengzijie}, we readily obtain the pointwise estimate for $k\neq 0$:
\begin{align}\label{transport-W-bound}
\bigl|
(\partial_t-k\partial_\xi)\mathcal W_k(t,\xi)
\bigr|
\leq
\mathcal A_k(t,\xi)
\left(
\frac{1}{32}\kappa^{1/3}
+
\frac{1}{3\pi}\kappa^{1/3}|k|^{2/3}
+
\frac{1}{6D_\delta}\delta_k'(\xi)
+
\mathcal G_k(t,\xi)
\right).
\end{align}

Furthermore, it follows from \eqref{quanju31} and \eqref{quanju32} that
\begin{align*}
\frac{1}{|k|}
\left(
|\widehat b_k^2|\,|\widehat u_k^1|
+
|\widehat u_k^2|\,|\widehat b_k^1|
\right)
\leq
\frac{2|\xi|}{k^2+\xi^2}
|\widehat{\mathbf u}_k|
|\widehat{\mathbf b}_k|
\leq
\frac{1}{2|k|}
\left(
|\widehat{\mathbf u}_k|^2
+
|\widehat{\mathbf b}_k|^2
\right).
\end{align*}
Substituting this pointwise  bound into the definition of ${\mathscr L}_2$ yields
\begin{align*}
|{\mathscr L}_2|
\leq
\frac{1}{2|\beta|}
\sum_{k\neq0}\int_{\mathbb R}
\frac{
\bigl|(\partial_t-k\partial_\xi)\mathcal W_k(t,\xi)\bigr|
}{|k|}
\Lambda_t^{2N}
\left(
|\widehat{\mathbf u}_k|^2
+
|\widehat{\mathbf b}_k|^2
\right)
\,d\xi.
\end{align*}

Inserting \eqref{transport-W-bound} and noting that $|k|\geq1$ on the non-zero modes, we obtain
 \begin{align*}
|{\mathscr L}_2|
\leq&
\left(
\frac{1}{64|\beta|}
+
\frac{1}{6\pi|\beta|}
\right)
\sum_{k\neq0}\int_{\mathbb R}
\mathcal A_k(t,\xi)\kappa^{1/3}|k|^{2/3}
\Lambda_t^{2N}
\left(
|\widehat{\mathbf u}_k|^2
+
|\widehat{\mathbf b}_k|^2
\right)
\,d\xi
\\
&+
\frac{1}{12D_\delta|\beta|}
\sum_{k\neq0}\int_{\mathbb R}
\mathcal A_k(t,\xi)\delta_k'(\xi)
\Lambda_t^{2N}
\left(
|\widehat{\mathbf u}_k|^2
+
|\widehat{\mathbf b}_k|^2
\right)
\,d\xi
\\
&+
\frac{1}{2|\beta|}
\sum_{k\neq0}\int_{\mathbb R}
\mathcal A_k(t,\xi)\mathcal G_k(t,\xi)
\Lambda_t^{2N}
\left(
|\widehat{\mathbf u}_k|^2
+
|\widehat{\mathbf b}_k|^2
\right)
\,d\xi.
\end{align*}
Using $1/64\leq1/6$, we have
\begin{align*}
\frac{1}{64|\beta|}
+
\frac{1}{6\pi|\beta|}
\leq
\frac{\pi+1}{6\pi|\beta|}.
\end{align*}
Consequently,
\begin{align}\label{quanju38}
|{\mathscr L}_2(t)|
\leq
\frac{\pi+1}{6\pi|\beta|}
\operatorname{CK}_1(t)
+
\frac{1}{12D_\delta|\beta|}
\operatorname{CK}_2(t)
+
\frac{1}{2|\beta|}
\operatorname{CK}_3(t).
\end{align}
Finally, we consider
\begin{align*}
    {\mathscr L}_3
    =
    2\operatorname{Re}
    \left\langle
        \frac{1}{ik\beta}\Lambda_t^N\widehat b_k^2,
        \frac{k^2-\xi^2}{k^2+\xi^2}
        \mathcal{W}_k(t,\xi)\Lambda_t^N\widehat u_k^2
    \right\rangle ,
\end{align*}
where the sum is again restricted to $k\neq0$. Since $\left|\frac{k^2-\xi^2}{k^2+\xi^2}\right|\leq1,$
the component estimates \eqref{quanju31}, \eqref{quanju32} and
Cauchy--Schwarz imply
\begin{align}\label{quanju39a}
\left|\mathscr L_3\right|
\leq&
\frac1{|\beta|}
\sum_{k\neq0}\int_{\mathbb R}
\mathcal W_k(t,\xi)\Lambda_t^{2N}(k,\xi)
\left(
|\widehat u_k^2(\xi)|^2
+
|\widehat b_k^2(\xi)|^2
\right)\,d\xi
\notag\\
=&
\frac1{|\beta|}
\sum_{k\neq0}\int_{\mathbb R}
(1-\chi_k)\mathcal W_k(t,\xi)\Lambda_t^{2N}
\left(
|\widehat u_k^2|^2+|\widehat b_k^2|^2
\right)\,d\xi\nn\\
&+\frac1{|\beta|}
\sum_{k\neq0}\int_{\mathbb R}
\chi_k\mathcal W_k(t,\xi)\Lambda_t^{2N}
\left(
|\widehat u_k^2|^2+|\widehat b_k^2|^2
\right)\,d\xi.
\end{align}
On the support of $1-\chi_k$, we have
$
\kappa^{1/2}|k|\geq1,
$
and consequently
$
\kappa^{1/3}|k|^{2/3}\geq1.
$
Since $\mathcal W_k(t,\xi)\leq\mathcal A_k(t,\xi)$, it follows that
\begin{align}\label{gaopin}
&\frac1{|\beta|}
\sum_{k\neq0}\int_{\mathbb R}
(1-\chi_k)\mathcal W_k(t,\xi)\Lambda_t^{2N}
\left(
|\widehat u_k^2|^2+|\widehat b_k^2|^2
\right)\,d\xi\nn\\
&\quad\leq
\frac1{|\beta|}
\sum_{k\neq0}\int_{\mathbb R}
\mathcal A_k(t,\xi)\Lambda_t^{2N}
\left(
|\widehat u_k^2|^2+|\widehat b_k^2|^2
\right)\,d\xi
\nn\\
&\quad\leq
\frac1{|\beta|}
\kappa^{1/3}
\sum_{k\neq0}\int_{\mathbb R}
|k|^{2/3}\mathcal A_k(t,\xi)\Lambda_t^{2N}
\left(
|\widehat{\mathbf u}_k|^2
+
|\widehat{\mathbf b}_k|^2
\right)\,d\xi
\nn\\
&\quad=
\frac1{|\beta|}\operatorname{CK}_1(t).
\end{align}

It remains to estimate the last term in \eqref{quanju39a}. For $k\neq0$, the
divergence-free conditions give
\[
k\widehat u_k^1+\xi\widehat u_k^2=0,
\qquad
k\widehat b_k^1+\xi\widehat b_k^2=0.
\]
Therefore,
\[
|\widehat u_k^2|^2
=
\frac{k^2}{k^2+\xi^2}
|\widehat{\mathbf u}_k|^2
=
\left\langle\frac{\xi}{k}\right\rangle^{-2}
|\widehat{\mathbf u}_k|^2,
\]
and similarly
\[
|\widehat b_k^2|^2
=
\left\langle\frac{\xi}{k}\right\rangle^{-2}
|\widehat{\mathbf b}_k|^2.
\]
The elementary inequality
\[
\bigl[\log(1+r)\bigr]^2\leq r,
\qquad r\geq1,
\]
implies
\[
\frac1{r^2}
\leq
\frac1{r[\log(1+r)]^2},
\qquad r\geq1.
\]
Taking
$
r=\left\langle\frac{\xi}{k}\right\rangle\geq1,
$
we obtain
\[
|\widehat u_k^2|^2
\leq
\delta'\left(\frac{\xi}{k}\right)
|\widehat{\mathbf u}_k|^2,
\qquad
|\widehat b_k^2|^2
\leq
\delta'\left(\frac{\xi}{k}\right)
|\widehat{\mathbf b}_k|^2.
\]
Using again $\mathcal W_k(t,\xi)\leq\mathcal A_k(t,\xi)$, we conclude that
\begin{align}\label{dipin}
\frac1{|\beta|}
\sum_{k\neq0}\int_{\mathbb R}
\chi_k\mathcal W_k(t,\xi)\Lambda_t^{2N}
\left(
|\widehat u_k^2|^2+|\widehat b_k^2|^2
\right)\,d\xi
&\leq
\frac1{|\beta|}
\sum_{k\neq0}\int_{\mathbb R}
\chi_k\delta'\left(\frac{\xi}{k}\right)
\mathcal A_k(t,\xi)\Lambda_t^{2N}
\left(
|\widehat{\mathbf u}_k|^2
+
|\widehat{\mathbf b}_k|^2
\right)\,d\xi
\nn\\
&=
\frac1{|\beta|}\operatorname{CK}_2(t).
\end{align}
Combining the estimates \eqref{gaopin} and \eqref{dipin} gives
\begin{equation}\label{quanju39}
\left|\mathscr L_3(t)\right|
\leq
\frac1{|\beta|}
\left(
\operatorname{CK}_1(t)+\operatorname{CK}_2(t)
\right).
\end{equation}

Inserting \eqref{quanju37}, \eqref{quanju38}, and
\eqref{quanju39} into the exact energy identity and using
\eqref{quanju30}, we obtain
\begin{equation}\label{quanju40}
    \frac{d}{dt}E(t)
    +c_D\operatorname{Dis}_1
    +c_1\operatorname{CK}_1
    +c_2\operatorname{CK}_2
    +c_3\operatorname{CK}_3
    \leq
    \sum_{j=1}^4|{\mathscr N}_j|.
\end{equation}
This completes the absorption of all the linear error terms.

\medskip\noindent
\textbf{Step 5. Nonlinear terms ${\mathscr N}_2$,
${\mathscr N}_3$, and ${\mathscr N}_4$ from the cross energy.}

 All
terms in ${\mathscr N}_2$,
${\mathscr N}_3$, and ${\mathscr N}_4$ have
nonzero output frequency $k\neq0$. Thus $\partial_x^{-1}$ is
uniformly bounded on their Fourier support.

We begin with ${\mathscr N}_2$ and decompose it into
\begin{align*}
{\mathscr N}_2
=
\sum_{j=1}^{4}{\mathscr N}_{2,j},
\end{align*}
where
\begin{align*}
{\mathscr N}_{2,1}
&:=
-\operatorname{Re}
\left\langle
\frac{1}{ik\beta}
\Lambda_t^N\widehat{\mathbf u\cdot\nabla b^2}_k,
\mathcal W_k(t,\xi)\Lambda_t^N\widehat u_k^1
\right\rangle_{\neq},\nn\\
{\mathscr N}_{2,2}
&:=
\operatorname{Re}
\left\langle
\frac{1}{ik\beta}
\Lambda_t^N\widehat{\mathbf b\cdot\nabla u^2}_k,
\mathcal W_k(t,\xi)\Lambda_t^N\widehat u_k^1
\right\rangle_{\neq},
\\
{\mathscr N}_{2,3}
&:=
-\operatorname{Re}
\left\langle
\frac{1}{ik\beta}
\Lambda_t^N\widehat b_k^2,
\mathcal W_k(t,\xi)\Lambda_t^N
\widehat{\mathbf u\cdot\nabla u^1}_k
\right\rangle_{\neq},\nn\\
{\mathscr N}_{2,4}
&:=
\operatorname{Re}
\left\langle
\frac{1}{ik\beta}
\Lambda_t^N\widehat b_k^2,
\mathcal W_k(t,\xi)\Lambda_t^N
\widehat{\mathbf b\cdot\nabla b^1}_k
\right\rangle_{\neq}.
\end{align*}

We first treat the terms in which the derivative falls on a vertical
component.  For $(\mathbf v,\mathbf w)=(\mathbf u,\mathbf b)$ or
$(\mathbf b,\mathbf u)$, incompressibility of $\mathbf w$ gives
\begin{align*}
\begin{aligned}
&\left[
(k-\ell)\widehat v_\ell^1(\eta)
+
(\xi-\eta)\widehat v_\ell^2(\eta)
\right]
\widehat w_{k-\ell}^2(\xi-\eta)
\\
&\quad
=
(k-\ell)
\left[
\widehat v_\ell^1(\eta)\widehat w_{k-\ell}^2(\xi-\eta)
-
\widehat v_\ell^2(\eta)\widehat w_{k-\ell}^1(\xi-\eta)
\right].
\end{aligned}
\end{align*}
Consequently,
\begin{align*}
\begin{aligned}
|{\mathscr N}_{2,1}|
&\lesssim
\sum_{\substack{k\neq0\\ \ell\in\mathbb Z}}
\int_{\mathbb R^2}
\frac{|k-\ell|}{|k|}
\mathcal W_k(t,\xi)\Lambda_t^{2N}(k,\xi)
|\widehat{\mathbf u}_\ell(\eta)|
|\widehat{\mathbf b}_{k-\ell}(\xi-\eta)|
|\widehat{\mathbf u}_k(\xi)|
\,d\eta\,d\xi,                                      \\
|{\mathscr N}_{2,2}|
&\lesssim
\sum_{\substack{k\neq0\\ \ell\in\mathbb Z}}
\int_{\mathbb R^2}
\frac{|k-\ell|}{|k|}
\mathcal W_k(t,\xi)\Lambda_t^{2N}(k,\xi)
|\widehat{\mathbf b}_\ell(\eta)|
|\widehat{\mathbf u}_{k-\ell}(\xi-\eta)|
|\widehat{\mathbf u}_k(\xi)|
\,d\eta\,d\xi .
\end{aligned}
\end{align*}
The mode $k-\ell=0$ vanishes since
$\widehat b_0^2=\widehat u_0^2=0$.  We divide the remaining
frequencies into
\begin{align*}
\frac12|k-\ell|\le |k|\le2|k-\ell|,
\qquad
2|k-\ell|\le |k|,
\qquad
2|k|\le |k-\ell|.
\end{align*}
In the first two regions, $|k-\ell|/|k|\lesssim1$; in the last one,
$|\ell|\approx|k-\ell|$, so the horizontal derivative is assigned to
the comparable input frequency.  Using $\mathcal W_k(t,\xi)\approx\mathcal A_k(t,\xi)$,
the submultiplicativity of $\mathcal A_k(t,\xi)$, and Lemma~\ref{leyoung}, we obtain
\begin{align*}
|{\mathscr N}_{2,1}|
\lesssim
X_{\mathbf u}X_{\mathbf b,\neq}X_{\mathbf u,\neq}
=\mathcal Z_1,
\qquad
|{\mathscr N}_{2,2}|
\lesssim
X_{\mathbf b}X_{\mathbf u,\neq}^2
=\mathcal Z_3.
\end{align*}

For the remaining two terms, incompressibility of the advecting
field yields, for $\mathbf v=\mathbf u$ or $\mathbf b$,
\begin{align*}
(k-\ell)\widehat v_\ell^1(\eta)
+
(\xi-\eta)\widehat v_\ell^2(\eta)
=
k\widehat v_\ell^1(\eta)
+
\xi\widehat v_\ell^2(\eta).
\end{align*}
After division by $k$, the factor $\xi/k$ is transferred to the
output magnetic component and removed by
\begin{align*}
\frac{\xi}{k}\widehat b_k^2(\xi)=-\widehat b_k^1(\xi).
\end{align*}
It follows that
\begin{align*}
|{\mathscr N}_{2,3}|
&\lesssim
\sum_{\substack{k\neq0\\ \ell\in\mathbb Z}}
\int_{\mathbb R^2}
\mathcal W_k(t,\xi)\Lambda_t^{2N}(k,\xi)
|\widehat{\mathbf u}_\ell(\eta)|
|\widehat{\mathbf u}_{k-\ell}(\xi-\eta)|
|\widehat{\mathbf b}_k(\xi)|
\,d\eta\,d\xi,                                      \\
|{\mathscr N}_{2,4}|
&\lesssim
\sum_{\substack{k\neq0\\ \ell\in\mathbb Z}}
\int_{\mathbb R^2}
\mathcal W_k(t,\xi)\Lambda_t^{2N}(k,\xi)
|\widehat{\mathbf b}_\ell(\eta)|
|\widehat{\mathbf b}_{k-\ell}(\xi-\eta)|
|\widehat{\mathbf b}_k(\xi)|
\,d\eta\,d\xi .
\end{align*}
Since $k\neq0$, at least one of the two input horizontal modes is
nonzero.  Lemma~\ref{leyoung} therefore gives
\begin{align*}
|{\mathscr N}_{2,3}|
\lesssim
X_{\mathbf u}X_{\mathbf u,\neq}X_{\mathbf b,\neq}
=\mathcal Z_1,
\qquad
|{\mathscr N}_{2,4}|
\lesssim
X_{\mathbf b}X_{\mathbf b,\neq}^2
=\mathcal Z_2.
\end{align*}
Combining the preceding estimates, we conclude that
\begin{align}\label{quanju47}
|{\mathscr N}_2|
\lesssim
\mathcal Z_1+\mathcal Z_2+\mathcal Z_3.
\end{align}

We next estimate ${\mathscr N}_3$.  Decompose
\begin{align*}
{\mathscr N}_3
=
\sum_{j=1}^4{\mathscr N}_{3,j}
\end{align*}
according to the nonlinearities
$\mathbf u\cdot\nabla u^2$, $-\mathbf b\cdot\nabla b^2$,
$\mathbf u\cdot\nabla b^1$, and $-\mathbf b\cdot\nabla u^1$.

For the first two terms, let $(\mathbf v,\mathbf w)$ be respectively
$(\mathbf u,\mathbf u)$ and $(\mathbf b,\mathbf b)$.  By
incompressibility of $\mathbf w$,
\begin{align*}
\begin{aligned}
&\left[
(k-\ell)\widehat v_\ell^1(\eta)
+
(\xi-\eta)\widehat v_\ell^2(\eta)
\right]
\widehat w_{k-\ell}^2(\xi-\eta)
\\
&\quad
=
(k-\ell)
\left[
\widehat v_\ell^1(\eta)\widehat w_{k-\ell}^2(\xi-\eta)
-
\widehat v_\ell^2(\eta)\widehat w_{k-\ell}^1(\xi-\eta)
\right].
\end{aligned}
\end{align*}
Hence
\begin{align*}
\begin{aligned}
|{\mathscr N}_{3,1}|
&\lesssim
\sum_{\substack{k\neq0\\ \ell\in\mathbb Z}}
\int_{\mathbb R^2}
\frac{|k-\ell|}{|k|}
\mathcal W_k(t,\xi)\Lambda_t^{2N}(k,\xi)
|\widehat{\mathbf u}_\ell(\eta)|
|\widehat{\mathbf u}_{k-\ell}(\xi-\eta)|
|\widehat{\mathbf b}_k(\xi)|
\,d\eta\,d\xi,
\\
|{\mathscr N}_{3,2}|
&\lesssim
\sum_{\substack{k\neq0\\ \ell\in\mathbb Z}}
\int_{\mathbb R^2}
\frac{|k-\ell|}{|k|}
\mathcal W_k(t,\xi)\Lambda_t^{2N}(k,\xi)
|\widehat{\mathbf b}_\ell(\eta)|
|\widehat{\mathbf b}_{k-\ell}(\xi-\eta)|
|\widehat{\mathbf b}_k(\xi)|
\,d\eta\,d\xi.
\end{aligned}
\end{align*}
The mode $k-\ell=0$ vanishes because
$\widehat u_0^2=\widehat b_0^2=0$.  Splitting into
\begin{align*}
\frac12|k-\ell|\le |k|\le2|k-\ell|,
\qquad
2|k-\ell|\le |k|,
\qquad
2|k|\le |k-\ell|,
\end{align*}
the ratio $|k-\ell|/|k|$ is uniformly bounded in the first two regions, while
$|\ell|\approx|k-\ell|$ in the last one.  Lemma~\ref{leyoung} therefore gives
\begin{align*}
|{\mathscr N}_{3,1}|
\lesssim
X_{\mathbf u}X_{\mathbf u,\neq}X_{\mathbf b,\neq}
=\mathcal Z_1,
\qquad
|{\mathscr N}_{3,2}|
\lesssim
X_{\mathbf b}X_{\mathbf b,\neq}^2
=\mathcal Z_2.
\end{align*}

For the last two terms, let $(\mathbf v,\mathbf w)$ be respectively
$(\mathbf u,\mathbf b)$ and $(\mathbf b,\mathbf u)$.  Incompressibility
of the advecting field gives
\begin{align*}
(k-\ell)\widehat v_\ell^1(\eta)
+
(\xi-\eta)\widehat v_\ell^2(\eta)
=
k\widehat v_\ell^1(\eta)+\xi\widehat v_\ell^2(\eta).
\end{align*}
After division by $k$, the factor $\xi/k$ is transferred to the
output velocity component and removed by
\begin{align*}
\frac{\xi}{k}\widehat u_k^2(\xi)=-\widehat u_k^1(\xi).
\end{align*}
Thus
\begin{align*}
\begin{aligned}
|{\mathscr N}_{3,3}|
&\lesssim
\sum_{\substack{k\neq0\\ \ell\in\mathbb Z}}
\int_{\mathbb R^2}
\mathcal W_k(t,\xi)\Lambda_t^{2N}(k,\xi)
|\widehat{\mathbf u}_\ell(\eta)|
|\widehat{\mathbf b}_{k-\ell}(\xi-\eta)|
|\widehat{\mathbf u}_k(\xi)|
\,d\eta\,d\xi,
\\
|{\mathscr N}_{3,4}|
&\lesssim
\sum_{\substack{k\neq0\\ \ell\in\mathbb Z}}
\int_{\mathbb R^2}
\mathcal W_k(t,\xi)\Lambda_t^{2N}(k,\xi)
|\widehat{\mathbf b}_\ell(\eta)|
|\widehat{\mathbf u}_{k-\ell}(\xi-\eta)|
|\widehat{\mathbf u}_k(\xi)|
\,d\eta\,d\xi.
\end{aligned}
\end{align*}
Since $k\neq0$, at least one input mode is nonzero.  Separating
$\ell=0$, $k-\ell=0$, and $\ell(k-\ell)\neq0$, and applying
Lemma~\ref{leyoung}, we obtain
\begin{align*}
|{\mathscr N}_{3,3}|
+
|{\mathscr N}_{3,4}|
\lesssim
X_{\mathbf b,\neq}X_{\mathbf u,\neq}X_{\mathbf u}
+
X_{\mathbf u,\neq}^2X_{\mathbf b}
=
\mathcal Z_1+\mathcal Z_3.
\end{align*}
Combining the four estimates yields
\begin{align}\label{quanju48}
 \bigl|{\mathscr N}_3\bigr|
 \lesssim
 \mathcal{Z}_1+\mathcal{Z}_2+\mathcal{Z}_3.
\end{align}
We finally estimate the pressure contribution
${\mathscr N}_4$.  Since the Riesz transforms are
zero-order Fourier multipliers, the second term satisfies
\begin{align}
\left|
\operatorname{Re}
\left\langle
\frac{1}{ik\beta}\Lambda_t^N\widehat b_k^2,
ik\widehat{\mathcal R_i}\widehat{\mathcal R_{i'}}
\mathcal W_k(t,\xi)\Lambda_t^N
\mathcal{F}{(u^iu^{i'}-b^ib^{i'})}_k
\right\rangle_{\neq}
\right|
\lesssim&
\left(
X_{\mathbf b,\neq}X_{\mathbf u,\neq}X_{\mathbf u}
+
X_{\mathbf b,\neq}^2X_{\mathbf b}
\right)\nn\\
=&(\mathcal Z_1+\mathcal Z_2),
\end{align}
where we used that the output mode is nonzero, and hence at least one
input mode in each quadratic product is nonzero.

For the first term, if $(i,i')\neq(2,2)$, its Fourier symbol is
uniformly bounded, since
\begin{align*}
\left|
\frac{\xi}{k}
\widehat{\mathcal R_i}(k,\xi)
\widehat{\mathcal R_{i'}}(k,\xi)
\right|
\le1,
\qquad k\neq0,\quad (i,i')\neq(2,2).
\end{align*}
Thus Lemma~\ref{leyoung} gives the same bound
\begin{align*}
\left|
{\mathscr N}_{4,1}
[\mathcal R_i\mathcal R_{i'}]
\right|
\lesssim
\frac1{|\beta|}(\mathcal Z_1+\mathcal Z_2),
\qquad (i,i')\neq(2,2).
\end{align*}

It remains to consider the $\mathcal R_2\mathcal R_2$ component.
Using
\begin{align*}
\left|\widehat{\mathcal R_2}^2(k,\xi)\right|
=
\frac{\xi^2}{k^2+\xi^2}\le1
\quad\text{and}\quad
\xi=\eta+(\xi-\eta),
\end{align*}
we distribute the remaining factor $\xi/k$ between the two inputs.
For the velocity product, incompressibility yields
\begin{align*}
\eta\widehat u_\ell^2(\eta)
=-\ell\widehat u_\ell^1(\eta),
\qquad
(\xi-\eta)\widehat u_{k-\ell}^2(\xi-\eta)
=-(k-\ell)\widehat u_{k-\ell}^1(\xi-\eta),
\end{align*}
and the analogous identities hold for $\mathbf b$.  The resulting
symbols contain only the ratios $\ell/k$ and $(k-\ell)/k$.  We split
the frequency space into
\begin{align*}
\frac12|k-\ell|\le |k|\le2|k-\ell|,
\qquad
2|k-\ell|\le |k|,
\qquad
2|k|\le |k-\ell|.
\end{align*}
The ratios are uniformly bounded in the first two regions; in the
last region, $|\ell|\approx|k-\ell|$, and the remaining horizontal
derivative is assigned to the comparable input frequency.  Lemma~\ref{leyoung}
therefore gives
\begin{align*}
\left|
{\mathscr N}_{4,1}
[\mathcal R_2\mathcal R_2]
\right|
\lesssim
\frac1{|\beta|}(\mathcal Z_1+\mathcal Z_2).
\end{align*}
Combining the preceding estimates, we conclude that
\begin{align}\label{quanju52}
 \bigl|{\mathscr N}_4\bigr|
 \lesssim
 \mathcal{Z}_1+\mathcal{Z}_2.
\end{align}

Collecting \eqref{quanju47}, \eqref{quanju48}, and
\eqref{quanju52}, we obtain
\begin{align*}
 \bigl|{\mathscr N}_2\bigr|
 +\bigl|{\mathscr N}_3\bigr|
 +\bigl|{\mathscr N}_4\bigr|
 \leq
 C\bigl(\mathcal{Z}_1+\mathcal{Z}_2+\mathcal{Z}_3\bigr),
\end{align*}
where $
 C=C(N, \beta, \vartheta)$
is independent of $t,\mu,\nu$, and $\kappa$.

\medskip

\noindent
\textbf{Step 6. Estimate of the principal nonlinear term
${\mathscr N}_1$.}

 For convenience, we write
\begin{align*}
		{\mathscr N}_1
 =&  {\mathscr N}_{1,1} + {\mathscr N}_{1,2} + {\mathscr N}_{1,3} + {\mathscr N}_{1,4},
		\end{align*}
with
		\begin{align*}
		&{\mathscr N}_{1,1} = -2 \operatorname{Re}\left\langle \Lambda^N_t\left( \u \cdot \nabla \u\right) ,  \mathcal{W}  \Lambda_t^N \u \right\rangle,\quad {\mathscr N}_{1,2} = 2 \operatorname{Re}\left\langle \Lambda^N_t \left( \mathbf{b} \cdot \nabla \mathbf{b} \right)   ,  \mathcal{W}  \Lambda_t^N  \u \right\rangle,\\
		&{\mathscr N}_{1,3} = -2 \operatorname{Re}\left\langle \Lambda^N_t\left( \u \cdot \nabla \mathbf{b} \right) ,  \mathcal{W}  \Lambda_t^N \mathbf{b} \right\rangle,\quad {\mathscr N}_{1,4}= 2 \operatorname{Re}\left\langle \Lambda^N_t \left( \mathbf{b} \cdot \nabla \u \right)   ,  \mathcal{W}  \Lambda_t^N  \mathbf{b} \right\rangle.
		\end{align*}

We first consider ${\mathscr N}_{1,1}$. Using
$\div\mathbf u=0$ and symmetrizing with respect to
$(k,\xi)$ and $(\ell,\eta)$, we obtain
\begin{align*}
{\mathscr N}_{1,1}
=&-\operatorname{Re}\sum_{k,\ell\in\mathbb Z}\int_{\mathbb R^2}
i\Bigl[
\ell\mathcal W_k(t,\xi)\Lambda_t^{2N}(k,\xi)
-k\mathcal W_\ell(t,\eta)\Lambda_t^{2N}(\ell,\eta)
\Bigr]
\widehat{u^1_{k-\ell}}(\xi-\eta)
\cdot
\widehat{\mathbf u}_\ell(\eta)
\cdot\overline{\widehat{\mathbf u}_k(\xi)}
\,d\eta\,d\xi
\notag\\
&-\operatorname{Re}\sum_{k,\ell\in\mathbb Z}\int_{\mathbb R^2}
i\Bigl[
\eta\mathcal W_k(t,\xi)\Lambda_t^{2N}(k,\xi)
-\xi\mathcal W_\ell(t,\eta)\Lambda_t^{2N}(\ell,\eta)
\Bigr]
\widehat{u^2_{k-\ell}}(\xi-\eta)
\cdot
\widehat{\mathbf u}_\ell(\eta)
\cdot\overline{\widehat{\mathbf u}_k(\xi)}
\,d\eta\,d\xi
\notag\\
=:&{\mathscr N}_{1,1,1}
+{\mathscr N}_{1,1,2}.
\end{align*}
By the commutator estimate and Lemma~\ref{leyoung}, we obtain
\begin{align}\label{quanju55}
\bigl|{\mathscr N}_{1,1,1}\bigr|
\lesssim&
\left\|
  \sqrt{\mathcal{W}}\Lambda_t^N\mathbf u_{\neq}
\right\|_{L^2}^2
\left\|
  \sqrt{\mathcal{W}}\Lambda_t^N\mathbf u
\right\|_{L^2}
\notag\\
&+
\kappa^{-1/6}
\left\|
  |\partial_x|^{1/3}
  \sqrt{\mathcal{W}}\Lambda_t^N\mathbf u_{\neq}
\right\|_{L^2}^2
\left\|
  \sqrt{\mathcal{W}}\Lambda_t^N\mathbf u
\right\|_{L^2}
\notag\\
\lesssim&
\mathcal{Z}_5+\mathcal{Z}_7.
\end{align}

We next estimate ${\mathscr N}_{1,1,2}$ by separating two
frequency regimes.

\medskip
\noindent
\textbf{Case 1.}
Suppose that
\begin{align*}
\Lambda_t(k-\ell,\xi-\eta)
\leq\frac12\Bigl(
\Lambda_t(k,\xi)+\Lambda_t(\ell,\eta)
\Bigr).
\end{align*}
The triangle inequality then gives
\begin{align*}
\Lambda_t(k,\xi)\approx\Lambda_t(\ell,\eta).
\end{align*}
In this region, we use the identity
\begin{align}\label{quanju56}
&\eta\mathcal W_k(t,\xi)\Lambda_t^{2N}(k,\xi)
-\xi\mathcal W_\ell(t,\eta)\Lambda_t^{2N}(\ell,\eta)\nn\\
&\quad=
\eta\bigl(\mathcal W_k(t,\xi)-\mathcal W_\ell(t,\eta)\bigr)
\Lambda_t^{2N}(k,\xi)+(\eta-\xi)\mathcal W_\ell(t,\eta)\Lambda_t^{2N}(k,\xi)
\nn\\
&\qquad
+\xi\mathcal W_\ell(t,\eta)
\Bigl(
\Lambda_t^{2N}(k,\xi)-\Lambda_t^{2N}(\ell,\eta)
\Bigr).
\end{align}
Accordingly, we write
\begin{align*}
{\mathscr N}_{1,1,2}
=
{\mathscr N}_{1,1,2,M}
+{\mathscr N}_{1,1,2,y}
+{\mathscr N}_{1,1,2,x},
\end{align*}
where the three terms correspond, respectively, to the multiplier,
derivative, and Sobolev commutators in
\eqref{quanju56}.

For the derivative commutator, incompressibility gives
\begin{align*}
(k-\ell)\widehat u^1_{k-\ell}(\xi-\eta)
+(\xi-\eta)\widehat u^2_{k-\ell}(\xi-\eta)=0.
\end{align*}
Hence Lemma~\ref{leyoung} implies
\begin{equation}\label{quanju57}
\bigl|{\mathscr N}_{1,1,2,y}\bigr|
\lesssim
X_{\mathbf u}X_{\mathbf u,\neq}^2
=\mathcal{Z}_5.
\end{equation}
For the Sobolev commutator, the comparability of the output frequencies implies
\begin{align*}
    |\Lambda_t^{2N}(k,\xi) - \Lambda_t^{2N}(\ell,\eta)| \lesssim \Lambda_t^N(\ell,\eta)\Lambda_t(k-\ell,\xi-\eta)\Lambda_t^{N-1}(k,\xi).
\end{align*}
Using incompressibility and the following geometric bound $$\frac{|k|}{(k^2+\xi^2)^{1/2}} \lesssim \langle t \rangle^{-1} \Lambda_t(k,\xi),$$ the Cauchy-Schwarz inequality gives:
\begin{align}\label{quanju58}
\bigl\|\sqrt{\mathcal{W}_k(t,\xi)}\Lambda_t\widehat{u^2}\bigr\|_{\ell_k^1L_\xi^1}
&\lesssim
\sum_{k\neq0}\int_{\mathbb R}
\frac{|k|\Lambda_t(k,\xi)}{(k^2+\xi^2)^{1/2}}
\sqrt{\mathcal W_k(t,\xi)}
\bigl|\widehat{\mathbf u}_k(\xi)\bigr|
\,d\xi \notag\\
&\lesssim
\langle t\rangle^{-1}
\sum_{k\neq0}\int_{\mathbb R}
\Lambda_t^2(k,\xi)
\sqrt{\mathcal W_k(t,\xi)}
\bigl|\widehat{\mathbf u}_k(\xi)\bigr|
\,d\xi \notag\\
&\lesssim
\langle t\rangle^{-1}X_{\mathbf u,\neq},
\qquad \text{for } N\geq4.
\end{align}
It follows from Lemma \ref{leyoung} that
\begin{align*}
    \bigl|{\mathscr N}_{1,1,2,x}\bigr| \lesssim X_{\mathbf u}X_{\mathbf u,\neq}^2 = \mathcal{Z}_5.
\end{align*}

By using the boundedness of the multipliers and inequality \eqref{quanju58}, we obtain  
\begin{equation}\label{quanju60}
\bigl|{\mathscr N}_{1,1,2,M}\bigr|
\lesssim
 \mathcal{Z}_4 ,
\end{equation}
after allocating the dissipative contribution to
$\operatorname{Dis}_1$. Combining
\eqref{quanju57}--\eqref{quanju60}, we obtain
\begin{equation}\label{quanju61}
\bigl|{\mathscr N}_{1,1,2}\bigr|
\lesssim
 \mathcal{Z}_4 +\mathcal{Z}_5
\qquad\text{in Case~1}.
\end{equation}

\medskip
\noindent
\textbf{Case 2.}
Suppose that
\begin{equation}\label{quanju62}
\Lambda_t(k-\ell,\xi-\eta)
\geq
\frac12\Bigl(
\Lambda_t(k,\xi)+\Lambda_t(\ell,\eta)
\Bigr).
\end{equation}
Set
\begin{align*}
m=k-\ell,\qquad
\zeta=\xi-\eta,\qquad
\chi_j=\chi(\kappa^{1/2}|j|).
\end{align*}

In this region, the advecting mode $(m,\zeta)$ carries the dominant
sheared frequency.
Under this condition, we split ${\mathscr{N}}_{1,1,2}$ into two components, the triangle inequality $$ |\eta| + |\xi| \le \Lambda_t(k,\xi) + \Lambda_t(\ell,\eta) +\left( |k|+|\ell|\right)t$$ gives
\begin{align*}
\bigl|{\mathscr N}_{1,1,2}\bigr|
\lesssim R_1+R_2,
\end{align*}
where
\begin{align}\label{R1-definition}
R_1
:=&
\sum_{k,\ell\in\mathbb Z}
\int_{\mathbb R^2}
{\bf1}_{\eqref{quanju62}}
\Bigl[
\Lambda_t(\ell,\eta)
\mathcal W_k(t,\xi)\Lambda_t^{2N}(k,\xi)
+
\Lambda_t(k,\xi)
\mathcal W_\ell(t,\eta)\Lambda_t^{2N}(\ell,\eta)
\Bigr]
\bigl|\widehat{u^2_{k-\ell}}(\xi-\eta)\bigr|
\notag\\[-1mm]
&\hspace{27mm}\times
\bigl|\widehat{\mathbf u}_\ell(\eta)\bigr|
\bigl|\widehat{\mathbf u}_k(\xi)\bigr|
\,d\eta\,d\xi,
\end{align}
\begin{align}
R_2
:=&
t\sum_{k,\ell\in\mathbb Z}
\int_{\mathbb R^2}
{\bf1}_{\eqref{quanju62}}
\Bigl[
|\ell|
\mathcal W_k(t,\xi)\Lambda_t^{2N}(k,\xi)
+
|k|
\mathcal W_\ell(t,\eta)\Lambda_t^{2N}(\ell,\eta)
\Bigr]
\bigl|\widehat{u^2_{k-\ell}}(\xi-\eta)\bigr|
\notag\\[-1mm]
&\hspace{27mm}\times
\bigl|\widehat{\mathbf u}_\ell(\eta)\bigr|
\bigl|\widehat{\mathbf u}_k(\xi)\bigr|
\,d\eta\,d\xi.\nn
\end{align}
On the support of \eqref{quanju62}, the Sobolev weights in
\eqref{R1-definition} can be distributed among the three factors.
Consequently, Lemma~\ref{leyoung} gives
\begin{equation}\label{quanju63}
R_1
\lesssim
X_{\mathbf u,\neq}^2X_{\mathbf u}
=
\mathcal Z_5.
\end{equation}

For $ R_2$, it suffices to estimate its first component, which we also denote by $ R_2$ with a slight abuse of notation. Noting that
$
t \, e^{-\frac{1}{64} \kappa^{\frac{1}{3}} t} \lesssim \kappa^{-\frac{1}{3}},
$
we then obtain
\begin{align*}
R_2
&\lesssim
\kappa^{-\frac{1}{3}}
\Bigg|
\sum_{\substack{k,\ell\in\mathbb Z\\ k\neq \ell}}
\int_{\mathbb R^2}
\sqrt{\mathcal W_{k-\ell}}\,
\Lambda_t^N(k-\ell,\xi-\eta)
\widehat{u^2_{k-\ell}}(\xi-\eta)
|\ell|\sqrt{\mathcal W_{\ell}}\,
\widehat{\mathbf u}_{\ell}(\eta)
\cdot
\sqrt{\mathcal W_k}\,
\Lambda_t^N(k,\xi)
\overline{\widehat{\mathbf u}_k(\xi)}
\,d\eta\,d\xi
\Bigg|.
\end{align*}
To estimate this trilinear term, we introduce the decomposition
\begin{align*}
1
=
\chi_k\chi_m
+
(1-\chi_k)
+
\chi_k(1-\chi_m),
\qquad m=k-\ell.
\end{align*}
Accordingly, we write
\begin{align*}
R_2=R_{2,1}+R_{2,2}+R_{2,3},
\end{align*}
where the three terms correspond to the three frequency regions determined by the above decomposition.

We first consider the contribution $R_{2,2}$. On the support of $1-\chi_k$, we have
$
\kappa^{1/3}|k|^{2/3}\gtrsim 1.
$
Thus, Lemma~\ref{leyoung} yields
\begin{align*}
R_{2,2}
\lesssim
\kappa^{-1/3}
X_{\mathbf u}\,
\operatorname{CK}_{2,\mathbf u}^{1/2}
\operatorname{CK}_{1,\mathbf u}^{1/2}.
\end{align*}
For the localized echo interaction $R_{2,1}$, we write
$k=m+\ell$, $\xi=\zeta+\eta$.
Then
\begin{align*}
R_{2,1}
:=&
\kappa^{-1/3}
\Bigg|
\sum_{\substack{m,\ell\in\mathbb Z\\ m\neq0,\ \ell\neq0}}
\int_{\mathbb R^2}
\mathbf 1_{\eqref{quanju62}}\,
\chi_m\chi_{m+\ell}\,
\sqrt{\mathcal W_m(t,\zeta)}
\Lambda_t^N(m,\zeta)
\widehat{u_m^2}(\zeta)
\notag\\
&\hspace{16mm}\cdot
|\ell|
\sqrt{\mathcal W_\ell(t,\eta)}
\widehat{\mathbf u}_\ell(\eta)
\cdot
\sqrt{\mathcal W_{m+\ell}(t,\zeta+\eta)}
\Lambda_t^N(m+\ell,\zeta+\eta)
\cdot
\overline{
\widehat{\mathbf u}_{m+\ell}(\zeta+\eta)
}
\,d\eta\,d\zeta
\Bigg|.
\end{align*}
Here the indicator $\mathbf 1_{\eqref{quanju62}}$ is evaluated at
$(k,\xi)=(m+\ell,\zeta+\eta)$.

By the Cauchy--Schwarz inequality, we obtain
\begin{align*}
R_{2,1}
\lesssim{}&
\kappa^{-1/3}
\Bigg(
\sum_{\substack{m,\ell\in\mathbb Z\\m\neq0,\ \ell\neq0}}
\int_{\mathbb R^2}
\chi_m^2
\left\langle\frac{\zeta}{m}\right\rangle
\left[
\log\left(
1+\left\langle\frac{\zeta}{m}\right\rangle
\right)
\right]^4
\notag\\
&\hspace{15mm}\times
\mathcal W_m(t,\zeta)
\Lambda_t^{2N}(m,\zeta)
|\widehat{u_m^2}(\zeta)|^2
|\ell|^2
\mathcal W_\ell(t,\eta)
\Lambda_t^{2N-2}(\ell,\eta)
|\widehat{\mathbf u}_\ell(\eta)|^2
\,d\eta\,d\zeta
\Bigg)^{1/2}
\notag\\
&\quad\times
\Bigg(
\sum_{\substack{m,\ell\in\mathbb Z\\m\neq0,\ \ell\neq0}}
\int_{\mathbb R^2}
\frac{
\chi_{m+\ell}^2
\mathcal W_{m+\ell}(t,\zeta+\eta)
\Lambda_t^{2N}(m+\ell,\zeta+\eta)
\left|
\widehat{\mathbf u}_{m+\ell}(\zeta+\eta)
\right|^2
}{
\Big(
\langle\ell\rangle^{2N-2}
+
|\eta+\ell t|^{2N-2}
\Big)
\left(
1+\left|\frac{\zeta}{m}\right|
\right)
\left[
\log\left(
1+\left\langle\frac{\zeta}{m}\right\rangle
\right)
\right]^2
}
\,d\eta\,d\zeta
\Bigg)^{1/2}.
\end{align*}
For the second factor, since $\chi_k^2\leq\chi_k$ and
$\mathcal W_k(t,\xi)\leq\mathcal A_k(t,\xi)$, the definition of
$\mathcal G_k$ implies
\begin{align*}
&\sum_{\substack{m,\ell\in\mathbb Z\\m\neq0,\ \ell\neq0}}
\int_{\mathbb R^2}
\frac{
\chi_{m+\ell}^2
\mathcal W_{m+\ell}(t,\zeta+\eta)
\Lambda_t^{2N}(m+\ell,\zeta+\eta)
\left|
\widehat{\mathbf u}_{m+\ell}(\zeta+\eta)
\right|^2
}{
\left(
\langle\ell\rangle^{2N-2}
+
|\eta+\ell t|^{2N-2}
\right)
\left(
1+\left|\frac{\zeta}{m}\right|
\right)
\left[
\log\left(
1+\left\langle\frac{\zeta}{m}\right\rangle
\right)
\right]^2
}
\,d\eta\,d\zeta
\\
&\quad\lesssim
\sum_{k\in\mathbb Z}\int_{\mathbb R}
\mathcal G_k(t,\xi)
\mathcal A_k(t,\xi)
\Lambda_t^{2N}(k,\xi)
|\widehat{\mathbf u}_k(\xi)|^2\,d\xi
=
\operatorname{CK}_{3,\mathbf u}(t).
\end{align*}
Consequently,
\begin{align*}
R_{2,1}
\lesssim&
\kappa^{-1/3}
X_{\mathbf u}(t)
\operatorname{CK}_{3,\mathbf u}^{1/2}(t)
\left(
\sum_{m\neq0}\int_{\mathbb R}
\chi_m^2
\frac{
\left\langle\frac{\zeta}{m}\right\rangle^{1/2}
}{
\log\left(
1+\left\langle\frac{\zeta}{m}\right\rangle
\right)
}
\left\langle\frac{\zeta}{m}\right\rangle
\mathcal W_m(t,\zeta)
\Lambda_t^{2N}(m,\zeta)
|\widehat{u_m^2}(\zeta)|^2\,d\zeta
\right)^{1/2}.
\end{align*}
Here we used the elementary bound
$
\left[
\log\left(1+\langle x\rangle\right)
\right]^6
\lesssim
\langle x\rangle^{1/2}.
$
Moreover, by the Cauchy--Schwarz inequality and incompressibility,
\begin{align*}
&\sum_{m\neq0}\int_{\mathbb R}
\chi_m^2
\frac{
\left\langle\frac{\zeta}{m}\right\rangle^{1/2}
}{
\log\left(
1+\left\langle\frac{\zeta}{m}\right\rangle
\right)
}
\left\langle\frac{\zeta}{m}\right\rangle
\mathcal W_m(t,\zeta)
\Lambda_t^{2N}(m,\zeta)
|\widehat{u_m^2}(\zeta)|^2\,d\zeta
\\
&\quad\leq
\left(
\sum_{m\neq0}\int_{\mathbb R}
\chi_m^2
\delta'\left(\frac{\zeta}{m}\right)
\mathcal W_m(t,\zeta)
\Lambda_t^{2N}(m,\zeta)
|\widehat{\mathbf u}_m(\zeta)|^2\,d\zeta
\right)^{1/2}
\\
&\qquad\quad\times
\left(
\sum_{m\neq0}\int_{\mathbb R}
\chi_m^2
\mathcal W_m(t,\zeta)
\Lambda_t^{2N}(m,\zeta)
|\widehat{\mathbf u}_m(\zeta)|^2\,d\zeta
\right)^{1/2}
\\
&\quad\lesssim
\operatorname{CK}_{2,\mathbf u}^{1/2}(t)
\left(
\kappa^{-1/3}
\operatorname{CK}_{1,\mathbf u}(t)
\right)^{1/2}
=
\kappa^{-1/6}
\operatorname{CK}_{2,\mathbf u}^{1/2}(t)
\operatorname{CK}_{1,\mathbf u}^{1/2}(t).
\end{align*}
Here we used
\begin{align*}
\left\langle\frac{\zeta}{m}\right\rangle^2
|\widehat{u_m^2}(\zeta)|^2
=
|\widehat{\mathbf u}_m(\zeta)|^2,
\qquad m\neq0,
\end{align*}
as well as $\chi_m^2\leq\chi_m$,
$\mathcal W_m(t,\zeta)\leq\mathcal A_m(t,\zeta)$, and $|m|\geq1$.

Therefore,
\begin{align*}
R_{2,1}
\lesssim
\kappa^{-5/12}
X_{\mathbf u}
\operatorname{CK}_{3,\mathbf u}^{1/2}
\operatorname{CK}_{2,\mathbf u}^{1/4}
\operatorname{CK}_{1,\mathbf u}^{1/4}.
\end{align*}

It remains to estimate $R_{2,3}$. On the support of $1-\chi_m$,
we have
$
\kappa^{1/3}|m|^{2/3}\gtrsim1.
$
Furthermore, incompressibility gives
\begin{align*}
\left\langle\frac{\zeta}{m}\right\rangle
\left[
\log\left(
1+\left\langle\frac{\zeta}{m}\right\rangle
\right)
\right]^4
|\widehat{u_m^2}(\zeta)|^2
\lesssim
|\widehat{\mathbf u}_m(\zeta)|^2.
\end{align*}
Consequently,
\begin{align*}
R_{2,3}
\lesssim
\kappa^{-1/3}
X_{\mathbf u}
\operatorname{CK}_{3,\mathbf u}^{1/2}
\operatorname{CK}_{1,\mathbf u}^{1/2}.
\end{align*}

Combining the estimates for $R_{2,1}$, $R_{2,2}$, and $R_{2,3}$,
we obtain
\begin{align*}
R_2
&\lesssim
\kappa^{-\frac{1}{3}} 	X_{\mathbf{u}} \mathrm{CK}_{3,\mathbf u}^{\frac{1}{2}}  \mathrm{CK}_{1,\mathbf u}^{\frac{1}{2}} + \kappa^{-\frac{1}{3}} 	X_{\mathbf{u}} \mathrm{CK}_{2,\mathbf u}^{\frac{1}{2}}  \mathrm{CK}_{1,\mathbf u}^{\frac{1}{2}}
+
\kappa^{-5/12}
X_{\mathbf u}
\operatorname{CK}_{3,\mathbf u}^{1/2}
\operatorname{CK}_{2,\mathbf u}^{1/4}
\operatorname{CK}_{1,\mathbf u}^{1/4}
\notag\\
&\lesssim \mathcal Z_8.
\end{align*}

Together with \eqref{quanju63}, this proves
\begin{align*}
\bigl|{\mathscr N}_{1,1,2}\bigr|
\lesssim
\mathcal Z_5+\mathcal Z_8
\qquad\text{in Case~2}.
\end{align*}
Finally, combining the estimates in Case~1 and Case~2 with
\eqref{quanju55}, we conclude that
\begin{equation}\label{quanju73}
\bigl|{\mathscr N}_{1,1}\bigr|
\lesssim
\mathcal Z_4+\mathcal Z_5+\mathcal Z_8+\mathcal {Z}_7.
\end{equation}

\vskip .1in
The estimate of ${\mathscr N}_{1,3}$ follows from the same
frequency decomposition, with $\mathbf u$ and $\mathbf b$
interchanged where appropriate. The only additional cubic term is
estimated by
\begin{align*}
X_{\mathbf b}X_{\mathbf b,\neq}X_{\mathbf u,\neq}
\leq
\frac12X_{\mathbf b}
\bigl(
X_{\mathbf b,\neq}^2+X_{\mathbf u,\neq}^2
\bigr)
=
\frac12(\mathcal Z_2+\mathcal Z_3).
\end{align*}
Consequently,
\begin{equation}\label{quanju74}
\bigl|{\mathscr N}_{1,3}\bigr|
\lesssim
\mathcal Z_2+\mathcal Z_3+\mathcal Z_4+\mathcal Z_8.
\end{equation}

We next estimate the magnetic stretching terms. Pairing
${\mathscr N}_{1,2}$ and
${\mathscr N}_{1,4}$ and symmetrizing, we obtain
\begin{align*}
&{\mathscr N}_{1,2}
+{\mathscr N}_{1,4}\nn\\
&\quad=
2\operatorname{Re}\sum_{k,\ell\in\mathbb Z}
\int_{\mathbb R^2}
i\Bigl[
\ell\mathcal W_k(t,\xi)\Lambda_t^{2N}(k,\xi)
-k\mathcal W_\ell(t,\eta)\Lambda_t^{2N}(\ell,\eta)
\Bigr]
\widehat{b^1_{k-\ell}}(\xi-\eta)
\widehat{\mathbf u}_\ell(\eta)
\cdot\overline{\widehat{\mathbf b}_k(\xi)}
\,d\eta\,d\xi
\notag\\
&\qquad+
2\operatorname{Re}\sum_{k,\ell\in\mathbb Z}
\int_{\mathbb R^2}
i\Bigl[
\eta\mathcal W_k(t,\xi)\Lambda_t^{2N}(k,\xi)
-\xi\mathcal W_\ell(t,\eta)\Lambda_t^{2N}(\ell,\eta)
\Bigr]
\widehat{b^2_{k-\ell}}(\xi-\eta)
\widehat{\mathbf u}_\ell(\eta)
\cdot\overline{\widehat{\mathbf b}_k(\xi)}
\,d\eta\,d\xi
\notag\\
&\quad=:
{\mathscr N}_x+{\mathscr N}_y.
\end{align*}

For the horizontal contribution, Lemma~\ref{commutator} and
Lemma~\ref{leyoung} give
\begin{equation}\label{quanju76}
\bigl|{\mathscr N}_x\bigr|
\lesssim
X_{\mathbf b,\neq}^2X_{\mathbf b}
+
X_{\mathbf u,\neq}^2X_{\mathbf b}
+
\operatorname{Dis}_1^{1/2}\operatorname{CK}_2^{1/2}X
\lesssim
\mathcal Z_2+\mathcal Z_3+\mathcal Z_4.
\end{equation}

For the vertical contribution, interchanging $(k,\xi)$ and
$(\ell,\eta)$ in one copy of the integral yields
\begin{align}\label{quanju77}
{\mathscr N}_y
=&
\operatorname{Re}\sum_{k,\ell\in\mathbb Z}
\int_{\mathbb R^2}
i\Bigl[
\eta\mathcal W_k(t,\xi)\Lambda_t^{2N}(k,\xi)
-\xi\mathcal W_\ell(t,\eta)\Lambda_t^{2N}(\ell,\eta)
\Bigr]
\widehat{b^2_{k-\ell}}(\xi-\eta)
\notag\\[-1mm]
&\hspace{20mm}\times
\left(
\widehat{\mathbf b}_\ell(\eta)
 \cdot\overline{\widehat{\mathbf u}_k(\xi)}
+
\widehat{\mathbf u}_\ell(\eta)
 \cdot\overline{\widehat{\mathbf b}_k(\xi)}
\right)
\,d\eta\,d\xi.
\end{align}

In Case~1, the multiplier, derivative, and Sobolev commutator
estimates imply
\begin{align*}
\bigl|{\mathscr N}_y\bigr|
\lesssim&
X_{\mathbf b}X_{\mathbf b,\neq}X_{\mathbf u,\neq}
+
X_{\mathbf b,\neq}^2X_{\mathbf u}
+
\operatorname{Dis}_1^{1/2}\operatorname{CK}_2^{1/2}X.
\end{align*}
Since
\begin{align*}
X_{\mathbf b}X_{\mathbf b,\neq}X_{\mathbf u,\neq}
\leq
\frac12X_{\mathbf b}
\bigl(
X_{\mathbf b,\neq}^2+X_{\mathbf u,\neq}^2
\bigr)
=
\frac12(\mathcal Z_2+\mathcal Z_3),
\end{align*}
we obtain
\begin{equation}\label{quanju79}
\bigl|{\mathscr N}_y\bigr|
\lesssim
\mathcal Z_2+\mathcal Z_3+\mathcal Z_4+\mathcal Z_6
\qquad\text{in Case~1}.
\end{equation}

In Case~2, the same argument as for ${\mathscr N}_{1,1,2}$ gives
\begin{equation}\label{quanju81}
\bigl|{\mathscr N}_y\bigr|
\lesssim
\mathcal Z_2+\mathcal Z_3+\mathcal Z_6+\mathcal Z_8
\qquad\text{in Case~2}.
\end{equation}

The remaining mixed commutator contribution satisfies
\begin{align*}
X_{\mathbf b,\neq}X_{\mathbf u,\neq}X_{\mathbf u}
=\mathcal Z_1.
\end{align*}
Collecting \eqref{quanju73}, \eqref{quanju74},
\eqref{quanju76}, \eqref{quanju79}, and \eqref{quanju81}, we conclude
that
\begin{equation}\label{quanju82}
\bigl|{\mathscr N}_1(t)\bigr|
\lesssim
\sum_{i=1}^8\mathcal Z_i(t).
\end{equation}

Combining \eqref{quanju82} with the weighted energy identity
and the estimates obtained in the preceding steps gives
\begin{align*}
\frac{d}{dt}E(t)
+c_D\operatorname{Dis}_1(t)
+c_1\operatorname{CK}_1(t)
+c_2\operatorname{CK}_2(t)
+c_3\operatorname{CK}_3(t)
\leq
C\sum_{i=1}^8\mathcal{Z}_i(t).
\end{align*}
The assumptions on $\beta$ imply
\begin{align*}
    c_D
    &=
    \frac16
    -
    \frac{\mu+\nu}{|\beta|\sqrt{\mu\nu}}
    \geq\vartheta>0,
\\
    c_1
    &=
    \frac{1}{6\pi}
    -
    \frac{7\pi+1}{6\pi|\beta|}
    =
    \frac{|\beta|-(7\pi+1)}
    {6\pi|\beta|}
    >0,
\\
    c_2
    &=
    \frac{1}{6D_\delta}
    -
    \frac{12D_\delta+1}
    {12D_\delta|\beta|}
    >0,
\\
    c_3
    &=
    1-\frac{1}{2|\beta|}
    >0.
\end{align*}
In particular, the positivity of $c_1$ follows solely from
\begin{align*}
    |\beta|>7\pi+1.
\end{align*}

Finally, the uniform equivalence of $\mathcal{W}$ and $\mathcal{A}$, together with
the control of the cross terms in the definition of $E$, yields
constants $0<c_E\leq C_E<\infty$, independent of
$t,\mu,\nu$, and $\kappa$, such that
\begin{align*}
c_EX^2(t)\leq E(t)\leq C_EX^2(t).
\end{align*}
This completes the proof of Proposition~\ref{quanju3}.
\end{proof}

\subsection{Proof of the nonlinear stability in Theorem~\ref{thm1.2}}

We first recall the local theory used below. Let $N\geq4$ and assume
that
\begin{align*}
 (\mathbf{u}_{\rm in},\mathbf{b}_{\rm in})
 \in H^N(\mathbb T\times\mathbb R),
 \qquad
 \operatorname{div}\mathbf{u}_{\rm in}
 =
 \operatorname{div}\mathbf{b}_{\rm in}=0.
\end{align*}
Then System \eqref{main} has a unique solution on some interval
$[0,T_{\rm loc}]$ such that
\begin{align}\label{quanju83}
 (\mathbf{u},\mathbf{b})
 &\in C([0,T_{\rm loc}];H^N),\qquad
 \bigl(\sqrt{\nu}\,\partial_y\mathbf{u},
       \sqrt{\mu}\,\partial_y\mathbf{b}\bigr)
 \in L^2(0,T_{\rm loc};H^N).
\end{align}
The divergence-free constraints are preserved, and the solution
extends beyond $T<\infty$ if
\begin{align}\label{quanju84}
 \int_0^T
 \left(
 \|\nabla\mathbf{u}(t)\|_{L^\infty}
 +\|\nabla\mathbf{b}(t)\|_{L^\infty}
 \right)\,dt<\infty.
\end{align}
Indeed, applying Friedrichs approximation to the Leray-projected
system and using the standard commutator estimate gives
\begin{align}\label{quanju85}
 \frac{d}{dt}
 \|(\mathbf{u},\mathbf{b})\|_{H^N}^2
 &+\nu\|\partial_y\mathbf{u}\|_{H^N}^2
 +\mu\|\partial_y\mathbf{b}\|_{H^N}^2\lesssim
 \left(
 1+\|\nabla\mathbf{u}\|_{L^\infty}
   +\|\nabla\mathbf{b}\|_{L^\infty}
 \right)
 \|(\mathbf{u},\mathbf{b})\|_{H^N}^2.
\end{align}
Here the constant term accounts for the linear Couette terms. For $N \geq 4$, the Sobolev embedding $H^N \hookrightarrow W^{1,\infty}$ ensures that the local estimates close and uniqueness holds. The \emph{a priori} weighted energy identities derived above are rigorously justified by applying them to the Friedrichs approximations and passing to the weak limit. Furthermore, on any finite interval $[0,T]$, the weighted norm $\| \Lambda_t^N (\mathbf{u},\mathbf{b})\|_{L^2}$ and the standard $H^N$ norm are equivalent. Consequently, a uniform bound on the weighted energy precludes finite-time blowup in $H^N$, immediately verifying the continuation criterion \eqref{quanju84}.

\vskip .1in
We now close the nonlinear estimates. Fix a positive constant $c$,
to be chosen below, and suppose that
\begin{align}\label{quanju86}
 \bigl\|
 \sqrt{\mathcal{W}}\Lambda_t^N
 (\mathbf{u},\mathbf{b})(t)
 \bigr\|_{L^2}
 \leq c\kappa^{1/2},
 \qquad 0\leq t\leq T.
\end{align}
The uniform bound for $\mathcal{M}_k^{(1)}(\xi)+\mathcal{M}_k^{(2)}(\xi)+\mathcal{M}_k^{(3)}(t,\xi)+\frac23$ implies
\begin{align*}
 \mathcal{W}_k(t,\xi)\approx \mathcal{A}_k(t,\xi).
\end{align*}
Since $|k|\geq1$ on the nonzero horizontal modes, the definition of
$\operatorname{CK}_1$ gives
\begin{align}\label{quanju87}
 \bigl\|
 \sqrt{\mathcal{W}}\Lambda_t^N
 (\mathbf{u}_{\neq},\mathbf{b}_{\neq})
 \bigr\|_{L^2}^2
 &\lesssim
 \bigl\|
 \sqrt{\mathcal{A}}\Lambda_t^N
 (\mathbf{u}_{\neq},\mathbf{b}_{\neq})
 \bigr\|_{L^2}^2\lesssim
 \kappa^{-1/3}\operatorname{CK}_1.
\end{align}
Similarly, every cubic remainder containing two nonzero modes
satisfies
\begin{align}\label{quanju89}
 \bigl\|
 \sqrt{\mathcal{W}}\Lambda_t^N
 (\mathbf{u}_{\neq},\mathbf{b}_{\neq})
 \bigr\|_{L^2}^2
 \bigl\|
 \sqrt{\mathcal{W}}\Lambda_t^N
 (\mathbf{u},\mathbf{b})
 \bigr\|_{L^2}\lesssim&
 c\kappa^{1/2}
 \bigl\|
 \sqrt{\mathcal{W}}\Lambda_t^N
 (\mathbf{u}_{\neq},\mathbf{b}_{\neq})
 \bigr\|_{L^2}^2\nn\\
 \lesssim& c\operatorname{CK}_1.
\end{align}
This includes $\mathcal{Z}_1,\mathcal{Z}_2,\mathcal{Z}_3,
\mathcal{Z}_5$, and $\mathcal{Z}_6$ in
Proposition~\ref{quanju3}.

We now estimate the nonlinear remainders. First, the remainder involving the physical dissipation satisfies
\begin{align}\label{quanju90}
\mathcal Z_4(t)
&=
\kappa^{-1/2}
\operatorname{Dis}_1^{1/2}(t)
\operatorname{CK}_2^{1/2}(t)
X(t)
\notag\\
&\leq
c\operatorname{Dis}_1^{1/2}(t)
\operatorname{CK}_2^{1/2}(t)
\notag\\
&\leq
\frac{c}{2}
\left(
\operatorname{Dis}_1(t)
+
\operatorname{CK}_2(t)
\right).
\end{align}
The term $\mathcal {Z}_7$ is treated separately. By the definition of $\operatorname{CK}_1(t)$, there holds
\[
H_{\neq}^2(t)
\lesssim
\kappa^{-1/3}\operatorname{CK}_1(t).
\]
Therefore,
\begin{align}\label{quanju90b}
\mathcal {Z}_7(t)
&=
\kappa^{-1/6}H_{\neq}^2(t)X(t)
\notag\\
&\lesssim
\kappa^{-1/6}
\kappa^{-1/3}
\operatorname{CK}_1(t)
\,c\kappa^{1/2}
\notag\\
&\lesssim
c\operatorname{CK}_1(t).
\end{align}
Thus $\mathcal {Z}_7$ is absorbed through $\operatorname{CK}_1$.

The echo estimates also produce the remainder
\[
\kappa^{-1/3}X_{\mathbf u}(t)
\operatorname{CK}_1^{1/2}(t)
\left(
\operatorname{CK}_2^{1/2}(t)
+
\operatorname{CK}_3^{1/2}(t)
\right).
\]
Since $X_{\mathbf u}(t)\leq X(t)\leq c\kappa^{1/2}$,
\[
\kappa^{-1/3}X_{\mathbf u}(t)
\leq
c\kappa^{1/6}
\leq c.
\]
Consequently, by Young's inequality,
\begin{align}\label{quanju90c}
&\kappa^{-1/3}X_{\mathbf u}(t)
\operatorname{CK}_1^{1/2}(t)
\left(
\operatorname{CK}_2^{1/2}(t)
+
\operatorname{CK}_3^{1/2}(t)
\right)
\notag\\
&\quad\leq
c\kappa^{1/6}
\left[
\operatorname{CK}_1^{1/2}(t)
\operatorname{CK}_2^{1/2}(t)
+
\operatorname{CK}_1^{1/2}(t)
\operatorname{CK}_3^{1/2}(t)
\right]
\notag\\
&\quad\leq
\frac{c\kappa^{1/6}}{2}
\left(
2\operatorname{CK}_1(t)
+
\operatorname{CK}_2(t)
+
\operatorname{CK}_3(t)
\right)
\notag\\
&\quad\lesssim
c\left(
\operatorname{CK}_1(t)
+
\operatorname{CK}_2(t)
+
\operatorname{CK}_3(t)
\right).
\end{align}

Finally, the second term in $\mathcal Z_8$ satisfies
\begin{align}\label{quanju91}
&\kappa^{-5/12}X(t)
\operatorname{CK}_3^{1/2}(t)
\operatorname{CK}_2^{1/4}(t)
\operatorname{CK}_1^{1/4}(t)
\notag\\
&\quad\leq
c\kappa^{1/12}
\operatorname{CK}_3^{1/2}(t)
\operatorname{CK}_2^{1/4}(t)
\operatorname{CK}_1^{1/4}(t)
\notag\\
&\quad\leq
c\kappa^{1/12}
\left(
\frac12\operatorname{CK}_3(t)
+
\frac14\operatorname{CK}_2(t)
+
\frac14\operatorname{CK}_1(t)
\right)
\notag\\
&\quad\lesssim
c\left(
\operatorname{CK}_1(t)
+
\operatorname{CK}_2(t)
+
\operatorname{CK}_3(t)
\right).
\end{align}
Combining \eqref{quanju90c} and \eqref{quanju91}, we conclude that
\begin{align}\label{quanju91c}
\mathcal Z_8(t)
\lesssim
c\left(
\operatorname{CK}_1(t)
+
\operatorname{CK}_2(t)
+
\operatorname{CK}_3(t)
\right).
\end{align}

Substituting \eqref{quanju87}--\eqref{quanju91c} into the energy
inequality of Proposition~\ref{quanju3}, we obtain
\begin{align}\label{quanju92}
&\frac{d}{dt}E(t)
+
\left(
\frac16-\frac{\mu+\nu}{|\beta|\sqrt{\mu\nu}}
\right)\operatorname{Dis}_1(t)
+
\left(
\frac{1}{6\pi} - \frac{7\pi+1}{6\pi|\beta|}
\right)\operatorname{CK}_1(t)
\notag\\
&\qquad+
\left(
\frac{1}{6D_\delta} - \frac{12D_\delta+1}{12D_\delta|\beta|}
\right)\operatorname{CK}_2(t)
+
\left(
1-\frac{1}{2|\beta|}
\right)\operatorname{CK}_3(t)
\notag\\
&\quad\leq
Cc\left(
\operatorname{Dis}_1(t)
+\operatorname{CK}_1(t)
+\operatorname{CK}_2(t)
+\operatorname{CK}_3(t)
\right),
\qquad 0\leq t\leq T.
\end{align}
By the assumptions on the background magnetic field,
\begin{align*}
\frac{6(\mu+\nu)}{|\beta|\sqrt{\mu\nu}}
\leq 1-6\vartheta,
\end{align*}
and hence
\begin{align*}
\frac16-\frac{\mu+\nu}{|\beta|\sqrt{\mu\nu}}
\geq\vartheta>0.
\end{align*}
Moreover, the condition $|\beta|>7\pi+1$ implies
\begin{align*}
\frac{1}{6\pi} - \frac{7\pi+1}{6\pi|\beta|}>0,
\qquad
\frac{1}{6D_\delta} - \frac{12D_\delta+1}{12D_\delta|\beta|}>0,
\qquad
1-\frac{1}{2|\beta|}>0.
\end{align*}
Define the uniform coercivity constant
\begin{align*}
c_*
:=
\min\left\{
\vartheta,\,
\frac{1}{6\pi} - \frac{7\pi+1}{6\pi|\beta|},\,
\frac{1}{6D_\delta} - \frac{12D_\delta+1}{12D_\delta|\beta|},\,
1-\frac{1}{2|\beta|}
\right\}>0.
\end{align*}
The constant $c_*$ depends only on $\beta$ and $\vartheta$ and is,
in particular, independent of $\mu$, $\nu$, and $\kappa$.  We now
choose the bootstrap constant $c>0$ sufficiently small so that
\begin{align*}
Cc\leq\frac{c_*}{2},
\end{align*}
where $C=C(N, \beta, \vartheta)$ is the
constant in Proposition~\ref{quanju3}.  It follows from
\eqref{quanju92} that
\begin{align*}
\frac{d}{dt}E(t)
+
\frac{c_*}{2}
\left(
\operatorname{Dis}_1(t)
+\operatorname{CK}_1(t)
+\operatorname{CK}_2(t)
+\operatorname{CK}_3(t)
\right)
\leq0,
\qquad 0\leq t\leq T.
\end{align*}
Integrating over $[0,t]$, we conclude that
\begin{align*}
E(t)
+
\frac{c_*}{2}
\int_0^t
\left(
\operatorname{Dis}_1
+\operatorname{CK}_1
+\operatorname{CK}_2
+\operatorname{CK}_3
\right)(\tau)\,d\tau
\leq E(0),
\qquad 0\leq t\leq T.
\end{align*}

By the coercivity estimate in Proposition~\ref{quanju3} and the
uniform equivalence $\mathcal{W}(0,k,\xi)\approx1$, there exists a
constant $C$, independent of $\kappa,\mu,\nu$, such that
\begin{align}\label{quanju95}
 \bigl\|
 \sqrt{\mathcal{W}}\Lambda_t^N
 (\mathbf{u},\mathbf{b})(t)
 \bigr\|_{L^2}^2
 \lesssim E(t)
 \leq E(0)
 \lesssim
 \|(\mathbf{u}_{\rm in},\mathbf{b}_{\rm in})\|_{H^N}^2
 \leq C\varepsilon_0^2\kappa.
\end{align}
Having fixed $c$, choose $\varepsilon_0>0$ sufficiently small that
the right-hand side of \eqref{quanju95} is at most
$\frac14c^2\kappa$. It follows that
\begin{align*}
 \bigl\|
 \sqrt{\mathcal{W}}\Lambda_t^N
 (\mathbf{u},\mathbf{b})(t)
 \bigr\|_{L^2}
 \leq\frac{c}{2}\kappa^{1/2},
 \qquad 0\leq t\leq T.
\end{align*}
This strictly improves \eqref{quanju86}. The local theory and a
standard continuity argument therefore extend the solution and the
weighted estimate to all $t\geq0$.

It remains to recover the estimates stated in the theorem. Since
$\mathcal{W}_k(t,\xi)\gtrsim1$, the global weighted bound gives
\begin{align*}
 \bigl\|
 \Lambda_t^N
 (\mathbf{u},\mathbf{b})(t)
 \bigr\|_{L^2}
 \lesssim\varepsilon_0\kappa^{1/2}.
\end{align*}
This is \eqref{GWP}. On the nonzero horizontal modes,
\begin{align*}
 \mathcal{W}_k(t,\xi)
 \approx \mathcal{A}_k(t,\xi)
 =
 e^{\frac{1}{32}\kappa^{1/3}t},
 \qquad k\neq0.
\end{align*}
Consequently,
\begin{align}\label{quanju98}
 \bigl\|
 \Lambda_t^N
 (\mathbf{u}_{\neq},\mathbf{b}_{\neq})(t)
 \bigr\|_{L^2}
 \lesssim
 \varepsilon_0\kappa^{1/2}
 e^{-\frac{1}{64}\kappa^{1/3}t},
\end{align}
which proves \eqref{enhancedis}.

Finally, the divergence-free conditions imply, for $k\neq0$,
\begin{align*}
 k\widehat{{u}}_{\neq}^{\,1}
 +\xi\widehat{{u}}_{\neq}^{\,2}=0,
 \qquad
 k\widehat{{b}}_{\neq}^{\,1}
 +\xi\widehat{{b}}_{\neq}^{\,2}=0.
\end{align*}
Hence
\begin{align*}
 t|k|\,
 \bigl|
 \widehat{{u}}_{\neq}^{\,2}
 \bigr|
 \leq
 |k|\,
 \bigl|
 \widehat{{u}}_{\neq}^{\,1}
 \bigr|
 +
 |\xi+kt|\,
 \bigl|
 \widehat{{u}}_{\neq}^{\,2}
 \bigr|,
\end{align*}
and the same estimate holds for
$\widehat{{b}}_{\neq}^{\,2}$. Since $|k|\geq1$, it follows
from \eqref{quanju98} that
\begin{align*}
 &\bigl\|
 ({u}_{\neq}^{\,1},
  {b}_{\neq}^{\,1})(t)
 \bigr\|_{L^2}
 +
 \langle t\rangle
 \bigl\|
 ({u}_{\neq}^{\,2},
  {b}_{\neq}^{\,2})(t)
 \bigr\|_{L^2}\lesssim
 \varepsilon_0\kappa^{1/2}
 e^{-\frac{1}{64}\kappa^{1/3}t}.
\end{align*}
This proves \eqref{invisciddamping} and completes the proof of
Theorem \ref{thm1.2}.\hfill $\square$

\bigskip

\vskip .2in
\section*{Acknowledgement}
\noindent{J. Wu was partially supported by the National Science Foundation of the United States under Grants DMS 2104682 and DMS 2309748. X. Zhai was partially supported by  the Guangdong Provincial Natural Science Foundation under grant 2024A1515030115. }

 \vskip .2in
\noindent{\bf Data Availability Statement} Data sharing is not applicable to this article as no
data sets were generated or analysed during the current study.

\vskip .2in

\noindent{\bf Conflict of Interest} The authors declare that they have no conflict of interest.

\end{document}